\documentclass[onecolumn]{autart}
\usepackage{} 
\usepackage{amsmath,bm,amssymb,amsfonts,mathrsfs,bbm}
\usepackage{cite}
\usepackage{url}
\usepackage{enumerate}
\usepackage{color}
\usepackage{wrapfig}
\usepackage{dcolumn}
\usepackage{lineno}
\usepackage{etoolbox}
\allowdisplaybreaks[4]

\numberwithin{equation}{section}

\newtheorem{theorem}{Theorem}
\newtheorem{lemma}{Lemma}[section]

\newtheorem{proposition}{Proposition}

\newtheorem{definition}{Definition}

\newtheorem{remark}{Remark}[section]
\newtheorem{assumption}{Assumption}

\newenvironment{theoremproof}{%
	{\bfseries\MakeUppercase{Proof of Theorem~\ref{thm:main result}:\enspace}}%
}{\hfill $\square$\par}

\makeatletter
\let\original@makecaption\@makecaption
\long\def\small@makecaption#1#2{%
	\vskip\abovecaptionskip
	\sbox\@tempboxa{\fontsize{8}{11.5}\selectfont {\fontsize{8.5}{11.5}\selectfont #1:} #2}%
	\ifdim \wd\@tempboxa >\hsize
	\fontsize{8}{11.5}\selectfont {\fontsize{8.5}{11.5}\selectfont #1}\quad #2\par
	\else
	\global \@minipagefalse
	\hb@xt@\hsize{\hfil\box\@tempboxa\hfil}%
	\fi
	\vskip\belowcaptionskip}

\newcommand{\smallcaption}[1]{%
	\let\@makecaption\small@makecaption
	\caption{#1}%
	\let\@makecaption\original@makecaption
}
\makeatother

\begin{document}

%
\begin{frontmatter}

\title{Polynomial iISS for a class of Timoshenko equations with external disturbances and infinite history memory} 

\author{Xinyu Guo},
\author{Jun Zheng}\ead{zhengjun2014@aliyun.com}
\address{{School of Mathematics}, Southwest Jiaotong University, Chengdu, China}

\begin{abstract}                           
 This paper investigates the robustness of Timoshenko equations subject to external disturbances and infinite history memory within the integral input-to-state stability (iISS) framework. While polynomial iISS (PiISS) is a powerful tool for characterizing the influence of these two factors, the existing definitions of PiISS rely on graph norms of the state, which are inconsistent with the classical iISS notion for infinite-dimensional systems. In this paper, we provide a rigorous PiISS definition for a class of Timoshenko equations solely via the state norm, and derive sufficient conditions for PiISS under the Tolksdorf condition on the relaxation function, as well as exponential iISS (EiISS) under a stronger condition. The well-posedness is established by using the semigroup and elliptic equation theories. The PiISS and EiISS are assessed by constructing appropriate Lyapunov functionals and employing various a priori estimates of solutions, thereby fully characterizing the impact of disturbances and history memory on system stability.
\end{abstract}

\end{frontmatter}
\endNoHyper

\section{Introduction}\label{sec:1}
\subsection{Brief review of stability theory for Timoshenko equations}
When studying the transverse vibration of beams, Timoshenko observed that a micro-element not only undergoes bending and rotation but also deforms due to shear forces. Based on this observation, Timoshenko first derived the following differential equation for transverse vibration with the shear effect
\begin{align*}
	\rho_1\phi_{tt}(x, t)- k\left(\phi_x(x, t) + \psi(x, t)\right)_x =0,\quad(x, t)\in (0, L)\times {{\mathbb{R}}_{>0}},\\
	\rho_2\psi_{tt}(x, t) - b\psi_{xx}(x, t) + k(\phi_{x}(x, t) + \psi(x, t)) = 0,\quad(x, t)\in (0, L)\times {{\mathbb{R}}_{>0}},
\end{align*}
where $\phi$ denotes the transverse displacement, $\psi$ denotes the rotation angle of the beam fibers, $\phi_x+\psi$ denotes the shear angle, and $\rho_1, \rho_2, k, b,$ and $L$ are positive constants. This equation provides a more realistic theoretical foundation for the dynamic analysis of engineering structures such as thick beams and elastic rods in the subsequent work \cite{Timoshenko1921}. Since 2000, considerable efforts have been devoted to conducting stability analysis of Timoshenko equations. By introducing various damping mechanisms, numerous scholars have examined the stability properties of such systems and established corresponding energy decay estimates; see \cite{Dridi2021,Liu2025,Rivera:2008,Raposo2005,Soufyane2003}. However, these studies have not taken into account the history memory effects arising from the long-term relaxation behavior of viscoelastic materials, and therefore cannot fully capture the lasting memory properties inherent in real viscoelastic materials.
\par In contrast to the aforementioned studies that neglect memory effects, the work of \cite{AmmarKhodja2003} incorporates a memory term into the Timoshenko system. Specifically, the viscoelastic memory effect is characterized by the convolution integral term in the rotation equation: $\int_{0}^{t}g(t-s)\psi_{xx}(s)\mathrm{d}s$, where the relaxation function $g$ quantifies the influence of historical strain on the current stress state. By using the energy method, multiplier techniques, and a suitably constructed Lyapunov functional, the authors of \cite{AmmarKhodja2003} demonstrated that the equality of wave speeds constitutes a sufficient condition for the exponential stability of the system. Furthermore, through spectral analysis and counterexamples, they verified the necessity of this condition. Notably, they proved that the decay rate of solutions is directly inherited from that of $g$, namely, exponential decay of $g$ yields exponential decay of solutions, while polynomial decay of $g$ leads to polynomial decay of solutions at the same rate. 
Building upon this framework, the authors of \cite{Guesmia2008} further relaxed the structural conditions imposed on the relaxation function $g$. By refining the multiplier method and energy estimates, they established the corresponding exponential and polynomial decay results under weaker conditions.
\par In \cite{Rivera2008}, the authors investigated the Timoshenko system with infinite memory described by $\int_{0}^{\infty}g(s)\psi_{xx}(x, t-s)\mathrm{d}s$, for which the relaxation function $g$, together with its first and second derivatives, was assumed to satisfy certain structural conditions. By introducing a memory variable, the authors reformulated the original equations, which simplifies the analysis. Then, they established the well-posedness and stability of the system by using the semigroup theory and spectral method. The results indicate that the energy of the system decays exponentially when the wave speed ratios are equal, and polynomially otherwise. The authors of \cite{Messaoudi2009} further relaxed the structural requirements on $g$. They proved that the decay rate of solutions is consistent with that of the relaxation function $g$ when the wave speed ratios are equal, and solutions still decay polynomially otherwise. In \cite{Guesmia2012}, the authors imposed even weaker structural conditions on $g$, thereby deriving more general energy decay estimates. Moreover, the stability results were extended to four classes of Timoshenko-heat coupled systems, further broadening the applicability of the memory-based approach.
\par In recent years, numerous scholars have investigated various Timoshenko systems involving memory-type dissipation, yielding a wealth of results concerning well-posedness, stability, and decay rates of solutions. Representative findings for classical models with memory damping can be found in \cite{Al-Mahdi2024,AlMahdi2021,AlOmari2021}, while studies on coupled or generalized systems with memory are presented in \cite{Boulaaras2026,SilvaRacke2021,LiuMao2024,Messaoudi2025,Zeng2025}. It is noteworthy that the majority of existing literature mainly focuses on scenarios in the absence of external disturbances. However, in practical engineering contexts, systems governed by partial differential equations (PDEs) are frequently subject to a variety of disturbances, including model uncertainties, controller actuation errors, and external inputs. To date, no relevant studies have addressed how such disturbances affect the stability of Timoshenko systems within the existing literature. Therefore, rigorously quantifying the influence of external disturbances on the stability of Timoshenko systems is of significant theoretical and practical importance.
\subsection{Input-to-state stability for PDEs}
In 1989, Sontag first introduced the concept of input-to-state stability (ISS) for finite-dimensional nonlinear systems to describe the influence of external inputs on system states \cite{Sontag1989}. ISS unifies the notions of external and internal stability: it requires that the state remains bounded for any bounded external input, and that the state asymptotically converges to the origin when the external input tends to zero. However, the boundedness assumption on inputs may be restrictive for certain complex systems. To address this limitation, Sontag proposed a weaker robustness notion, namely, the integral input-to-state stability (iISS), in 1998 \cite{Sontag1998}. The iISS accommodates unbounded inputs with bounded energy while still characterizing their effect on system stability; see \cite{Sontag1998}. Collectively, the study of ISS, iISS, and their various variants form the comprehensive framework of ISS theory.
\par Subsequently, researchers including Dashkovskiy, Logemann, Mironchenko, and Prieur, among others, extended ISS theory from finite-dimensional systems to infinite-dimensional systems governed by PDEs; see \cite{Dashkovskiy2013,Logemann2013,MP2011,PM2012}. In recent years, numerous scholars have further investigated and applied ISS theory to PDEs in various scenarios and established different ISS and iISS estimates for the systems; see, e.g., \cite{Chen2025,Jacob2019,2018Karafyllis,Karafyllis2021,Karafyllis2017,Zheng2018,Zheng2019,Zheng2020,Zheng2020b,Zheng:2021,Zheng2021,Zheng2022b,Zheng2026} for parabolic equations, and \cite{PM2012,Karafyllis2022,Karafyllis2023,Lamare2018,Zheng2026} for hyperbolic equations, respectively. For comprehensive surveys, see \cite{Karafyllis2018,Mironchenko2023a,Mironchenko2020,Zheng2020b,Zheng2020,Zheng2026}.
\par Furthermore, most of the existing work focuses on exponential ISS (EISS) or exponential iISS (EiISS) for specific PDEs with disturbances. However, many PDE systems are stable but not exponentially stable, exhibiting asymptotic behaviors that are weaker than exponential decay. Consequently, establishing the EISS and EiISS is not suitable in this setting; see \cite{Wakaiki:2022}. To address this limitation, the authors of \cite{Wakaiki:2022} introduced several stability notions for infinite-dimensional systems, including  polynomial iISS (PiISS) and polynomial ISS. For abstract bilinear systems, it was shown that the system possesses PiISS under certain conditions on the admissibility of the input operator. Nevertheless, the PiISS defined therein relies on the graph norm of initial states, and hence, deviates from the standard definition of classical iISS for infinite-dimensional systems. Moreover, the proposed admissibility condition is relatively restrictive and remains difficult to verify for concrete PDE systems. Importantly, the construction of Lyapunov functionals tailored to PiISS has not been addressed for concrete PDEs. Consequently, research on PiISS for PDE systems remains limited, and apart from the work in \cite{Wakaiki:2022}, no other relevant studies have been reported in the existing literature. In particular, a rigorous PiISS estimate has not been established for PDEs. 
\subsection{Challenges and contributions of the present paper}
Motivated by the afore-mentioned research gaps, the present paper introduces in-domain disturbances into a class of Timoshenko equations with infinite history memory, and analyzes the stability within the iISS framework, thereby clarifying the influnce of disturbances and history memory on system behavior. The present work meets three main challenges. 
\begin{itemize}
	\item The first concerns robustness characterization. The introduction of external disturbances destroys the asymptotic stability inherent to the undisturbed system, thereby rendering it nontrivial to characterize system robustness within the ISS framework. 
	\item The second challenge lies in stability analysis. The presence of disturbances renders the Lyapunov functional commonly adopted in existing studies for asymptotic stability analysis no longer dissipative. Thus, constructing a suitable Lyapunov candidate for robustness analysis becomes highly challenging and requires considerably more refined estimates of the solutions. 
	\item The third challenge exists in  well-posedness analysis. Compared with the existing literature, the simultaneous presence of both disturbances and infinite history memory significantly complicates the well-posedness analysis when linear operator semigroup theory is employed, and therefore other tools are required.
\end{itemize}
\par Overall, the main contributions of this paper are fourfold.
\begin{itemize}
	\item We characterize the robustness of the Timoshenko system subject to in-domain disturbances and infinite history memory within the iISS framework.
	\item We provide a rigorous definition of polynomial iISS (PiISS) by using only the state norm rather than the graph norm, which is consistent with the classical iISS definition, to quantify the influnce of external disturbances and infinite history memory.
	\item We construct an appropriate Lyapunov functional and establish a PiISS estimate for the considered Timoshenko system.  Moreover, under a stronger condition on the relaxation function, we establish an EiISS estimate for the system.
	\item The present work adopts weaker structural conditions and relaxes the requirements for the relaxation function, leading to stability results with broader applicability.
\end{itemize}
\subsection{Organization of the paper}
The rest of this paper is organized as follows. We begin by introducing some notations. In Section \ref{sec:3}, we formulate the problem and state the main results. In Section \ref{sec:4}, we address the well-posedness of the considered equation. In Section \ref{sec:5}, we construct several auxiliary functionals and establish their \textit{a priori} estimates, which serve as necessary preparations for the Lyapunov stability analysis in subsequent sections. In Section \ref{sec:6}, we construct a Lyapunov functional and establish the PiISS and EiISS estimates for the considered system. Finally, we provide concluding remarks in Section \ref{sec:8}.
\par \textbf{Notation.} \quad Let $\mathbb{N}_+$, $\mathbb{R}$, $\mathbb{R}_{\leq0}$, $\mathbb{R}_{>0}$, and $\mathbb{R}_{\geq 0}$ denote the sets of positive integers, real numbers, nonpositive real numbers, positive real numbers, and nonnegative real numbers, respectively. Let $\mathbb{R}^n(n \geq 1)$ denote the $n$-dimensional real Euclidean space.
\par We first define function spaces for single-variable functions. For any open or closed domain $\Omega \subset \mathbb{R}$ and $p\in[1, +\infty)$, let $L^p(\Omega)$ denote the space of all measurable functions~$u$ satisfying $\int_{\Omega}|u(x)|^p\mathrm{d}x<+\infty$. The norm  $\|\cdot\|$  on $L^p(\Omega)$ is defined by $\| u \|_{L^p(\Omega)}=\left(\int_{\Omega}{|u(x)|^p}\mathrm{d}x\right)^{\frac{1}{p}}$. Let~$H_{0}^{1}(\Omega)=\{u:\Omega\to \mathbb{R} \mid u, |\nabla u|\in L^{2}(\Omega), u=0 \text{ on } \partial\Omega \}$ equipped with the norm $\| u \|_{H_{0}^{1}(\Omega)}=\| \nabla u \|_{L^{2}(\Omega)}$. For a positive integer $m$, set $H^{m}(\Omega)=\{u:\Omega\to \mathbb{R} \mid u\in L^{2}(\Omega)~\text{and the}~i\text{-th weak derivative}~D^{i}u\in L^{2}(\Omega), i=1, 2,..., m\}$. For any $T\in \mathbb{R}_{>0}$, define $C([0,T];\mathbb{R}_{\ge 0})=\left\{u : [0,T] \to \mathbb{R}_{\ge 0} \mid u \text{ is continuous on } [0,T]\right\}.$ For a nonnegative integer~$m$, define $C^m(\mathbb{R}_{\geq 0}; \mathbb{R}_{>0}) = \{u:\mathbb{R}_{\geq 0} \to \mathbb{R}_{>0} \mid u \text{ has } m \text{-th continuous derivatives on}~\mathbb{R}_{\geq 0} \}$.
\par We next define function spaces for bivariate functions. Let $L^2(\mathbb{R}_{\geq0}; L^2(0, 1)) = \{u: \mathbb{R}_{\geq0} \rightarrow L^2(0,1) \mid \int_0^{\infty} \int_0^1 | u(x, t) |^{2} \mathrm{d}x\mathrm{d}t<+\infty\}$ equipped with the norm
\begin{center}
	$\| u \|_{L^2(\mathbb{R}_{\geq0}; L^{2}(0, 1))} =\left( \int_0^{\infty} \int_0^1 | u(x, t) |^{2} \mathrm{d}x\mathrm{d}t\right)^{\frac{1}{2}}$.
\end{center}
For a function $g\in C^{0}(\mathbb{R}_{\geq 0};\mathbb{R}_{>0})$, let $L_{g}^{2}(\mathbb{R}_{\geq0}; H_{0}^{1}(0, 1))$ be the space of measurable functions $u:\mathbb{R}_{\geq0}\to H_{0}^{1}(0, 1)$ satisfying $\int_{0}^{1}\int_{0}^{\infty}g(s)|u_{x}(x, s)|^{2}\mathrm{d}s\mathrm{d}x < +\infty$, equipped with the norm and inner product
\begin{center}
	$\| u \|_{L_{g}^{2}(\mathbb{R}_{\geq0}; H_{0}^{1}(0, 1))}=\left(\int_{0}^{1}\int_{0}^{\infty}g(s)|u_{x}(x, s)|^{2}\mathrm{d}s\mathrm{d}x\right)^{\frac{1}{2}}$\\
\end{center}
and
\begin{center}
	$\langle u_1, u_2\rangle_{L_{g}^{2}(\mathbb{R}_{\geq0}; H_{0}^{1}(0, 1))}=\int_{0}^{1}\int_{0}^{\infty}g(s)u_{1x}(x, s)u_{2x}(x, s)\mathrm{d}s\mathrm{d}x$,
\end{center}
respectively. For a nonnegative integer $m$ and a vector space $H$, let $C^{m}(\mathbb{R}_{\geq0}; H)=\{u:\mathbb{R}_{\geq0}\to H\mid\frac{\partial^{i}u}{\partial t^{i}}(\cdot, t)\in H~\text{and}~\frac{\partial^{i}u}{\partial t^{i}}(x, t)~\text{is continuous in }~t\in \mathbb{R}_{\geq0}~\text{for all}~i=0,..., m\}.$
\par For comparison functions, let
\begin{align*}
	\mathcal{K}\! &= \!\{\mu\!:\!\mathbb{R}_{\ge0}\to\mathbb{R}_{\ge0} \!\mid\!\mu\text{ is strictly increasing and}~\mu(0)=0\};\\
	\mathcal{K}_{\infty}\! &=\! \left\{\theta\in\mathcal{K} \!\mid\!\lim_{r\to+\infty}\!\theta(r)=+\infty\right\};\\
	\mathcal{L}\! &= \!\left\{\mu\!:\!\mathbb{R}_{\ge0}\to\mathbb{R}_{\ge0} \!\mid\!\mu\text{ is strictly decreasing and }\lim_{t\to+\infty}\!\mu(t)=0\right\};\\
	\mathcal{KL}\! &=\! \{\beta\!:\!\mathbb{R}_{\ge0}\!\times\!\mathbb{R}_{\ge0}\to\mathbb{R}_{\ge0} \!\mid\!
	\beta(\cdot,t)\in\mathcal{K}, \forall t\in\mathbb{R}_{\ge0},~\text{and}~
	\beta(r,\cdot)\in\mathcal{L}, \forall r\in\mathbb{R}_{>0}\}.
\end{align*}
\section{Problem formulation and main results}\label{sec:3}In this section, we first present the equation for the considered Timoshenko system incorporating infinite history memory and distributed in-domain disturbances. Then, we state the main results of this paper, namely, the PiISS, as well as EiISS, for the considered system.
\subsection{Problem formulation}We begin by extending the classical Timoshenko equations to include in-domain disturbances and infinite history memory, which gives rise to the original system. To facilitate the subsequent well-posedness and stability analysis, we then perform an equivalent reformulation by introducing history memory variables. This leads to the system that serves as the focus of this paper.
\par The Timoshenko equation under consideration is given as follows:
\begin{subequations}\label{eq:ssys}
	\begin{align}
		&\rho_1\phi_{tt}(x, t) \!-\! k\left(\phi_x(x, t) \!+\! \psi(x, t)\right)_x \!= \! f(x, t),\quad(x, t)\in (0,1)\times {{\mathbb{R}}_{>0}},\\
		&\rho_2\psi_{tt}(x, t) \!-\! b\psi_{xx}(x, t) \!+ \!\int_{0}^{\infty}\!\! g(s)\psi_{xx}(x, t \!-\! s)\mathrm{d}s \!+\! k(\phi_{x}(x, t) \!+\! \psi(x, t))\!=\! h(x, t),\quad(x, t)\in (0,1)\times {{\mathbb{R}}_{>0}},\\
		&\phi(0, t)\!=\!\phi(1, t)\!=\!\psi(0, t)\!=\!\psi(1, t)\!=\!0, \quad t\in {{\mathbb{R}}_{\geq 0}},\\
		&\phi(x, 0)\!=\!\phi_0(x),\quad \phi_t(x, 0)\!=\!\phi_1(x),\quad x \in (0, 1),\\
		&\psi(x, -t)\!=\!\psi_0(x, t),\quad \psi_t(x, 0)\!=\!\psi_1(x),\quad x \in (0, 1), t\in {{\mathbb{R}}_{\geq 0}},
	\end{align}
\end{subequations}
where $\rho_1$, $\rho_2$, $k$, and $b$ are positive constants, $\phi(x, t)$ denotes the transverse displacement, $\psi(x, t)$ denotes the rotation angle of the beam fibers, $g(s)$ is the relaxation function, $f(x, t)$ and $h(x, t)$ are in-domain disturbances, and $\phi_{0}(x), \phi_{1}(x), \psi_{0}(x, t),$ and $\psi_{1}(x)$ are given initial data. To facilitate the subsequent analysis of well-posedness and stability of solutions to this system, we introduce the history variable $\eta^t(x,s)$, which is defined by
\begin{align}\label{eq:his}
	\eta^{t}(x, s)&=\psi(x, t)-\psi(x, t-s),\quad \forall x \in (0, 1), s, t\in {{\mathbb{R}}_{\geq 0}}.
\end{align}
\par Consequently, system \eqref{eq:ssys} can be writen as
\begin{subequations}\label{eq:sys}
	\begin{align}
		& \rho_1\phi_{tt}(x, t) \!-\! k\left(\phi_x(x, t) \!+\! \psi(x, t)\right)_x \!=\! f(x, t),\quad(x, t)\in (0,1)\times {{\mathbb{R}}_{>0}}, \label{eq:sys01}\\
		& \rho_2\psi_{tt}(x, t) \!-\! \hat{b}\psi_{xx}(x, t) \!-\! \int_{0}^{\infty}\!\! g(s)\eta_{xx}^{t}(x, s)\mathrm{d}s \!+\! k(\phi_{x}(x, t)\! +\! \psi(x, t)) \!=\! h(x, t),\quad(x, t)\in (0,1)\times {{\mathbb{R}}_{>0}}, \label{eq:sys02}\\
		& \eta_t^{t}(x, s) \!+\! \eta_s^{t}(x, s) \!-\! \psi_{t}(x, t) \!=\! 0,\quad(x, t)\in (0, 1)\times {{\mathbb{R}}_{>0}}, s\in {{\mathbb{R}}_{>0}},\label{eq:sys03}\\
		& \phi(0, t)\!=\!\phi(1, t)\!=\!\psi(0, t)\!=\!\psi(1, t)\!=\!\eta^{t}(0, s)\!=\!\eta^{t}(1, s)\!=\!0,\quad s, t\in {{\mathbb{R}}_{\geq 0}},\label{eq:ib01}\\
		& \phi(x, 0)\!=\!\phi_0(x),\quad \phi_t(x, 0)\!=\!\phi_1(x),\quad x \in (0, 1),\label{eq:ib02}\\
		& \psi(x, -t)\!=\!\psi_0(x, t),\quad \psi_t(x, 0)\!=\!\psi_1(x),\quad x \in (0, 1), t\in {{\mathbb{R}}_{\geq 0}}, \label{eq:ib03}\\
		& \eta^0(x, s)\!=\!\eta_0(x, s),\quad x \in (0, 1), s\in {{\mathbb{R}}_{\geq 0}},
	\end{align}
\end{subequations}
where $\hat{b} = b -\int_{0}^{\infty}g(s) \mathrm{d}s$ is a constant, and $\eta_{0}(x, s)=\psi_{0}(x, 0)-\psi_{0}(x, s)$ represents the given initial data. Therefore, in the sequel, we consider system \eqref{eq:sys}, which integrates the history variable with in-domain disturbances, as the object of our investigation.
\par Throughout this paper, we always propose the following assumptions on the system coefficients, relaxation function, initial data, and external disturbances.
\begin{assumption}\label{ass:Assumption3}
	The equal wave speed ratio condition holds, namely~$\frac{{\rho }_{1}}{k}=\frac{{\rho }_{2}}{b}$.
\end{assumption}
\begin{assumption}\label{ass:Assumption2}
	The relaxation function $g\in C^{0}({{\mathbb{R}}_{\geq 0}}; {{\mathbb{R}}_{>0}}) \cap C^{1}({{\mathbb{R}}_{\geq 0}}; {{\mathbb{R}_{\leq0}}})$ satisfies
	\begin{align*}
		g_{0}=\int_{0}^{\infty}g(s)\mathrm{d}s<b,
	\end{align*}
	and there exists a function $H\in C^{0}({{\mathbb{R}}_{\geq0}}; {{\mathbb{R}}_{>0}})\cap C^{1}({{\mathbb{R}}_{>0}}; {{\mathbb{R}}_{>0}})$ satisfying $H(0)=0$ and the Tolksdorf's condition:
	\begin{equation}
		1\leq \delta_{0}\leq \frac{H'(\tau)\tau}{H(\tau)}\leq \delta_{1}<\frac{3}{2},\quad\forall \tau\in {{\mathbb{R}}_{> 0}},\label{eq:2.4}
	\end{equation}
	such that
	\begin{equation*}
		g'(t)\leq-k_{0}H(g(t)),\quad\forall t\in {{\mathbb{R}}_{> 0}},
	\end{equation*}
	where $\delta_0, \delta_1,$~and $k_0$ are positive constants.
\end{assumption}
\begin{assumption}\label{ass:Assumption1}
	The initial data $\phi_0$ and $\psi_0(\cdot, 0)$ belong to $H_{0}^{1}(0, 1)$ with $\underset{s\geq 0}{\mathop{\sup}}\|\psi_{0x}(\cdot, s)\|_{L^2(0, 1)}<+\infty$, $\phi_1$ and $\psi_1$ belong to $L^{2}(0,1)$, and $\eta_0$ belongs to $L_{g}^{2}(\mathbb{R}_{\geq0}; H_{0}^{1}(0, 1));$ the in-domain disturbances $f$ and $h$ belong to $L^{2}\left(\mathbb{R}_{\geq0}; L^{2}(0, 1)\right)$.
\end{assumption}

\begin{remark}\label{rem:Remark1}
	The structural condition~\eqref{eq:2.4}, which are much weaker than the one proposed in \cite{Messaoudi2009, Rivera2008}, was first introduced by Tolksdorf for studying the regularity of solutions to a class of elliptic equations; see \cite{Tolksdorf:1983}. Functions satisfying this structural condition include not only the classical exponential type, but also the logarithmic, the variable exponent, and other types. Below we present several representative examples (see \cite{Zheng:2021,Zheng:2022} for more)$:$ 
	\begin{enumerate}[(i)]
		\item $H(\tau )={{\tau }^{p}}, \forall p\ge 1;$
		\item $H(\tau )={{\tau }^{a}}{{\log }_{c}}(b\tau +d)$, which satisfies \eqref{eq:2.4} with ${\delta_{0}}=a$ and ${\delta_{1}}=a+\frac{1}{\ln d}$, where $b>0$ and $c, d>1$ are constants;
		\item $H(\tau )= \left\{ \begin{array}{ll} a\tau^p, & 0 \leq \tau < \tau_0 \\ b\tau^{q(\tau)-1}, & \tau \geq \tau_0 \end{array} \right.$, where $\tau_0 > 1$ and $a, b, p > 0$ are constants, and $q \in C^1([\tau_0, +\infty); \mathbb{R})$ satisfies 
		$\begin{cases}
			\delta_{0} \leq q'(\tau)\tau \ln \tau + q(\tau) - 1 \leq \delta_{1}, & \forall \tau \geq \tau_0,\\
			p = q'(\tau_0)\tau_0 \ln \tau_0 + q(\tau_0) - 1, \\
			a = b\tau_0^{q(\tau_0)-1-p}.
		\end{cases}$
	\end{enumerate}
	\par Note that condition~\eqref{eq:2.4} implies the following properties of the function~$H$ (see \cite{Zheng:2021,Lieberman:1991})$:$
	\begin{equation}
		\min\left\{s^{\delta_0}, s^{\delta_{1}}\right\}H(\tau)\leq H(s\tau)\leq \max\left\{s^{\delta_0}, s^{\delta_{1}}\right\}H(\tau),\quad\forall s, \tau\in {\mathbb{R}}_{\geq0}.\label{eq:3.5}
	\end{equation}
\end{remark}
\subsection{Main results}In this subsection, we first specify the concept of PiISS in a rigorous manner, then state the main results of this paper. Let
\begin{align*}
	\mathcal{H}=\left\{(u, v, w, z, l)\in H_0^1(0, 1)\times L^2(0, 1)\times H_0^1(0, 1)\times L^2(0, 1)\times L_g^2(\mathbb{R}_{\geq0}; H_0^1(0, 1))\right\},
\end{align*}
which is a Hilbert space endowed with the inner product $\langle \cdot, \cdot \rangle_\mathrm{H}$ and the norm $\|\cdot\|_\mathrm{H}:$
\begin{align*}
	\langle X_{1}, X_{2}\rangle_\mathrm{H}=\!\int_{0}^{1}\!\left(\rho_{1}v_1v_2+\hat{b}w_{1x}w_{2x}+\rho_{2}z_1z_2+k(u_{1x}+w_{1})(u_{2x}+w_{2})+\!\int_{0}^{\infty}\! g(s)l_{1x}l_{2x}\mathrm{d}s\right)\!\mathrm{d}x
\end{align*}
and
\begin{align*}
	\|X\|_\mathrm{H}&=\left(k\|u_{x}+w\|_{L^2(0, 1)}^{2}+\rho_{1}\|v\|_{L^2(0, 1)}^{2}+\hat{b}\|w_{x}\|_{L^2(0, 1)}^{2}+\rho_{2}\|z\|_{L^2(0, 1)}^{2}+\|l\|_{L_{g}^{2}(\mathbb{R}_{\geq0}; H_{0}^{1}(0, 1))}^{2}\right)^{\frac{1}{2}},
\end{align*}
respectively, for all $X_{1}=(u_1, v_1, w_1, z_1, l_1)\in \mathcal{H}$, $X_{2}=(u_2, v_2, w_2, z_2, l_2)\in\mathcal{H}$, and $X=(u, v, w, z, l)\in \mathcal{H}$. To simplify calculations and clearly characterize each state variable, we introduce a norm $\|\cdot\|_{\mathcal{H}}$ defined as follows:
\begin{align*}
	\|X\|_{\mathcal{H}}&=\left(\|u_x\|_{L^2(0, 1)}^{2}+\|v\|_{L^2(0, 1)}^{2}+\|w_x\|_{L^2(0, 1)}^{2}+\|z\|_{L^2(0, 1)}^{2}+\|l\|_{L_{g}^{2}(\mathbb{R}_{\geq0}; H_{0}^{1}(0, 1))}^{2}\right)^{\frac{1}{2}}.
\end{align*}
The proof of the equivalence of $\|\cdot\|_\mathrm{H}$ and $\|\cdot\|_{\mathcal{H}}$ is provided in Section \ref{sec:6}. Since these two norms are equivalent, all subsequent stability analysis will be carried out with respect to $\|\cdot\|_{\mathcal{H}}$. Naturally, all conclusions also hold for $\|\cdot\|_\mathrm{H}$.
\par We now present the definition of PiISS for system \eqref{eq:sys}.
\begin{definition}[\cite{Wakaiki:2022}]\label{def:Def 1}
	System \eqref{eq:sys} is said to be polynomially integral input-to-state state (PiISS) with respect to in-domain disturbances $f$ and $h$ if there exist $\theta_{1},\theta_{2}\in \mathcal{K}_{\infty}$, $\mu_1,\mu_2\in \mathcal{K}$, and $\beta\in \mathcal{K}\mathcal{L}$ satisfying
	\begin{align*}
		\beta(r,t)=O(t^{-m}) \quad \text{as}~t\to +\infty, \quad \forall r\in \mathbb{R}_{>0},
	\end{align*}
	where $m\in \mathbb{R}_{>0}$ is a constant, such that for all $t\in \mathbb{R}_{\geq0}$, the state $X(t)=(\phi, \phi_t, \psi, \psi_t, \eta^t)\in \mathcal{H}$ admits the following estimate$:$
	\begin{align}
		\|X(t)\|_{\mathcal{H}} &\leq \beta\left(\|X_0\|_{\mathcal{H}}+\underset{s\geq 0}{\mathop{\sup}} \left\|\psi_{0x}(\cdot, s) \right\|_{L^2(0, 1)}, t\right) + \theta_1\left(\int_0^t \mu_1\left(\|f(\cdot, \tau)\|_{L^{2}(0, 1)}\right)\mathrm{d}\tau\right)\nonumber\\
		&\quad  + \theta_2\left(\int_0^t \mu_2\left(\|h(\cdot, \tau)\|_{L^{2}(0, 1)}\right)\mathrm{d}\tau\right).\label{eq:3.06}
	\end{align}
\end{definition}
\begin{remark}
	Within the semi-uniform framework, Wakaiki introduced the notion of PiISS in \cite{Wakaiki:2022} for a class of bilinear systems. Note that since a graph norm $\|\cdot\|_{\mathcal{A}}$ of the initial state, namely, $\|X_{0}\|_{\mathcal{A}}=\|X_{0}\|_{\mathcal{H}}+\|\mathcal{A}X_{0}\|_{\mathcal{H}}$, $X_{0}\in D(\mathcal{A})$ is adopted in the right hand-side of \eqref{eq:3.06}, the notion used in \cite{Wakaiki:2022} is inconsistent with the classical iISS definition for infinite-dimensional systems (see \cite{Karafyllis2018}), where $\mathcal{A}$ represents the $C_0$ semigroup generator for the considered system under an abstract form and $D(\mathcal{A})$ is the domain of $\mathcal{A}$. In contrast, the same norm for the initial state and the system state is used in \eqref{eq:3.06}, thereby providing a rigorous definition of PiISS for PDEs.
\end{remark}
\begin{remark}\label{rem:remark2}
	System \eqref{eq:sys} is said to be exponentially input-to-state stable (EiISS) if $\beta\in \mathcal{KL}$ in \eqref{eq:3.06} is replaced by
	\begin{equation*}
		\beta(r, t) = \theta_{3}(r)\exp(-mt),\quad \forall r\in \mathbb{R}_{> 0}, t \in {\mathbb{R}}_{\geq0},
	\end{equation*}
	where $\theta_3\in \mathcal{K}_{\infty}$ and $m\in \mathbb{R}_{>0}$ is a constant; see \cite{Zheng2018}.
\end{remark}
\par The following theorem is the main result of this paper, which establishes a PiISS estimate for system \eqref{eq:sys}. The detailed proof is presented in Section \ref{subsec:7}.
\begin{theorem}\label{thm:main result} If $\delta_{1}\in \left(1, \frac{3}{2}\right)$, system \eqref{eq:sys} is PiISS, having the following estimate$:$
	\begin{align*}
		\|X(t)\|_{\mathcal{H}}&\leq C\left(\|X_0\|_{\mathcal{H}}+\underset{s\geq0}{\mathop{\sup}} \left\|\psi_{0x}(\cdot, s) \right\|_{L^2(0, 1)}\right)(1+t)^{-\frac{1}{4(\delta_{1}-1)}}+C\left(\!\int_{0}^{t}\!\!{\left\| f(\cdot, \tau ) \right\|_{{{L}^{2}(0, 1)}}^{2}\mathrm{d}\tau }\right)^{\frac{1}{2}}\nonumber\\
		&\quad+C\left(\!\int_{0}^{t}\!\!{\left\| h(\cdot, \tau ) \right\|_{{{L}^{2}(0, 1)}}^{2}\mathrm{d}\tau }\right)^{\frac{1}{2}},\quad \forall t\in \mathbb{R}_{\geq0},\label{eq:2.5}
	\end{align*}
	where $C$ is a positive constant, depending only on $\rho_{1}, \rho_{2}, k, b, k_{0}, \delta_{0}, \delta_{1}$, and the functions $g$ and $H$. In particular, when $f\equiv h\equiv0$, system \eqref{eq:sys} is polynomially stable, having the following estimate$:$
	\begin{equation*}
		\|X(t)\|_{\mathcal{H}}\leq C\left(\|X_0\|_{\mathcal{H}}+\underset{s\geq 0}{\mathop{\sup}} \left\|\psi_{0x}(\cdot, s) \right\|_{L^2(0, 1)}\right)(1+t)^{-\frac{1}{2(\delta_{1}-1)}},\quad \forall t\in \mathbb{R}_{\geq0}.
	\end{equation*}
	Furthermore, if $\delta_{1}=1$, system \eqref{eq:sys} is EiISS, having the following estimate$:$
	\begin{align*}
		\|X(t)\|_{\mathcal{H}}&\leq C\|X_0\|_{\mathcal{H}}\exp\left(-\frac{M}{2}t\right)+C\left(\!\int_{0}^{t}\!\!{\left\| f(\cdot, \tau ) \right\|_{{{L}^{2}(0, 1)}}^{2}\mathrm{d}\tau }\right)^{\frac{1}{2}}+C\left(\!\int_{0}^{t}\!\!{\left\| h(\cdot, \tau ) \right\|_{{{L}^{2}(0, 1)}}^{2}\mathrm{d}\tau }\right)^{\frac{1}{2}},\quad \forall t\in \mathbb{R}_{\geq0},\label{eq:2.7}
	\end{align*}
	where $C$ and $M$ are positive constants, depending only on $\rho_{1}, \rho_{2}, k, b, k_{0}, \delta_{0}, \delta_{1}$, and the functions $g$ and $H$. In particular, when $f \equiv h\equiv0$, system \eqref{eq:sys} is exponentially stable, having the following estimate$:$
	\begin{equation*}
		\|X(t)\|_{\mathcal{H}}\leq C\|X_0\|_{\mathcal{H}}\exp\left(-\frac{M}{2}t\right),\quad \forall t\in \mathbb{R}_{\geq0}.
	\end{equation*}
\end{theorem}
\section{Well-posedness analysis}\label{sec:4}
In this section, we employ the semigroup theory for linear operators and the Lax-Milgram theorem, as well as the $L^2$ theory, for elliptic equations to prove the well-posedness of system \eqref{eq:sys}.
We first rewrite the system as an abstract evolution equation $\dot{U}=\mathcal{A}U+F$ in the Hilbert space $\mathcal{H}$, where $\mathcal{A}: D(\mathcal{A})\subset\mathcal{H}\to\mathcal{H}$ is a linear operator, which generates a contraction semigroup, $D(\mathcal{A})$ is the domain of $\mathcal{A}$, and $F$ represents the inhomogeneous term. To verify the conditions of the Lumer-Phillips theorem, the dissipativity of $\mathcal{A}$ is proved by using the system structure conditions and the inner product on $\mathcal{H}$. For surjectivity, that is, $R(I-\mathcal{A})=\mathcal{H}$, we prove it by employing the Lax-Milgram theorem and the $L^2$ theory for elliptic equations.
\par We first define the state space $\mathcal{H}$, the operator $\mathcal{A}$, and its domain $D(\mathcal{A})$ as follows:
\begin{align*}
	\mathcal{H}&=H_0^1(0, 1)\times L^2(0, 1)\times H_0^1(0, 1)\times L^2(0, 1)\times L_g^2(\mathbb{R}_{\geq0}; H_0^1(0, 1)),\\
	\mathcal{A}&=\begin{pmatrix}
		0 & 1 & 0 & 0 & 0 \\
		\frac{k}{\rho_1} \partial_x^2(\cdot) & 0 & \frac{k}{\rho_1} \partial_x(\cdot) & 0 & 0 \\
		0 & 0 & 0 & 1 & 0 \\
		-\frac{k}{\rho_2} \partial_x(\cdot) & 0 & \left(\frac{\hat{b}}{\rho_2} \partial_x^2 - \frac{k}{\rho_2} \right)(\cdot) & 0 & \frac{1}{\rho_2} \int_0^\infty g(s) \partial_x^2(\cdot) \, \mathrm{d}s \\
		0 & 0 & 0 & 1 & -\partial_s(\cdot)
	\end{pmatrix},\\
	D(\mathcal{A})&=\bigg\{(u, v, w, z, l)\in \mathcal{H} \mid u\in H^2(0, 1), \partial_x^{2}\left(\hat{b}w+\int_0^{\infty}g(s)l(s)\mathrm{d}s\right)\in L^2(0, 1), v, z\in H_0^1(0, 1),\\
	&\quad \partial_{s}l\in L_g^2(\mathbb{R}_{\geq0}; H_0^1(0, 1)), l(0)=0\bigg\}.
\end{align*}
\par With the above definitions, we present the result on the well-posedness of system \eqref{eq:sys}.
\begin{proposition}\label{pro:Proposition1}
	System \eqref{eq:sys} admits a unique mild solution $(\phi, \phi_t, \psi, \psi_t, \eta^t)\in C^{0}(\mathbb{R}_{\geq0}; \mathcal{H})$. Furthermore, if $(\phi_0, \phi_1, \psi_0(\cdot, 0), \psi_1, \eta_0) \in D(\mathcal{A})$ and $f, h\in C^{1}(\mathbb{R}_{\geq0}; L^{2}(0, 1))$, then system \eqref{eq:sys} admits a unique strong solution $(\phi, \phi_t, \psi,$ $\psi_t, \eta^t)\in C^{0}(\mathbb{R}_{\geq0}; D(\mathcal{A}))\cap C^{1}(\mathbb{R}_{\geq0}; \mathcal{H})$.
\end{proposition}
\begin{pf}
	Set $U=(\phi, \phi_t, \psi, \psi_t, \eta^t)^{\top}$, $U_0=(\phi_0, \phi_1, \psi_0(\cdot, 0), \psi_1, \eta_0)^{\top}$, $F=\left(0, \frac{f}{\rho_1}, 0, \frac{h}{\rho_2}, 0\right)^{\top}$.
	Then system~\eqref{eq:sys} can be written in an abstract form:
	\begin{equation}
		\dot{U}=\mathcal{A}U+F,\quad U(0)=U_0.\label{eq:U0}
	\end{equation}
	\par Next, we show in two steps that $\mathcal{A}$ is an infinitesimal generator of a contraction semigroup.
	\par \textbf{Step 1: prove the dissipativity of the operator $\mathcal{A}$.} By Assumption \ref{ass:Assumption2}, for any $U\in D(\mathcal{A})$, we have
	\begin{align*}
		\mathrm{Re}\langle\mathcal{A}U, U\rangle_\mathrm{H}
		&=\int_{0}^{1}{\left({{\rho }_{1}}{{\phi }_{t}(x, t)}{{\phi }_{tt}(x, t)}+{{\rho }_{2}}{{\psi }_{t}(x, t)}{{\psi }_{tt}(x, t)}+{\hat{b}{{\psi }_{x}(x, t)}{{\psi }_{xt}(x, t)}}\right)}\mathrm{d}x\nonumber\\
		&\quad+\int_{0}^{1}\left(k({{\phi }_{x}(x, t)}+\psi(x, t) ){{({{\phi }_{x}(x, t)}+\psi(x, t) )}_{t}}+{\int_{0}^{\infty }{g(s)}}\eta _{x}^{t}(x, s)\eta _{xt}^{t}(x, s)\mathrm{d}s\right)\mathrm{d}x\\
		& =k\int_{0}^{1}{\left({{({{\phi }_{x}(x, t)}+\psi(x, t) )}_{x}}{{\phi }_{t}(x, t)}\right)}\mathrm{d}x+\int_{0}^{1}{{{\psi }_{t}(x, t)}\int_{0}^{\infty }{g(s)\eta _{xx}^{t}(x, s)}}\mathrm{d}s\mathrm{d}x\nonumber\\
		&\quad +\hat{b}\int_{0}^{1}{\left({{\psi }_{xx}(x, t)}{{\psi }_{t}(x, t)}+{{\psi }_{x}(x, t)}{{\psi }_{xt}(x, t)}\right)}\mathrm{d}x+\int_{0}^{1}{\int_{0}^{\infty }{g(s)\eta _{x}^{t}(x, s)\eta _{xt}^{t}(x, s)\mathrm{d}s\mathrm{d}x}}\nonumber\\
		&\quad+k\int_{0}^{1}({{\phi }_{x}(x, t)}+\psi(x, t) )\left({{({{\phi }_{x}(x, t)}+\psi (x, t))}_{t}}-\psi_{t}(x, t)\right)\mathrm{d}x\nonumber\\  
		&=-\int_{0}^{\infty }{g(s)}\int_{0}^{1}{\eta _{xs}^{t}(x, s)\eta _{x}^{t}(x, s)\mathrm{d}x\mathrm{d}s}\nonumber\\ 
		&=\frac{1}{2}{{\int_{0}^{1}{\int_{0}^{\infty }{{g}'(s)\left| \eta _{x}^{t}(x, s) \right|^{2}}}}}\mathrm{d}s\mathrm{d}x\nonumber\\
		&\leq-\frac{k_0}{2}{{\int_{0}^{1}{\int_{0}^{\infty }{H(g(s))\left| \eta _{x}^{t}(x, s) \right|^{2}}}}}\mathrm{d}s\mathrm{d}x\nonumber\\
		&\leq0.
	\end{align*}
	Thus, $\mathcal{A}$  is a dissipative operator.
	\par \textbf{Step 2: prove the surjectivity of the operator $I-\mathcal{A}$.} For any $\tilde{F}=(f_1, f_2, f_3, f_4, f_5)^{\top}\in\mathcal{H}$, consider the following equation:
	\begin{align}
		(I-\mathcal{A})\tilde{U}=\tilde{F},\label{eq:U2}
	\end{align}
	where $\tilde{U}=(u, v, w, z, l)^{\top}$. By the definition of the operator $\mathcal{A}$, we obtain
	\begin{subequations}\label{eq:6.4all}
		\begin{align}
			u-v=f_1,\label{eq:6.4a}\\
			-\frac{k}{\rho_{1}}u_{xx}+v-\frac{k}{\rho_{1}}w_{x}=f_2,\label{eq:6.4b}\\
			w-z=f_3,\label{eq:6.4c}\\
			z-\frac{\hat{b}}{\rho_{2}}w_{xx}-\frac{1}{\rho_{2}}\int_{0}^{\infty}{g(s)l_{xx}(s)}\mathrm{d}s+\frac{k}{\rho_{2}}(u_x+w)=f_4,\label{eq:6.4d}\\
			-z+l+l_s=f_5.\label{eq:6.4e}
		\end{align}
	\end{subequations}
	From \eqref{eq:6.4a}, \eqref{eq:6.4c}, \eqref{eq:6.4e}, and $l(x, 0)=0$, it follows that
	\begin{align}
		v&=u-f_1,\label{eq:6.5a}\\
		z&=w-f_3,\label{eq:6.5b}\\
		l&=\left(1-e^{-s}\right)z+\int_{0}^{s}{e^{\tau-s}f_5(\tau)}\mathrm{d}\tau.\label{eq:6.5c}
	\end{align}
	For any $u_1, u_2 \in L^2(0,1)$, we define the inner product $\langle u_1, u_2 \rangle = \int_{0}^{1} u_1 u_2 \mathrm{d}x$. Substituting \eqref{eq:6.5a}, \eqref{eq:6.5b}, and \eqref{eq:6.5c} into \eqref{eq:6.4b} and \eqref{eq:6.4d}, we consider the following weak formulation for \eqref{eq:6.4b} and \eqref{eq:6.4d}, namely, for any test functions $\tilde{u}, \tilde{w} \in H_0^1(0,1)$, we consider
	\begin{subequations}\label{eq:6.7a}
		\begin{align}
			k\langle u_x, \tilde{u}_{x}\rangle+\rho_{1}\langle u, \tilde{u}\rangle+k\langle w, \tilde{u}_x\rangle&=\rho_{1}\langle f_1+f_2, \tilde{u}\rangle,\label{eq:6.7aa}\\
			k\langle u_x, \tilde{w}\rangle+k\langle w, \tilde{w}\rangle+C_{g}\langle w_x, \tilde{w}_x\rangle+\rho_{2}\langle w, \tilde{w}\rangle&=-\left\langle \int_{0}^{\infty}g(s)\int_{0}^{s}e^{\tau-s}f_{5x}(\tau)\mathrm{d}\tau\mathrm{d}s, \tilde{w}_{x}\right\rangle+\rho_{2}\langle f_3+f_4, \tilde{w}\rangle\nonumber\\
			&\quad+\int_{0}^{\infty}g(s)\left(1-e^{-s}\right)\mathrm{d}s\langle f_{3x}, \tilde{w}_{x}\rangle,\label{eq:6.7ab}
		\end{align}
	\end{subequations}
	where, in view of Assumption \ref{ass:Assumption2}, $C_{g}=\hat{b}+\int_{0}^{\infty}g(s)\left(1-e^{-s}\right)\mathrm{d}s$ is a positive constant.
	Set $V=H_0^1(0, 1)\times H_0^1(0, 1)$. We define the bilinear functional on $V$ as follows:
	\begin{align*}
		a\left(Z_{1}, Z_{2}\right)=\rho_{1}\langle u, \tilde{u}\rangle+\rho_{2}\langle w, \tilde{w}\rangle+C_{g}\langle w_x, \tilde{w}_x\rangle+k\langle u_{x}+w, \tilde{u}_{x}+\tilde{w}\rangle,
	\end{align*}
	where $Z_{1}=(u, w)$ and $Z_{2}=(\tilde{u}, \tilde{w})$. Then, \eqref{eq:6.7aa} and \eqref{eq:6.7ab} can be written as
	\begin{align*}
		a\left(Z_{1}, Z_{2}\right)=\mathcal{F}(Z_{2}),
	\end{align*}
	where 
	\begin{align*}
		\mathcal{F}(Z_{2})\!=\!\rho_{1}\langle f_1\!+\! f_2, \tilde{u}\rangle\!+\!\rho_{2}\langle f_3\!+\!f_4, \tilde{w}\rangle\!+\!\int_{0}^{\infty}\!\!\! g(s)\left(1\!-\!e^{-s}\right)\mathrm{d}s\langle f_{3x}, \tilde{w}_{x}\rangle\!-\!\left\langle \int_{0}^{\infty}\!\!\! g(s)\int_{0}^{s}\!\!\! e^{\tau\!-\! s}f_{5x}(\tau)\mathrm{d}\tau\mathrm{d}s, \tilde{w}_{x}\right\rangle.
	\end{align*}
	By the H\"{o}lder inequality, it follows that
	\begin{align*}
		|\mathcal{F}(Z_{2})|&\leq \rho_{1}\|f_1+f_2\|_{L^{2}(0, 1)}\|\tilde{u}\|_{L^{2}(0, 1)}+\rho_{2}\|f_3+f_4\|_{L^{2}(0, 1)}\|\tilde{w}\|_{L^{2}(0, 1)} +g_{0}\|f_3\|_{H_0^1(0, 1)}\|\tilde{w}\|_{H_0^1(0, 1)}\\
		&\quad +\int_{0}^{\infty}g(s)\int_{0}^{s}e^{\tau-s}\|f_{5x}(\tau)\|_{L^{2}(0, 1)}\mathrm{d}\tau\mathrm{d}s\|\tilde{w}_{x}\|_{L^{2}(0, 1)}.
	\end{align*}
	Note that
	\begin{align}
		\int_{0}^{\infty}g(s)\int_{0}^{s}e^{\tau-s}\|f_{5x}(\tau)\|_{L^{2}(0, 1)}\mathrm{d}\tau\mathrm{d}s
		&=\int_{0}^{\infty}\int_{\tau}^{\infty}g(s)e^{\tau-s}\mathrm{d}s\|f_{5x}(\tau)\|_{L^{2}(0, 1)}\mathrm{d}\tau\nonumber\\
		&=\int_{0}^{\infty}\int_{\tau}^{\infty}g(s)e^{\tau-s}\mathrm{d}s\left(g^{-\frac{1}{2}}(\tau)g^{\frac{1}{2}}(\tau)\right)\|f_{5x}(\tau)\|_{L^{2}(0, 1)}\mathrm{d}\tau\nonumber\\
		&\leq \left(\int_{0}^{\infty}g^{-1}(\tau)\left(\int_{\tau}^{\infty}g(s)e^{\tau-s}\mathrm{d}s\right)^{2}\mathrm{d}\tau\right)^{\frac{1}{2}}\left(\int_{0}^{\infty}g(\tau)\|f_{5x}(\tau)\|_{L^{2}(0, 1)}^{2}\mathrm{d}\tau\right)^{\frac{1}{2}}\nonumber\\
		&\leq \left(\int_{0}^{\infty}g^{-1}(\tau)\left(g(\tau)\int_{\tau}^{\infty}e^{\tau-s}\mathrm{d}s\right)^{2}\mathrm{d}\tau\right)^{\frac{1}{2}}\|f_{5}\|_{L_{g}^{2}(\mathbb{R}_{\geq0}; H_0^1(0, 1))}\nonumber\\
		&\leq \sqrt{g_0}\|f_{5}\|_{L_{g}^{2}(\mathbb{R}_{\geq0}; H_0^1(0, 1))}.\label{eq:4.8a}
	\end{align}
	Setting $\xi_{1}=\max\left\{\rho_{1}, \rho_{2}, g_0, \sqrt{g_0}\right\}$, we obtain
	\begin{align*}
		|\mathcal{F}(Z_{2})|&\leq \xi_{1}\left(\|f_1+f_2\|_{L^{2}(0, 1)}+\|f_3+f_4\|_{L^{2}(0, 1)}+\|f_3\|_{H_0^1(0, 1)}+\|f_{5}\|_{L_{g}^{2}(\mathbb{R}_{\geq0}; H_0^1(0, 1))}\right)\|Z_{2}\|_{V}.
	\end{align*}
	We next prove the boundedness and coercivity of the bilinear functional $a\left(Z_{1}, Z_{2}\right)$.
	\begin{enumerate}[(i)]
		\item Boundedness of the bilinear functional $a\left(Z_{1}, Z_{2}\right)$.
		\begin{align*}
			|a\left(\! Z_{1}, Z_{2}\!\right)|&\leq \rho_{1}\|u\|_{L^{2}(0, 1)}\|\tilde{u}\|_{L^{2}(0, 1)}\!+\!\rho_{2}\|w\|_{L^{2}(0, 1)}\|\tilde{w}\|_{L^{2}(0, 1)}\!+\! C_{g}\|w_{x}\|_{L^{2}(0, 1)}\|\tilde{w}_{x}\|_{L^{2}(0, 1)}\\
			&\quad+\! k\|u_{x}\!+\! w\|_{L^{2}(0, 1)}\|\tilde{u}_{x}\!+\!\tilde{w}\|_{L^{2}(0, 1)}\\
			&\leq \left(\!\rho_{1}\!+\!\rho_{2}\!+\! C_{g}\!\right)\|Z_{1}\|_{V}\|Z_{2}\|_{V} \!+\!2k\left(\!\sqrt{\|u_{x}\|_{L^2(0, 1)}^{2}\!+\!\|w\|_{L^2(0, 1)}^{2}}\!\right)\left(\!\sqrt{\|\tilde{u}_{x}\|_{L^2(0, 1)}^{2}\!+\!\|\tilde{w}\|_{L^2(0, 1)}^{2}}\!\right)\\
			&\leq \left(\!\rho_{1}\!+\!\rho_{2}\!+\! C_{g}\!+\!2k\!\right)\|Z_{1}\|_{V}\|Z_{2}\|_{V}.
		\end{align*}
		\item Coercivity of the bilinear functional $a\left(Z_{1}, Z_{2}\right)$. Set $\xi_{2}=\min\left\{\rho_{1}, \rho_{2}, C_{g}, k\right\}$. It follows that
		\begin{align*}
			a\left(Z_{1}, Z_{1}\right)&=\rho_{1}\|u\|_{L^{2}(0, 1)}^{2}+\rho_{2}\|w\|_{L^{2}(0, 1)}^{2}+C_{g}\|w_{x}\|_{L^{2}(0, 1)}^{2}+k\|u_{x}+w\|_{L^{2}(0, 1)}^{2}\\
			&\geq \xi_{2}\left(\|u\|_{L^{2}(0, 1)}^{2}+2\|w\|_{L^{2}(0, 1)}^{2}+\|w_{x}\|_{L^{2}(0, 1)}^{2}+\|u_{x}\|_{L^{2}(0, 1)}^{2}+2\int_{0}^{1}u_{x}w\mathrm{d}x\right).
		\end{align*}
		By Young's inequality, we obtain
		\begin{align*}
			2\int_{0}^{1}|u_{x}w|\mathrm{d}x\leq \frac{2}{3}\int_{0}^{1}|u_{x}|^{2}\mathrm{d}x+\frac{3}{2}\int_{0}^{1}w^{2}\mathrm{d}x,
		\end{align*}
		which yields
		\begin{align*}
			2\int_{0}^{1}u_{x}w\mathrm{d}x\geq -\frac{2}{3}\int_{0}^{1}|u_{x}|^{2}\mathrm{d}x-\frac{3}{2}\int_{0}^{1}w^{2}\mathrm{d}x.
		\end{align*}
		Thus, it follows that
		\begin{align*}
			a\left(Z_{1}, Z_{1}\right)&\geq \xi_{2}\left(\|u\|_{L^{2}(0, 1)}^{2}+\frac{1}{2}\|w\|_{L^{2}(0, 1)}^{2}+\|w_{x}\|_{L^{2}(0, 1)}^{2}+\frac{1}{3}\|u_{x}\|_{L^{2}(0, 1)}^{2}\right)\\
			&\geq \frac{1}{3}\xi_{2}\left(\|u\|_{L^{2}(0, 1)}^{2}+\|w\|_{L^{2}(0, 1)}^{2}+\|w_{x}\|_{L^{2}(0, 1)}^{2}+\|u_{x}\|_{L^{2}(0, 1)}^{2}\right)\\
			&\geq  \frac{2}{3}\xi_{2}\|Z_{1}\|_{V}^{2}.
		\end{align*}
	\end{enumerate}
	Then, by the Lax-Milgram theorem \cite[p.315, Theorem 1]{Evans:2010}, equation \eqref{eq:6.7a} admits a unique weak solution $(u,w)\in H_0^1(0, 1)\times H_0^1(0, 1)$. Combining \eqref{eq:6.5a}, \eqref{eq:6.5b}, \eqref{eq:6.5c}, and \eqref{eq:4.8a}, we deduce that \eqref{eq:6.4all} admits a unique weak solution $\tilde{U}=(w, v, w, z, l)\in H_0^1(0, 1) \times H_0^1(0, 1) \times H_0^1(0, 1) \times H_0^1(0, 1) \times L_{g}^{2}(\mathbb{R}_{\geq0}; H_0^1(0, 1)))$. To ensure that the solution $\tilde{U}\in D(\mathcal{A})$, we need to improve the regularity of the weak solution. Indeed, \eqref{eq:6.4e} implies that $l_{s}\in L_{g}^{2}(\mathbb{R}_{\geq0}; H_0^1(0, 1))$. Let $\mathcal{G}=\hat{b}w+\int_{0}^{\infty}g(s)l(s)\mathrm{d}s$. Then from \eqref{eq:6.4b} and \eqref{eq:6.4d}, we obtain
	\begin{align}
		u_{xx}&=\frac{\rho_{1}}{k}v-w_{x}-\frac{\rho_{1}}{k}f_2,\label{eq:4.8}\\
		\mathcal{G}_{xx}&=ku_x+kw+\rho_{2}z-\rho_{2}f_4.\label{eq:4.9}
	\end{align}
	Moreover, the right-hand sides of \eqref{eq:4.8} and \eqref{eq:4.9} belong to $L^{2}(0, 1)$. By the $L^2$-regularity theory for elliptic equations \cite[Chapter 6.3]{Evans:2010}, we have $u, \mathcal{G}\in H^{2}(0, 1)$. Thus, $\tilde{U}=(u, v, w, z, l)\in D(\mathcal{A})$, which implies that the operator $I-\mathcal{A}$ is surjective. By the Lumer-Phillips theorem \cite[p.14, Theorem 4.3]{Pazy1983}, $\mathcal{A}$ is an infinitesimal generator of a contraction semigroup on the Hilbert space $\mathcal{H}$. 
	\par 
	According to the semigroup theory (see, e.g., \cite[p.106]{Pazy1983}), we conclude that for $U_0\in \mathcal{H}$ and $F\in L^{2}(\mathbb{R}_{\geq0}; \mathcal{H})$, equation \eqref{eq:U0} admits a unique mild solution $U\in C^{0}(\mathbb{R}_{\geq0}; \mathcal{H})$. Furthermore, $U_0\in D(\mathcal{A})$ and $F\in C^{1}(\mathbb{R}_{\geq0}; \mathcal{H})$, then equation \eqref{eq:U0} admits a unique strong solution $U\in C^{0}(\mathbb{R}_{\geq0}; D(\mathcal{A}))\cap C^{1}(\mathbb{R}_{\geq0}; \mathcal{H})$; see \cite[p.107 Corollary 2.5]{Pazy1983}. This completes the proof of Proposition\ref{pro:Proposition1}.$\hfill\square$
\end{pf}
\section{Construction of auxiliary functionals}\label{sec:5}In this section, we construct several auxiliary functionals to deal with the product terms of $ \phi, \psi$, and $\eta^{t}$, as well as their derivatives, and establish several \textit{a priori} estimates for them. Following the method in~\cite{Messaoudi2009,Rivera2002}, these functionals will be used to facilitate the construction of a Lyapunov functional.
\par For all $t\in \mathbb{R}_{\geq0}$ and states $(\phi, \phi_{t}, \psi, \psi_{t}, {\eta}^{t})$ of system \eqref{eq:sys}, we consider the following auxiliary functionals,
\begin{align}
	{{I}_{1}}(t)&\!=\!\int_{0}^{1}\!\!\!{\left(\!{{\rho }_{2}}{{\psi }_{t}(x, t)}\psi(x, t) \!+\!{{\rho }_{1}}{{\phi }_{t}(x, t)}w(x, t)\!\right)}\mathrm{d}x,\label{eq:13}\\
	{{I}_{2}}(t)&\!=\!-\!{{\rho }_{2}}\int_{0}^{1}\!\!\!{{{\psi }_{t}}(x, t)\int_{0}^{\infty }\!\!\!\!{g(s){{\eta }^{t}}(x, s)\mathrm{d}s\mathrm{d}x}},\label{eq:14}\\
	J(t)&\!=\!{{\rho }_{2}}\int_{0}^{1}\!\!\!{{{\psi }_{t}(x, t)}\left(\!{{\phi }_{x}(x, t)}\!+\!\psi(x, t) \!\right)}\mathrm{d}x \! +\!\frac{{{\rho }_{1}}\hat{b}}{k}\int_{0}^{1}\!\!\!{{{\psi }_{x}(x, t)}{{\phi }_{t}(x, t)}\mathrm{d}x}\!+\!\frac{{{\rho }_{1}}}{k}\int_{0}^{1}\!\!\!{{{\phi }_{t}}(x, t)\int_{0}^{\infty }\!\!\!\!{g(s)\eta _{x}^{t}(x, s)\mathrm{d}s\mathrm{d}x}},\label{eq:15}\\
	K(t)&\!=\!-{{\rho }_{1}}\int_{0}^{1}\!\!\!{\phi(x, t) {{\phi }_{t}(x, t)}\mathrm{d}x}\!-\!{{\rho }_{2}}\int_{0}^{1}\!\!\!{\psi(x, t) {{\psi }_{t}(x, t)}\mathrm{d}x},\label{eq:16}
\end{align}
where $w$ is the solution to the equation
\begin{align*}
	-{{w}_{xx}}&={{\psi }_{x}},\\
	w(0)&=w(1)=0.
\end{align*}
\par We first establish two lemmas, which will facilitate the subsequent derivations.
\begin{lemma}\label{lem:Lemma3-8} For any $\delta_{1}\in \big[1, \frac{3}{2}\big)$, the functions $g(t), g^{\frac{1}{2}}(t)$ ,and $g^{2-\delta_{1}}(t)$ are integrable over $(0, \infty)$.
\end{lemma}
\begin{pf}
	From \eqref{eq:3.5}, it follows that for any $t \in {\mathbb{R}}_{\geq0}$,
	\begin{align*}
		\min\left\{\left(\frac{g(t)}{g(0)}\right)^{\delta_0}, \left(\frac{g(t)}{g(0)}\right)^{\delta_{1}}\right\}H(g(0))&\leq H(g(t))\leq \max\left\{\left(\frac{g(t)}{g(0)}\right)^{\delta_0}, \left(\frac{g(t)}{g(0)}\right)^{\delta_{1}}\right\}H(g(0)).
	\end{align*}
	Since $g\in C^{0}({{\mathbb{R}}_{\geq 0}}; {{\mathbb{R}}_{>0}}) \cap C^{1}({{\mathbb{R}}_{\geq 0}}; {{\mathbb{R}_{\leq0}}})$, we have
	\begin{equation}
		H(g(t))\geq \left(\frac{g(t)}{g(0)}\right)^{\delta_1}H(g(0)),\quad\forall t\in {\mathbb{R}}_{\geq0}.\label{eq:2.6}
	\end{equation}
	Let $k_1=k_0\frac{H(g(0))}{g^{\delta_{1}}(0)}$. Then we obtain
	\begin{align*}
		g'(t)\leq -k_0H(g(t))\leq -k_1g^{\delta_{1}}(t), \quad\forall t\in \mathbb{R}_{>0}.
	\end{align*}
	When $\delta_{1}=1$, it follows from Gronwall's inequality that
	\begin{equation*}
		g(t)\leq g(0)\exp(-k_{1}t), \quad\forall t\in \mathbb{R}_{\geq0}.
	\end{equation*}
	Thus, we have
	\begin{align*}
		\int_{0}^{\infty}{g(t)}\mathrm{d}t&\leq\int_{0}^{\infty}{g(0)\exp(-k_1t)}\mathrm{d}t=\frac{g(0)}{k_1}<+\infty,\\
		\int_{0}^{\infty}{g^{\frac{1}{2}}(t)}\mathrm{d}t&\leq\int_{0}^{\infty}{g^{\frac{1}{2}}(0)\exp\left(-\frac{k_1}{2}t\right)}\mathrm{d}t=\frac{2g^{\frac{1}{2}}(0)}{k_1}<+\infty,\\
		\int_{0}^{\infty}{g^{2-\delta_{1}}(t)}\mathrm{d}t&=\int_{0}^{\infty}{g(t)}\mathrm{d}t<+\infty.
	\end{align*}
	For $\delta_{1}\in \left(1, \frac{3}{2}\right)$, set $k_2=\min\left\{g^{-(\delta_{1}-1)}(0), k_1(\delta_{1}-1)\right\}$. Then we have
	\begin{align*}
		g(t)\leq \left(g^{-(\delta_{1}-1)}(0)+k_1(\delta_{1}-1)t\right)^{-\frac{1}{\delta_{1}-1}}\leq k_2^{-\frac{1}{\delta_{1}-1}}(1+t)^{-\frac{1}{\delta_{1}-1}}, \quad t\in \mathbb{R}_{\geq0}.
	\end{align*}
	Since $\delta_{1}\in \left(1, \frac{3}{2}\right)$, it follows that $\frac{1}{\delta_{1}-1}>1, \frac{1}{2(\delta_{1}-1)}>1$ and $\frac{2-\delta_{1}}{\delta_{1}-1}>1$. Therefore,
	\begin{align*}
		\int_{0}^{\infty}{g(t)}\mathrm{d}t&\leq k_2^{-\frac{1}{\delta_{1}-1}}\int_{0}^{\infty}(1+t)^{-\frac{1}{\delta_{1}-1}}\mathrm{d}t<+\infty,\\
		\int_{0}^{\infty}{g^{\frac{1}{2}}(t)}\mathrm{d}t&\leq k_2^{-\frac{1}{2(\delta_{1}-1)}}\int_{0}^{\infty}(1+t)^{-\frac{1}{2(\delta_{1}-1)}}\mathrm{d}t<+\infty,\\
		\int_{0}^{\infty}{g^{2-\delta_{1}}(t)}\mathrm{d}t&\leq k_2^{-\frac{2-\delta_{1}}{\delta_{1}-1}}\int_{0}^{\infty}(1+t)^{-\frac{2-\delta_{1}}{\delta_{1}-1}}\mathrm{d}t<+\infty.
	\end{align*}
	This completes the proof of Lemma \ref{lem:Lemma3-8}.$\hfill\square$
\end{pf}
\begin{lemma}\label{lem:Lemma3-6}Set~$G_{0}=\int_{0}^{\infty }{{{g}^{2-{\delta_{1}}}}(s)\mathrm{d}s}$. For $\delta_{1}\in \big[1, \frac{3}{2}\big)$, it holds that
	\begin{align}
		\int_{0}^{1}{{{\left(\int_{0}^{\infty }{g(s)\eta _{x}^{t}}(x, s)\mathrm{d}s\right)^{2}}}\mathrm{d}x\le }{G_{0}}\int_{0}^{1}{\int_{0}^{\infty }{{{g}^{{\delta_{1}}}}(s){{\left| \eta _{x}^{t}(x, s) \right|}^{2}}\mathrm{d}s\mathrm{d}x}},\quad \forall t\in \mathbb{R}_{\geq0},\label{eq:8}\\
		\int_{0}^{1}{{{\left(\int_{0}^{\infty }{g(s){{\eta }^{t}}}(x, s)\mathrm{d}s\right)^{2}}}\mathrm{d}x\le }{G_{0}}\int_{0}^{1}{\int_{0}^{\infty }{{{g}^{{\delta_{1}}}}(s){{\left| \eta _{x}^{t}(x, s) \right|}^{2}}\mathrm{d}s\mathrm{d}x}},\quad \forall t\in \mathbb{R}_{\geq0}.\label{eq:9}
	\end{align}
\end{lemma}
\begin{pf}By H\"{o}lder's inequality, for all $t\in \mathbb{R}_{\geq0}$, we obtain
	\begin{align*} 
		\int_{0}^{\infty }{g(s)\eta _{x}^{t}}(x, s)\mathrm{d}s&=\int_{0}^{\infty }{{{g}^{1-\frac{{\delta_{1}}}{2}}}(s){{g}^{\frac{{\delta_{1}}}{2}}}(s)}\eta _{x}^{t}(x, s)\mathrm{d}s\\
		&\le {{\left(\int_{0}^{\infty }{{{g}^{2-{\delta_{1}}}}(s)\mathrm{d}s}\right)}^{\frac{1}{2}}}{{\left(\int_{0}^{\infty }{{{g}^{{\delta_{1}}}}(s){{\left| \eta _{x}^{t}(x, s) \right|}^{2}}\mathrm{d}s}\right)}^{\frac{1}{2}}},
	\end{align*}
	which implies that
	\begin{align}
		\int_{0}^{1}{{{\left(\int_{0}^{\infty }{g(s)\eta _{x}^{t}}(x, s)\mathrm{d}s\right)^{2}}}\mathrm{d}x }
		&\le {G_{0}}\int_{0}^{1}{\int_{0}^{\infty }{{{g}^{{\delta_{1}}}}(s){{\left| \eta _{x}^{t}(x, s) \right|}^{2}}\mathrm{d}s\mathrm{d}x}}.\label{eq:10}
	\end{align}
	Similarly, by H\"{o}lder's inequality, for all $t\in \mathbb{R}_{\geq0}$, we obtain
	\begin{equation}
		{{\left(\int_{0}^{\infty }{g(s)\eta ^{t}}(x, s)\mathrm{d}s\right)^{2}}}\le {G_{0}}\int_{0}^{\infty }{{{g}^{{\delta_{1}}}}(s){{\left| \eta ^{t}(x, s) \right|}^{2}}\mathrm{d}s}.\label{eq:11}
	\end{equation}
	By Friedrichs' inequality (\cite[p. 262]{Mironchenko2023pre}), for all $t\in \mathbb{R}_{\geq0}$, it holds that
	\begin{equation}
		\int_{0}^{1}\int_{0}^{\infty }{{{g}^{{\delta_{1}}}}(s){{\left| \eta ^{t}(x, s) \right|}^{2}}\mathrm{d}s}\mathrm{d}x\leq \int_{0}^{1}\int_{0}^{\infty }{{{g}^{{\delta_{1}}}}(s){{\left| \eta_{x} ^{t}(x, s) \right|}^{2}}\mathrm{d}s}\mathrm{d}x.\label{eq:12}
	\end{equation}
	Combining \eqref{eq:10}, \eqref{eq:11}, and \eqref{eq:12}, we get \eqref{eq:8} and \eqref{eq:9}.$\hfill\square$
\end{pf}
\par Next, we estimate the derivative of ${{I}_{1}}(t)$.
\begin{lemma}\label{lem:Lemma4-1} Let~$(\phi, \phi_t, \psi, \psi_t, \eta^t)$ be a solution to system~\eqref{eq:sys}. Then, for any~${{\varepsilon }_{1}}\in {\mathbb{R}}_{>0}$ and ${{\lambda }_{1}}\in {\mathbb{R}}_{>0}$, it holds that
	\begin{align}
		\frac{\mathrm{d}}{\mathrm{d}t}{{I}_{1}}(t)&\le \!\left(\!-\hat{b}+3{{\lambda }_{1}}\!\right)\!\int_{0}^{1}\!{\psi _{x}^{2}(x, t)}\mathrm{d}x+{{\varepsilon }_{1}}{{\rho }_{1}}\!\int_{0}^{1}\!{\phi _{t}^{2}(x, t)}\mathrm{d}x+\!\left(\!{{\rho }_{2}}+\frac{{{\rho }_{1}}}{4{{\varepsilon }_{1}}}\!\right)\!\int_{0}^{1}\!{\psi _{t}^{2}(x, t)}\mathrm{d}x\nonumber\\
		&\quad+\frac{{G_{0}}}{4{{\lambda }_{1}}}\!\int_{0}^{1}\!{\int_{0}^{\infty }\!{{{g}^{{\delta_{1}}}}(s){{\left| \eta _{x}^{t}(x, s) \right|}^{2}}}}\mathrm{d}s\mathrm{d}x+\frac{1}{4{{\lambda }_{1}}}\!\int_{0}^{1}\!\left(\!{f}^{2}(x, t)+{h}^{2}(x, t)\!\right)\!\mathrm{d}x, \quad\forall t\in \mathbb{R}_{>0},\label{eq:17}
	\end{align}
	where $G_{0}$ is defined as in Lemma \ref{lem:Lemma3-6}.
\end{lemma}
\begin{pf}Combining \eqref{eq:sys} and \eqref{eq:13}, for all~$t\in \mathbb{R}_{>0}$, it follows from integration by parts that
	\begin{align} 
		\frac{\mathrm{d}}{\mathrm{d}t}{{I}_{1}}(t)&=\int_{0}^{1}\!\!{\psi(x, t) \left(\hat{b}{{\psi }_{xx}(x, t)}+\int_{0}^{\infty }\!\!{g(s)}\eta _{xx}^{t}(x, s)\mathrm{d}s-k({{\phi }_{x}(x, t)}+\psi(x, t) )\right)}\mathrm{d}x+\int_{0}^{1}\!\!{{{\rho }_{2}}\psi _{t}^{2}(x, t)\mathrm{d}x}\nonumber\\
		&\quad+\int_{0}^{1}\!\!{k{{({{\phi }_{x}(x, t)}+\psi(x, t))}_{x}}w(x, t)\mathrm{d}x}+\int_{0}^{1}\!\!\psi(x, t)h(x, t)\mathrm{d}x+\int_{0}^{1}\!\!{f(x, t)w(x, t)\mathrm{d}x}+\int_{0}^{1}\!\!{{{\rho }_{1}}{{\phi }_{t}(x, t)}{{w}_{t}(x, t)}\mathrm{d}x}\nonumber \\ 
		& =-\hat{b}\int_{0}^{1}\!\!{\psi _{x}^{2}(x, t)\mathrm{d}x}-\int_{0}^{1}\!\!{{{\psi }_{x}}(x, t)\int_{0}^{\infty }\!\!{g(s)\eta _{x}^{t}(x, s)\mathrm{d}s\mathrm{d}x}}-k\int_{0}^{1}\!\!{{{\psi }^{2}(x, t)}\mathrm{d}x}+k\int_{0}^{1}\!\!{w_{x}^{2}(x, t)\mathrm{d}x}\nonumber\\
		&\quad+\int_{0}^{1}\!\!{h(x, t)\psi(x, t) \mathrm{d}x}+{{\rho }_{2}}\int_{0}^{1}\!\!{\psi _{t}^{2}(x, t)\mathrm{d}x}+\int_{0}^{1}\!\!{f(x, t)w(x, t)\mathrm{d}x}+{{\rho }_{1}}\int_{0}^{1}\!\!{{{\phi }_{t}(x, t)}{{w}_{t}(x, t)}\mathrm{d}x}.\label{eq:18}
	\end{align}
	Next, we estimate each term on the right-hand side of \eqref{eq:18}. By Friedrichs' inequality, for all $t \in {\mathbb{R}}_{>0}$, we obtain
	\begin{align}
		&\int_{0}^{1}{w^{2}(x, t)\mathrm{d}x}\leq\int_{0}^{1}{w_{x}^{2}(x, t)\mathrm{d}x}\le \int_{0}^{1}{{{\psi }^{2}(x, t)}}\mathrm{d}x\le \int_{0}^{1}{\psi _{x}^{2}(x, t)}\mathrm{d}x,\label{eq:19} \\ 
		&\int_{0}^{1}{w_{t}^{2}(x, t)\mathrm{d}x}\le \int_{0}^{1}{w_{tx}^{2}(x, t)\mathrm{d}x}\le \int_{0}^{1}{\psi _{t}^{2}(x, t)}\mathrm{d}x.\label{eq:20} 
	\end{align}
	By Young's inequality, \eqref{eq:19}, and \eqref{eq:20}, for any~${{\varepsilon }_{1}}\in {\mathbb{R}}_{>0}, {{\lambda }_{1}} \in {\mathbb{R}}_{>0}$, and $t \in {\mathbb{R}}_{>0}$, we obtain
	\begin{align*}
		{{\rho }_{1}}\int_{0}^{1}{{{\phi }_{t}(x, t)}{{w}_{t}(x, t)}\mathrm{d}x}&\le {{\varepsilon }_{1}}{{\rho }_{1}}\int_{0}^{1}{\phi _{t}^{2}(x, t)}\mathrm{d}x+\frac{{{\rho }_{1}}}{4{{\varepsilon }_{1}}}\int_{0}^{1}{\psi _{t}^{2}(x, t)}\mathrm{d}x,\\ 
		\int_{0}^{1}{h(x, t)\psi(x, t) }\mathrm{d}x&\le {{\lambda }_{1}}\int_{0}^{1}{\psi _{x}^{2}(x, t)}\mathrm{d}x+\frac{1}{4{{\lambda }_{1}}}\int_{0}^{1}{{{h}^{2}}(x, t)\mathrm{d}x}, \\ 
		\int_{0}^{1}{f(x, t)w(x, t) }\mathrm{d}x&\le {{\lambda }_{1}}\int_{0}^{1}{\psi _{x}^{2}(x, t)}\mathrm{d}x+\frac{1}{4{{\lambda }_{1}}}\int_{0}^{1}{{{f}^{2}}(x, t)\mathrm{d}x},
	\end{align*}
	and
	\begin{align*}
		-\int_{0}^{1}{{{\psi }_{x}}(x, t)\int_{0}^{\infty }{g(s)\eta _{x}^{t}(x, s)\mathrm{d}s\mathrm{d}x}}&\le \frac{1}{4{{\lambda }_{1}}}\int_{0}^{1}{{{\left(\int_{0}^{\infty }{g(s)\eta _{x}^{t}}(x, s)\mathrm{d}s\right)^{2}}}\mathrm{d}x}+{{\lambda }_{1}}\int_{0}^{1}{\psi _{x}^{2}(x, t)\mathrm{d}x}\\
		&\le\frac{{G_{0}}}{4{{\lambda }_{1}}}\int_{0}^{1}{\int_{0}^{\infty }{{{g}^{{\delta_{1}}}}(s){{\left| \eta _{x}^{t}(x, s) \right|}^{2}}\mathrm{d}s\mathrm{d}x}}+{{\lambda }_{1}}\int_{0}^{1}{\psi _{x}^{2}(x, t)\mathrm{d}x}.
	\end{align*}
	Substituting the above estimates into \eqref{eq:18}, we obtain \eqref{eq:17}.$\hfill\square$
\end{pf}
\par The following lemma provides a bound for the derivative of ${{I}_{2}}(t)$.
\begin{lemma}\label{lem:Lemma4-2} Let~$(\phi, \phi_t, \psi, \psi_t, \eta^t)$ be a solution to system \eqref{eq:sys}. Then, for any~${{\varepsilon }_{2}}\in {\mathbb{R}}_{>0}$ and $t\in \mathbb{R}_{>0}$, it holds that
	\begin{align}
		\frac{\mathrm{d}}{\mathrm{d}t}{{I}_{2}}(t)&\leq-\frac{{{\rho }_{2}}{{g}_{0}}}{2}\int_{0}^{1}\!\!{\psi _{t}^{2}(x, t)\mathrm{d}x}+{{\varepsilon }_{2}}{{{\hat{b}}}^{2}}\int_{0}^{1}\!\!{\psi _{x}^{2}(x, t)\mathrm{d}x}+\frac{1}{4{{\varepsilon }_{2}}}\int_{0}^{1}\!\!{{{h}^{2}}(x, t)\mathrm{d}x}+{{\varepsilon }_{2}}{{k}^{2}}\int_{0}^{1}\!\!{{{({{\phi }_{x}(x, t)}+\psi(x, t) )}^{2}}\mathrm{d}x}\nonumber\\
		&\quad-\frac{{{\rho }_{2}}g(0)}{2{{g}_{0}}}\int_{0}^{1}\!\!{{{\int_{0}^{\infty }\!\!{g'(s)\left| \eta _{x}^{t}(x, s) \right|^{2}}}}}\mathrm{d}s\mathrm{d}x+{G_{0}}\left(1+\frac{1}{2{{\varepsilon }_{2}}}+{{\varepsilon }_{2}}\right)\int_{0}^{1}\!\!{{{\int_{0}^{\infty }\!\!{{{g}^{{\delta_{1}}}}(s)\left| \eta _{x}^{t}(x, s) \right|^{2}}}}}\mathrm{d}s\mathrm{d}x,\label{eq:21} 
	\end{align}
	where $G_{0}$ is defined as in Lemma \ref{lem:Lemma3-6}.
\end{lemma}
\begin{pf}Combining \eqref{eq:sys} and \eqref{eq:14}, for all $t\in \mathbb{R}_{>0}$, it follows from integration by parts that
	\begin{align}
		\frac{\mathrm{d}}{\mathrm{d}t}{{I}_{2}}(t)
		& =\hat{b}\int_{0}^{1}{{{\psi }_{x}(x, t)}}\int_{0}^{\infty }{g(s)\eta _{x}^{t}(x, s)\mathrm{d}s\mathrm{d}x}+\int_{0}^{1}{{{\left(\int_{0}^{\infty }{g(s)\eta _{x}^{t}(x, s)\mathrm{d}s}\right)^{2}}}\mathrm{d}x}-{{\rho }_{2}}{{g}_{0}}\int_{0}^{1}{\psi _{t}^{2}(x, t)}\mathrm{d}x\nonumber\\
		&\quad+k\int_{0}^{1}{({{\phi }_{x}(x, t)}+\psi(x, t) )\int_{0}^{\infty }{g(s){{\eta }^{t}}(x, s)\mathrm{d}s\mathrm{d}x}}+{{\rho }_{2}}\int_{0}^{1}{{{\psi }_{t}}(x, t)\int_{0}^{\infty }{g(s)\eta _{s}^{t}(x, s)\mathrm{d}s\mathrm{d}x}}\nonumber\\
		&\quad-\int_{0}^{1}{h(x, t)\int_{0}^{\infty }{g(s){{\eta }^{t}}(x, s)\mathrm{d}s\mathrm{d}x}}.  \label{eq:22}
	\end{align}
	Next, we estimate each term on the right-hand side of \eqref{eq:22}. By Young's inequality and Lemma \ref{lem:Lemma3-6}, for any~${{\varepsilon }_{2}}\in {\mathbb{R}}_{>0}$ and $ t \in {\mathbb{R}}_{>0}$, we obtain
	\begin{align}
		\!\int_{0}^{1}\!\!{{{\left(\!\int_{0}^{\infty }\!\!{g(s)\eta _{x}^{t}(x, s)\mathrm{d}s}\right)}^{2}}\mathrm{d}x}&\le {G_{0}}\!\int_{0}^{1}\!\!{\int_{0}^{\infty }\!\!{{{g}^{{\delta_{1}}}}(s){{\left| \eta _{x}^{t}(x, s) \right|}^{2}}\mathrm{d}s\mathrm{d}x}},\label{eq:5.7a}\\
		\hat{b}\!\int_{0}^{1}\!\!{{{\psi }_{x}(x, t)}}\!\int_{0}^{\infty }\!\!{g(s)\eta _{x}^{t}(x, s)\mathrm{d}s\mathrm{d}x}&\le {{\varepsilon }_{2}}{{{\hat{b}}}^{2}}\!\int_{0}^{1}\!\!{\psi _{x}^{2}(x, t)\mathrm{d}x}\!+\!\frac{{G_{0}}}{4{{\varepsilon }_{2}}}\!\int_{0}^{1}\!\!{\int_{0}^{\infty }\!\!{{{g}^{{\delta_{1}}}}(s){{\left| \eta _{x}^{t}(x, s) \right|}^{2}}\mathrm{d}s\mathrm{d}x}},\\ 
		-\!\int_{0}^{1}\!\!{h(x, t)\!\int_{0}^{\infty }\!\!{g(s){{\eta }^{t}}(x, s)\mathrm{d}s\mathrm{d}x}}&\le {{\varepsilon }_{2}}{G_{0}}\!\int_{0}^{1}\!\!{\int_{0}^{\infty }\!\!{{{g}^{{\delta_{1}}}}(s){{\left| \eta _{x}^{t}(x, s) \right|}^{2}}\mathrm{d}s\mathrm{d}x}}\!+\!\frac{1}{4{{\varepsilon }_{2}}}\!\int_{0}^{1}\!\!{{{h}^{2}}(x, t)\mathrm{d}x},
	\end{align}
	and 
	\begin{align}
		k \int_{0}^{1}{({{\phi }_{x}(x, t)}+\psi(x, t) )\int_{0}^{\infty }{g(s){{\eta }^{t}}(x, s)\mathrm{d}s\mathrm{d}x}}&\le {{\varepsilon }_{2}}{{k}^{2}}\int_{0}^{1}{{{({{\phi }_{x}(x, t)}+\psi(x, t) )}^{2}}\mathrm{d}x}\nonumber\\
		&\quad +\frac{{G_{0}}}{4{{\varepsilon }_{2}}}\int_{0}^{1}{\int_{0}^{\infty }{{{g}^{{\delta_{1}}}}(s){{\left| \eta _{x}^{t}(x, s) \right|}^{2}}\mathrm{d}s\mathrm{d}x}}.
	\end{align}
	By Young's inequality and  H\"{o}lder's inequality, we deduce that
	\begin{align}
		{{\rho }_{2}}\int_{0}^{1}\!\!{{{\psi }_{t}}(x, t)\!\int_{0}^{\infty }\!\!{g(s)\eta _{s}^{t}(x, s)\mathrm{d}s\mathrm{d}x}}
		&=-{{\rho }_{2}}\int_{0}^{1}\!\!{{{\psi }_{t}}(x, t)\!\int_{0}^{\infty }\!\!{g'(s){{\eta }^{t}}(x, s)\mathrm{d}s\mathrm{d}x}}\nonumber\\
		&\le \frac{{{g}_{0}}{{\rho }_{2}}}{2}\!\int_{0}^{1}\!\!{\psi _{t}^{2}(x, t)\mathrm{d}x}+\frac{{{\rho }_{2}}}{2{{g}_{0}}}\!\int_{0}^{1}\!\!{\left(\int_{0}^{\infty }\!\!{g'(s)\mathrm{d}s}\right)\left(\int_{0}^{\infty }\!\!{g'(s){{\left| {{\eta }^{t}}(x, s) \right|}^{2}}\mathrm{d}s}\right)\mathrm{d}x}\nonumber\\
		& \le \frac{{{g}_{0}}{{\rho }_{2}}}{2}\!\int_{0}^{1}\!\!{\psi _{t}^{2}(x, t)\mathrm{d}x}-\frac{{{\rho }_{2}}g(0)}{2{{g}_{0}}}\!\int_{0}^{1}\!\!{\int_{0}^{\infty }\!\!{g'(s){{\left| \eta _{x}^{t}(x, s) \right|}^{2}}\mathrm{d}s\mathrm{d}x}},\quad\forall t \in {\mathbb{R}}_{>0}. \label{eq:5.7e}
	\end{align}
	Substituting the above estimates \eqref{eq:5.7a}-\eqref{eq:5.7e} into \eqref{eq:22}, we obtain \eqref{eq:21}.$\hfill\square$
\end{pf}
\par The following lemma provides a bound for the derivative of ${J}(t)$.
\begin{lemma}\label{lem:Lemma4-3} Let~$(\phi, \phi_t, \psi, \psi_t, \eta^t)$ be a solution to system \eqref{eq:sys}. Then, for any~${{\varepsilon }_{3}}\in {\mathbb{R}}_{>0}$, it holds that for all $t\in \mathbb{R}_{>0},$
	\begin{align}
		\frac{\mathrm{d}}{\mathrm{d}t}J(t)&\le \left[{{\phi }_{x}(x, t)}\left(\hat{b}{{\psi }_{x}(x, t)}+\int_{0}^{\infty }\!\!{g(s)\eta _{x}^{t}(x, s)\mathrm{d}s}\right)\right]_{x=0}^{x=1}-k\int_{0}^{1}\!\!{{{\left({{\phi }_{x}(x, t)}+\psi(x, t) \right)}^{2}}\mathrm{d}x}+{{\rho }_{2}}\int_{0}^{1}\!\!{\psi _{t}^{2}(x, t)\mathrm{d}x}\nonumber\\
		&\quad+{{\varepsilon }_{3}}\int_{0}^{1}\!\!{\phi _{t}^{2}(x, t)\mathrm{d}x}+{{\varepsilon }_{3}}\int_{0}^{1}\!\!{{{\left({{\phi }_{x}(x, t)}+\psi(x, t) \right)}^{2}}\mathrm{d}x}+\frac{{\hat{b}}}{k}{{\varepsilon }_{3}}\int_{0}^{1}\!\!{\psi _{x}^{2}(x, t)\mathrm{d}x}+\left(\frac{1}{4k{{\varepsilon }_{3}}}+\frac{{\hat{b}}}{4k{{\varepsilon }_{3}}}\right)\int_{0}^{1}\!\!{{{f}^{2}}(x, t)\mathrm{d}x}\nonumber\\
		&\quad+\frac{1}{4{\varepsilon }_{3}}\int_{0}^{1}\!\!{{{h}^{2}}(x, t)\mathrm{d}x}-g(0)\frac{\rho _{1}^{2}}{4{{\varepsilon }_{3}}{{k}^{2}}}\int_{0}^{1}\!\!{\int_{0}^{\infty }\!\!{g'(s){{\left| \eta _{x}^{t}(x, s) \right|}^{2}}\mathrm{d}s\mathrm{d}x}}+\frac{1}{k}{{\varepsilon }_{3}}{G_{0}}\int_{0}^{1}\!\!{\int_{0}^{\infty }\!\!{{{g}^{{\delta_{1}}}}(s){{\left| \eta _{x}^{t}(x, s) \right|}^{2}}\mathrm{d}s\mathrm{d}x}}, \label{eq:23}
	\end{align}
	where $G_{0}$ is defined as in Lemma \ref{lem:Lemma3-6}.
\end{lemma}
\begin{pf}Combining \eqref{eq:sys} and \eqref{eq:15}, for all $t\in \mathbb{R}_{>0}$, it follows from integration by parts that
	\begin{align}
		\frac{\mathrm{d}}{\mathrm{d}t}J(t)&=\!\int_{0}^{1}\!\!\!{\left({{\phi }_{x}(x, t)}\!+\!\psi(x, t) \right)\left(\hat{b}{{\psi }_{xx}(x, t)}\!+\!\int_{0}^{\infty }\!\!\!{g(s)\eta _{xx}^{t}(x, s)\mathrm{d}s}\right)}\mathrm{d}x \!+\!\frac{{{\rho }_{1}}\hat{b}}{k}\!\int_{0}^{1}\!\!\!{{{\psi }_{xt}(x, t)}{{\phi }_{t}(x, t)}\mathrm{d}x}\nonumber\\
		&\quad-\!\int_{0}^{1}\!\!\!{\left({{\phi }_{x}(x, t)}+\psi(x, t) \right)}\left({k\left({{\phi }_{x}(x, t)}\!+\!\psi(x, t) \right)\!-\! h(x, t)}\right)\mathrm{d}x \!+\!{{\rho }_{2}}\!\int_{0}^{1}\!\!\!{{{\psi }_{t}(x, t)}{{\left({{\phi }_{x}(x, t)}\!+\!\psi(x, t) \right)}_{t}}}\mathrm{d}x\nonumber\\
		&\quad+\!\frac{{\hat{b}}}{k}\!\int_{0}^{1}\!\!\!{{{\psi }_{x}(x, t)}(k{{({{\phi }_{x}(x, t)}\!+\!\psi(x, t) )}_{x}}\!+\! f(x, t))\mathrm{d}x}\!+\!\int_{0}^{1}\!\!\!{{{\left({{\phi }_{x}(x, t)}\!+\!\psi(x, t) \right)}_{x}}\!\int_{0}^{\infty }\!\!\!{g(s)\eta _{x}^{t}(x, s)\mathrm{d}s\mathrm{d}x}}\nonumber\\
		&\quad +\!\frac{{{\rho }_{1}}}{k}\!\int_{0}^{1}\!\!\!{{{\phi }_{t}}(x, t)\!\int_{0}^{\infty }\!\!\!{g(s)\eta _{xt}^{t}(x, s)\mathrm{d}s\mathrm{d}x}}\!+\!\frac{1}{k}\!\int_{0}^{1}\!\!\!{f(x, t)\!\int_{0}^{\infty }\!\!\!{g(s)\eta _{x}^{t}(x, s)\mathrm{d}s\mathrm{d}x}}.\label{eq:24}
	\end{align}
	We first simplify each term on the right-hand side of \eqref{eq:24}. For all~$t\in \mathbb{R}_{>0}$, we obtain
	\begin{align}
		\hat{b}\!\int_{0}^{1}\!\!\left({{\phi }_{x}(x, t)}\!+\!\psi(x, t) \right){{\psi }_{xx}(x, t)}\mathrm{d}x \!&=\!\left[\hat{b}{\phi }_{x}(x, t){\psi }_{x}(x, t)\right]_{x=0}^{x=1} \!-\!\hat{b}\!\int_{0}^{1}\!\!{\left({{\phi }_{x}(x, t)}\!+\!\psi(x, t)\right)_{x}\psi_{x}(x, t)}\mathrm{d}x,\label{eq:5.24a}\\
		\frac{{{\rho }_{1}}}{k}\!\int_{0}^{1}\!\!{{{\phi }_{t}}(x, t)\!\!\int_{0}^{\infty }\!\!\!{g(s)\eta _{xt}^{t}(x, s)\mathrm{d}s\mathrm{d}x}}\!&=\!\frac{{{\rho }_{1}}}{k}g_{0}\!\int_{0}^{1}\!\!{{\phi }_{t}(x, t)\psi_{xt}(x, t)}\mathrm{d}x \!+\!\frac{{{\rho }_{1}}}{k}\!\int_{0}^{1}\!\!{{\phi }_{t}(x, t)\!\!\int_{0}^{\infty }\!\!\! g'(s)\eta _{x}^{t}(x, s)\mathrm{d}s}\mathrm{d}x,\label{eq:5.25a}\\
		{{\rho }_{2}}\!\int_{0}^{1}\!\!{{{\psi }_{t}(x, t)}{{\left({{\phi }_{x}(x, t)}\!+\!\psi(x, t) \right)}_{t}}}\mathrm{d}x \! &=\!{\rho }_{2}\!\int_{0}^{1}\!\!{{\psi }_{t}^{2}(x, t)}\mathrm{d}x \!-\!{\rho }_{2}\!\int_{0}^{1}\!\!{{\psi }_{xt}(x, t){\phi }_{t}(x, t)}\mathrm{d}x,
	\end{align}
	and
	\begin{align}
		\int_{0}^{1}{\left({{\phi }_{x}(x, t)}+\psi(x, t) \right)\int_{0}^{\infty }{g(s)\eta _{xx}^{t}(x, s)\mathrm{d}s}}\mathrm{d}x&=-\int_{0}^{\infty }{g(s)\int_{0}^{1}\left({{\phi }_{x}(x, t)}+\psi(x, t)\right)_{x}\eta _{x}^{t}(x, s)\mathrm{d}x\mathrm{d}s}\nonumber\\
		&\quad +\left[{\phi }_{x}(x, t)\int_{0}^{\infty }{g(s)\eta _{x}^{t}(x, s)\mathrm{d}s}\right]_{x=0}^{x=1}.
	\end{align}
	By Assumption \ref{ass:Assumption3}, it follows that
	\begin{align}
		\left(\frac{{{\rho }_{1}}\hat{b}}{k}+\frac{{{\rho }_{1}}}{k}g_{0}-{\rho }_{2}\right)\int_{0}^{1}{{\psi }_{xt}(x, t){\phi }_{t}(x, t)}\mathrm{d}x=0,\quad \forall t\in \mathbb{R}_{>0}.\label{eq:5.24e}
	\end{align}
	Substituting the above simplifications \eqref{eq:5.24a}-\eqref{eq:5.24e} into \eqref{eq:24}, for all $t\in \mathbb{R}_{>0}$, we obtain
	\begin{align}
		\frac{\mathrm{d}}{\mathrm{d}t}J(t)&=\left[\hat{b}{\phi }_{x}(x, t){\psi }_{x}(x, t)\right]_{x=0}^{x=1}\!+\left[{\phi }_{x}(x, t)\!\int_{0}^{\infty }\!{g(s)\eta _{x}^{t}(x, s)\mathrm{d}s}\right]_{x=0}^{x=1}\!+{{\rho }_{2}}\!\int_{0}^{1}\!{{\psi }_{t}^{2}(x, t)}\mathrm{d}x-k \!\int_{0}^{1}\!{\left({{\phi }_{x}(x, t)}+\psi(x, t)\right)^{2}}\mathrm{d}x\nonumber\\
		&\quad+\frac{{{\rho }_{1}}}{k}\!\int_{0}^{1}\!{{\phi }_{t}(x, t)\!\int_{0}^{\infty }\! g'(s)\eta _{x}^{t}(x, s)\mathrm{d}s}\mathrm{d}x+\!\int_{0}^{1}\!{\left({{\phi }_{x}(x, t)}+\psi(x, t)\right)h(x, t)}\mathrm{d}x+\frac{1}{k}\!\int_{0}^{1}\!{f(x, t)\!\int_{0}^{\infty }\!{g(s)\eta _{x}^{t}(x, s)\mathrm{d}s\mathrm{d}x}}\nonumber\\
		&\quad +\frac{{\hat{b}}}{k}\!\int_{0}^{1}\!{{{\psi }_{x}(x, t)}f(x, t)}\mathrm{d}x.\label{eq:25}
	\end{align}
	Subsequently, we estimate each term on the right-hand side of \eqref{eq:25}. By Young's inequality, H\"{o}lder's inequality, and Lemma \ref{lem:Lemma3-6}, for any~${{\varepsilon }_{3}} \in {\mathbb{R}}_{>0}$ and $t \in {\mathbb{R}}_{>0}$, we obtain
	\begin{align*}
		\frac{{\hat{b}}}{k}\int_{0}^{1}{{{\psi }_{x}(x, t)}f(x, t)}\mathrm{d}x&\leq\frac{{\hat{b}}}{k}{{\varepsilon }_{3}}\int_{0}^{1}{{{\psi }_{x}}^{2}(x, t)}\mathrm{d}x+\frac{{\hat{b}}}{4k{\varepsilon }_{3}}\int_{0}^{1}{f^{2}(x, t)}\mathrm{d}x,\\
		\int_{0}^{1}{\left({{\phi }_{x}(x, t)}+\psi(x, t)\right)h(x, t)}\mathrm{d}x &\leq{\varepsilon }_{3}\int_{0}^{1}\left({{\phi }_{x}(x, t)}+\psi(x, t)\right)^{2}\mathrm{d}x +\frac{1}{4{\varepsilon }_{3}}\int_{0}^{1} h^{2}(x, t)\mathrm{d}x,\\
		\frac{1}{k}\int_{0}^{1}{f(x, t)\int_{0}^{\infty }{g(s)\eta _{x}^{t}(x, s)\mathrm{d}s\mathrm{d}x}}&\leq\frac{{\varepsilon }_{3}}{k}G_{0}\int_{0}^{1}{\int_{0}^{\infty }{{{g}^{{\delta_{1}}}}(s){{\left| \eta _{x}^{t}(x, s) \right|}^{2}}\mathrm{d}s\mathrm{d}x}}+\frac{1}{4k{\varepsilon }_{3}}\int_{0}^{1}{f^{2}(x, t)}\mathrm{d}x,
	\end{align*}
	and
	\begin{align*}
		\frac{{{\rho }_{1}}}{k}\int_{0}^{1}{{\phi }_{t}(x, t)\int_{0}^{\infty }g'(s)\eta _{x}^{t}(x, s)\mathrm{d}s}\mathrm{d}x&\leq-g(0)\frac{{\rho }_{1}^{2}}{4{\varepsilon }_{3}k^{2}}\int_{0}^{1}{\int_{0}^{\infty }g'(s){\left| \eta _{x}^{t}(x, s) \right|}^{2}\mathrm{d}s}\mathrm{d}x+{{\varepsilon }_{3}}\int_{0}^{1}{{\phi }_{t}^{2}(x, t)}\mathrm{d}x.
	\end{align*}
	Substituting the above estimates into \eqref{eq:25}, we obtain \eqref{eq:23}.$\hfill\square$
\end{pf}

\par The following lemma provides an estimate for the boundary terms appearing in inequality \eqref{eq:23}. To facilitate the handling of boundary terms, we introduce the linear function ~$q(x)=2-4x$ defined on~$[0,1]$, which takes values of equal magnitude and opposite signs at the two endpoints, providing a symmetric treatment of the boundary terms.
\begin{lemma}\label{lem:Lemma4-4} Let~$(\phi, \phi_t, \psi, \psi_t, \eta^t)$ be a solution to system \eqref{eq:sys}. Then, for any~${{\varepsilon }_{3}}$, ${{\lambda }_{2}}\in {\mathbb{R}}_{>0}$, it holds that for all $t\in \mathbb{R}_{>0}$,
	\begin{align}
		&\quad\left[{{\phi }_{x}(x, t)}\left(\hat{b}{{\psi }_{x}(x, t)}+\!\int_{0}^{\infty }\!{g(s)\eta _{x}^{t}}(x, s)\mathrm{d}s\right)\right]_{x=0}^{x=1}\nonumber\\
		&\le -\frac{{{\varepsilon }_{3}}}{k}\frac{\mathrm{d}}{\mathrm{d}t}\!\int_{0}^{1}\!{{{\rho }_{1}}q(x){{\phi }_{t}(x, t)}{{\phi }_{x}(x, t)}\mathrm{d}x}-\frac{g(0){{\rho }_{2}}}{4\varepsilon _{3}^{2}}\!\int_{0}^{1}\!{\int_{0}^{\infty }\!{g'(s){{\left| \eta _{x}^{t}(x, s) \right|}^{2}}\mathrm{d}s\mathrm{d}x}}+\frac{{{k}^{2}}{{\lambda }_{2}}}{{{\varepsilon }_{3}}}\!\int_{0}^{1}\!{{{({{\phi }_{x}(x, t)}+\psi(x, t) )}^{2}}\mathrm{d}x}\nonumber\\
		&\quad -\frac{1}{4{{\varepsilon }_{3}}}\frac{\mathrm{d}}{\mathrm{d}t}\!\int_{0}^{1}\!{{{\rho }_{2}}q(x){{\psi }_{t}(x, t)}\left(\hat{b}{{\psi }_{x}(x, t)}+\!\int_{0}^{\infty }\!{g(s)\eta _{x}^{t}}(x, s)\mathrm{d}s\right)\mathrm{d}x}+\left(\frac{{{\rho }_{2}}}{4}+\frac{{{\rho }_{2}}\left(\hat{b}+{{g}_{0}}\right)}{2{{\varepsilon }_{3}}}\right)\!\int_{0}^{1}\!{\psi _{t}^{2}(x, t)\mathrm{d}x}\nonumber\\
		&\quad +\frac{1}{{{\varepsilon }_{3}}}\left(\varepsilon _{3}^{2}+{{{\hat{b}}}^{2}}+\frac{{{{\hat{b}}}^{2}}}{8{{\lambda }_{2}}}+2{{{\hat{b}}}^{2}}{{\varepsilon }_{3}}\right)\!\int_{0}^{1}\!{\psi _{x}^{2}(x, t)\mathrm{d}x}+\frac{2{{\rho }_{1}}}{k}{{\varepsilon }_{3}}\!\int_{0}^{1}\!{\phi _{t}^{2}(x, t)\mathrm{d}x}+4{{\varepsilon }_{3}}\!\int_{0}^{1}\!{\phi _{x}^{2}(x, t)\mathrm{d}x}\nonumber\\
		&\quad+\left(\frac{{G_{0}}}{{{\varepsilon }_{3}}}+\frac{{G_{0}}}{8{{\varepsilon }_{3}}{{\lambda }_{2}}}+2{G_{0}}\right)\!\int_{0}^{1}\!{\int_{0}^{\infty }{{{g}^{{\delta_{1}}}}(s){{\left| \eta _{x}^{t}(x, s) \right|}^{2}}\mathrm{d}s\mathrm{d}x}}+\frac{1}{16{{\varepsilon }_{3}^{2}}}\!\int_{0}^{1}\!{{{h}^{2}}(x, t)\mathrm{d}x}+\frac{{{\varepsilon }_{3}}}{{{k}^{2}}}\!\int_{0}^{1}\!{{{f}^{2}}(x, t)\mathrm{d}x}, \label{eq:26}
	\end{align}
	where $G_{0}$ is defined as in Lemma \ref{lem:Lemma3-6}.
\end{lemma}
\begin{pf}By Young's inequality, for any~${{\varepsilon }_{3}}\in {\mathbb{R}}_{>0}$ and $t\in {\mathbb{R}}_{>0}$, we obtain:
	\begin{align}
		\left[{{\phi }_{x}(x, t)}\left(\hat{b}{{\psi }_{x}(x, t)}\!+\!\int_{0}^{\infty }\!\!\!{g(s)\eta _{x}^{t}}(x, s)\mathrm{d}s\right)\right]_{x=0}^{x=1}&\leq{{\varepsilon }_{3}}\left[{{\phi }_{x}^{2}(1, t)}\!+\!{{\phi }_{x}^{2}(0, t)}\right]\!+\!\frac{1}{4{{\varepsilon }_{3}}}{\left(\hat{b}{{\psi }_{x}(0, t)}\!+\!\int_{0}^{\infty }\!\!\!{g(s)\eta _{x}^{t}(0, s)}\mathrm{d}s\right)}^{2}\nonumber\\
		&\quad +\!\frac{1}{4{{\varepsilon }_{3}}}{\left(\hat{b}{{\psi }_{x}(1, t)}\!+\!\int_{0}^{\infty }\!\!\!{g(s)\eta _{x}^{t}(1, s)}\mathrm{d}s\right)}^{2}.\label{eq:32}
	\end{align}
	We define the following functionals:
	\begin{align}
		&J_1(t)=\int_{0}^{1}{{{\rho }_{2}}q(x){{\psi }_{t}(x, t)}\left(\hat{b}{{\psi }_{x}(x, t)}+\int_{0}^{\infty }{g(s)\eta _{x}^{t}}(x, s)\mathrm{d}s\right)\mathrm{d}x},\quad\forall t \in {\mathbb{R}}_{\geq0},\label{eq:27}\\
		&J_2(t)=\int_{0}^{1}{{{\rho }_{1}}q(x){{\phi }_{t}(x, t)}{{\phi }_{x}(x, t)}\mathrm{d}x},\quad\forall t \in {\mathbb{R}}_{\geq0}.\label{eq:28}
	\end{align}
	Combining \eqref{eq:sys} and \eqref{eq:27}, for all $t\in \mathbb{R}_{>0}$, it follows from integration by parts that
	\begin{align*}
		\frac{\mathrm{d}}{\mathrm{d}t}J_1(t)
		&=\int_{0}^{1}{q(x)\left(\hat{b}{{\psi }_{xx}(x, t)}+\int_{0}^{\infty }{g(s)\eta _{xx}^{t}}(x, s)\mathrm{d}s\right)\left(\hat{b}{{\psi }_{x}(x, t)}+\int_{0}^{\infty }{g(s)\eta _{x}^{t}}(x, s)\mathrm{d}s\right)\mathrm{d}x}\\
		&\quad-\int_{0}^{1}{q(x)\left(k\left({{\phi }_{x}(x, t)}+\psi(x, t) \right)- h(x, t)\right)\left(\hat{b}{{\psi }_{x}(x, t)}+\int_{0}^{\infty }{g(s)\eta _{x}^{t}}(x, s)\mathrm{d}s\right)\mathrm{d}x}\\
		&\quad+\int_{0}^{1}{{{\rho }_{2}}q(x){{\psi }_{t}(x, t)}\left(\hat{b}{{\psi }_{xt}(x, t)}+\int_{0}^{\infty }{g(s)\eta _{xt}^{t}}(x, s)\mathrm{d}s\right)\mathrm{d}x} \\ 
		& =\left[\frac{q(x)}{2}{{\left(\hat{b}{{\psi }_{x}(x, t)}+\int_{0}^{\infty }{g(s)\eta _{x}^{t}}(x, s)\mathrm{d}s\right)^{2}}}\right]_{x=0}^{x=1}+{{\rho }_{2}}\int_{0}^{1}{q(x){{\psi }_{t}(x, t)}\int_{0}^{\infty }{g'(s)\eta _{x}^{t}(x, s)\mathrm{d}s\mathrm{d}x}}\\
		&\quad+2\int_{0}^{1}{{{\left(\hat{b}{{\psi }_{x}(x, t)}+\int_{0}^{\infty }{g(s)\eta _{x}^{t}}(x, s)\mathrm{d}s\right)^{2}}}\mathrm{d}x}+2{{\rho }_{2}}\left(\hat{b}+{{g}_{0}}\right)\int_{0}^{1}{\psi _{t}^{2}(x, t)\mathrm{d}x}\\
		&\quad+\int_{0}^{1}{q(x)\left(h(x, t)- k\left({{\phi }_{x}(x, t)}+\psi(x, t) \right)\right)\left(\hat{b}{{\psi }_{x}(x, t)}+\int_{0}^{\infty }{g(s)\eta _{x}^{t}}(x, s)\mathrm{d}s\right)\mathrm{d}x}. 
	\end{align*}
	Then, for all~$t\in \mathbb{R}_{>0}$, it follows that
	\begin{align*}
		&\quad{\left(\hat{b}{{\psi }_{x}(0, t)}+\!\int_{0}^{\infty }\!\!{g(s)\eta _{x}^{t}(0, s)}\mathrm{d}s\right)}^{2}+{\left(\hat{b}{{\psi }_{x}(1, t)}+\!\int_{0}^{\infty }\!\!{g(s)\eta _{x}^{t}(1, s)}\mathrm{d}s\right)}^{2}\nonumber\\
		&=-\frac{\mathrm{d}}{\mathrm{d}t}\!\int_{0}^{1}\!\!{{{\rho }_{2}}q(x){{\psi }_{t}(x, t)}\left(\hat{b}{{\psi }_{x}(x, t)}+\!\int_{0}^{\infty }\!\!{g(s)\eta _{x}^{t}}(x, s)\mathrm{d}s\right)\mathrm{d}x}+2\!\int_{0}^{1}\!\!{{{\left(\hat{b}{{\psi }_{x}(x, t)}+\!\int_{0}^{\infty }\!\!{g(s)\eta _{x}^{t}}(x, s)\mathrm{d}s\right)^{2}}}\mathrm{d}x}\nonumber\\
		&\quad+\!\int_{0}^{1}\!\!{q(x)\left(h(x, t)-k\left({{\phi }_{x}(x, t)}+\psi(x, t) \right)\right)\left(\hat{b}{{\psi }_{x}(x, t)}+\!\int_{0}^{\infty }\!\!{g(s)\eta _{x}^{t}}(x, s)\mathrm{d}s\right)\mathrm{d}x}+2{{\rho }_{2}}\left(\hat{b}+{{g}_{0}}\right)\!\int_{0}^{1}\!\!{\psi _{t}^{2}(x, t)\mathrm{d}x}\nonumber\\
		&\quad+{{\rho }_{2}}\!\int_{0}^{1}\!\!{q(x){{\psi }_{t}(x, t)}\!\int_{0}^{\infty }\!\!{g'(s)\eta _{x}^{t}(x, s)\mathrm{d}s\mathrm{d}x}}.
	\end{align*}
	Next, we estimate each term on the right-hand side of the above equation. By Young's inequality and Lemma \ref{lem:Lemma3-6}, for any~${{\varepsilon }_{3}}, {{\lambda }_{2}}\in {\mathbb{R}}_{>0}$ and $t\in \mathbb{R}_{>0}$,  we obtain
	\begin{align*}
		2\int_{0}^{1}{{{\left(\hat{b}{{\psi }_{x}(x, t)}+\int_{0}^{\infty }{g(s)\eta _{x}^{t}}(x, s)\mathrm{d}s\right)^{2}}}\mathrm{d}x}&\leq 4G_{0}\int_{0}^{1}{\int_{0}^{\infty }{g^{\delta_{1}}(s){{\left| \eta _{x}^{t}(x, s) \right|}^{2}}}\mathrm{d}s}\mathrm{d}x+4\int_{0}^{1}{\hat{b}^{2}{\psi }_{x}^{2}(x, t)}\mathrm{d}x,
	\end{align*}
	\begin{align*}
		{{\rho }_{2}}\int_{0}^{1}{q(x){{\psi }_{t}(x, t)}\int_{0}^{\infty }{g'(s)\eta _{x}^{t}(x, s)\mathrm{d}s\mathrm{d}x}}&\leq {{\rho }_{2}}{{\varepsilon }_{3}}\int_{0}^{1}{{{\psi }_{t}}^{2}(x, t)}\mathrm{d}x-g(0)\frac{{\rho }_{2}}{{\varepsilon }_{3}}\int_{0}^{1}{\int_{0}^{\infty }{g'(s){{\left| \eta _{x}^{t}(x, s) \right|}^{2}}}\mathrm{d}s}\mathrm{d}x,
	\end{align*}
	\begin{align*}
		&\quad -k\int_{0}^{1}{q(x)\left({{\phi }_{x}(x, t)}+\psi(x, t) \right)\left(\hat{b}{{\psi }_{x}(x, t)}+\int_{0}^{\infty }{g(s)\eta _{x}^{t}}(x, s)\mathrm{d}s\right)\mathrm{d}x}\\
		&\leq 4k^{2}{{\lambda }_{2}}\int_{0}^{1}{\left({{\phi }_{x}(x, t)}+\psi(x, t) \right)^{2}}\mathrm{d}x +\frac{G_{0}}{2{{\lambda }_{2}}}\int_{0}^{1}{\int_{0}^{\infty }{g^{\delta_{1}}(s){{\left| \eta _{x}^{t}(x, s) \right|}^{2}}}\mathrm{d}s}\mathrm{d}x+\frac{\hat{b}^{2}}{2{{\lambda }_{2}}}\int_{0}^{1}{{{\psi }_{x}}^{2}(x, t)}\mathrm{d}x,
	\end{align*}
	and
	\begin{align*}
		\int_{0}^{1}{q(x)h(x, t)\left(\hat{b}{{\psi }_{x}(x, t)}+\int_{0}^{\infty }{g(s)\eta _{x}^{t}}(x, s)\mathrm{d}s\right)\mathrm{d}x}&\leq 8{\hat{b}}^{2}{{\varepsilon }_{3}}\int_{0}^{1}{{{\psi }_{x}}^{2}(x, t)}\mathrm{d}x+\frac{1}{4{{\varepsilon }_{3}}}\int_{0}^{1}{h^{2}(x, t)}\mathrm{d}x\\
		&\quad +8{{\varepsilon }_{3}}{G_{0}}\int_{0}^{1}{\int_{0}^{\infty }{g^{\delta_{1}}(s){{\left| \eta _{x}^{t}(x, s) \right|}^{2}}}\mathrm{d}s}\mathrm{d}x.
	\end{align*}
	Hence, for all~$t\in \mathbb{R}_{>0}$, we conclude that
	\begin{align}
		&\quad{\left(\hat{b}{{\psi }_{x}(0, t)}+\int_{0}^{\infty }{g(s)\eta _{x}^{t}(0, s)}\mathrm{d}s\right)}^{2}+{\left(\hat{b}{{\psi }_{x}(1, t)}+\int_{0}^{\infty }{g(s)\eta _{x}^{t}(1, s)}\mathrm{d}s\right)}^{2}\nonumber\\
		&\leq -\frac{\mathrm{d}}{\mathrm{d}t}\int_{0}^{1}{{{\rho }_{2}}q(x){{\psi }_{t}(x, t)}\left(\hat{b}{{\psi }_{x}(x, t)}+\int_{0}^{\infty }{g(s)\eta _{x}^{t}}(x, s)\mathrm{d}s\right)\mathrm{d}x}+\left({{\rho }_{2}}{{\varepsilon }_{3}}+2{{\rho }_{2}}\left(\hat{b}+g_{0}\right)\right)\int_{0}^{1}{{{\psi }_{t}}^{2}(x, t)}\mathrm{d}x\nonumber\\
		&\quad +\frac{1}{4{{\varepsilon }_{3}}}\int_{0}^{1}{h^{2}(x, t)}\mathrm{d}x+\left(4\hat{b}^{2}+\frac{\hat{b}^{2}}{2{{\lambda }_{2}}}+8{\hat{b}}^{2}{{\varepsilon }_{3}}\right)\int_{0}^{1}{{{\psi }_{x}}^{2}(x, t)}\mathrm{d}x+4k^{2}{{\lambda }_{2}}\int_{0}^{1}{\left({{\phi }_{x}(x, t)}+\psi(x, t) \right)^{2}}\mathrm{d}x\nonumber\\
		&\quad -g(0)\frac{{\rho }_{2}}{{\varepsilon }_{3}}\int_{0}^{1}{\int_{0}^{\infty }{g'(s){{\left| \eta _{x}^{t}(x, s) \right|}^{2}}}\mathrm{d}s}\mathrm{d}x+\left(4G_{0}+\frac{G_{0}}{2{{\lambda }_{2}}}+8{{\varepsilon }_{3}}{G_{0}}\right)\int_{0}^{1}{\int_{0}^{\infty }{g^{\delta_{1}}(s){{\left| \eta _{x}^{t}(x, s) \right|}^{2}}}\mathrm{d}s}\mathrm{d}x.\label{eq:30}
	\end{align}
	Similarly, combining \eqref{eq:sys}, \eqref{eq:28}, and Young's inequality, for all~$t\in \mathbb{R}_{>0}$, we obtain
	\begin{equation*}
		\begin{split}
			\frac{\mathrm{d}}{\mathrm{d}t}J_{2}(t)
			& \le \frac{1}{2}\left[kq(x)\phi _{x}^{2}(x, t)\right]_{x=0}^{x=1}+2k\int_{0}^{1}{\phi _{x}^{2}(x, t)\mathrm{d}x}+2\int_{0}^{1}{k\left| {{\phi }_{x}(x, t)}{{\psi }_{x}(x, t)} \right|\mathrm{d}x}+2{{\rho }_{1}}\int_{0}^{1}{\phi _{t}^{2}(x, t)\mathrm{d}x}\\
			&\quad+2\int_{0}^{1}{\left| {{\phi }_{x}(x, t)}f(x, t) \right|\mathrm{d}x}\\
			&\leq-k\left(\phi _{x}^{2}(1, t)+\phi _{x}^{2}(0, t)\right)+4k\int_{0}^{1}{\phi _{x}^{2}(x, t)\mathrm{d}x}+k\int_{0}^{1}{{\psi }_{x}^{2}(x, t)}\mathrm{d}x+2{{\rho }_{1}}\int_{0}^{1}{\phi _{t}^{2}(x, t)\mathrm{d}x}+\frac{1}{k}\int_{0}^{1}{f^{2}(x, t)}\mathrm{d}x,
		\end{split}
	\end{equation*}
	which implies that
	\begin{align}
		\left(\phi _{x}^{2}(1, t)+\phi _{x}^{2}(0, t)\right)&\leq -\frac{1}{k}\frac{\mathrm{d}}{\mathrm{d}t}\int_{0}^{1}{{{\rho }_{1}}q(x){{\phi }_{t}(x, t)}{{\phi }_{x}(x, t)}\mathrm{d}x}+4\int_{0}^{1}{\phi _{x}^{2}(x, t)\mathrm{d}x}+\int_{0}^{1}{{\psi }_{x}^{2}(x, t)}\mathrm{d}x+2\frac{{{\rho }_{1}}}{k}\int_{0}^{1}{\phi _{t}^{2}(x, t)\mathrm{d}x}\nonumber\\
		&\quad +\frac{1}{k^{2}}\int_{0}^{1}{f^{2}(x, t)}\mathrm{d}x, \quad \forall t\in \mathbb{R}_{>0}.\label{eq:31}
	\end{align}
	Combining \eqref{eq:32}, \eqref{eq:30}, and \eqref{eq:31}, we obtain \eqref{eq:26}.$\hfill\square$
\end{pf}
\par The following lemma provides a bound for the derivative of ${K}(t)$.
\begin{lemma}\label{lem:Lemma4-5} Let~$(\phi, \phi_t, \psi, \psi_t, \eta^t)$ be a solution to system \eqref{eq:sys}. Then, for any~${{\varepsilon }_{3}}\in {\mathbb{R}}_{>0}$, it holds that for all $t\in \mathbb{R}_{>0}$,
	\begin{align}
		\frac{\mathrm{d}}{\mathrm{d}t}K(t)&\le -{{\rho }_{1}}\int_{0}^{1}{\phi _{t}^{2}(x, t)\mathrm{d}x}-{{\rho }_{2}}\int_{0}^{1}{\psi _{t}^{2}(x, t)\mathrm{d}x}+\left(\hat{b}+2{{\varepsilon }_{3}}\right)\int_{0}^{1}{\psi _{x}^{2}(x, t)\mathrm{d}x}+k\int_{0}^{1}{{{\left({{\phi }_{x}(x, t)}+\psi(x, t) \right)}^{2}}\mathrm{d}x}\nonumber\\
		&\quad+\frac{{G_{0}}}{4{{\varepsilon }_{3}}}\int_{0}^{1}{\int_{0}^{\infty }{{{g}^{{\delta_{1}}}}(s){{\left| \eta _{x}^{t}(x, s) \right|}^{2}}\mathrm{d}s\mathrm{d}x}}+{{\varepsilon }_{3}}\int_{0}^{1}{\phi _{x}^{2}(x, t)\mathrm{d}x}+\frac{1}{4{{\varepsilon }_{3}}}\int_{0}^{1}{\left({{f}^{2}}(x, t)+{{h}^{2}}(x, t)\right)\mathrm{d}x},\label{eq:33}
	\end{align}
	where $G_{0}$ is defined as in Lemma\ref{lem:Lemma3-6}.
\end{lemma}
\begin{pf}Combining \eqref{eq:sys} and \eqref{eq:16}, we obtain via integration by parts that
	\begin{align}
		\frac{\mathrm{d}}{\mathrm{d}t}K(t)& =-\!{{\rho }_{1}}\!\int_{0}^{1}\!{\phi _{t}^{2}(x, t)\mathrm{d}x}\!-\!{{\rho }_{2}}\!\int_{0}^{1}\!{\psi _{t}^{2}(x, t)\mathrm{d}x}\!+\! k \!\int_{0}^{1}\!{{{\left({{\phi }_{x}(x, t)}\!+\!\psi(x, t) \right)}^{2}}\mathrm{d}x}+\!\hat{b}\!\int_{0}^{1}\!{\psi _{x}^{2}(x, t)\mathrm{d}x}\!\nonumber\\
		&\quad+\!\int_{0}^{1}\!{{{\psi }_{x}(x, t)}\!\int_{0}^{\infty }\!{g(s)\eta _{x}^{t}(x, s)\mathrm{d}s\mathrm{d}x}}\!-\!\int_{0}^{1}\!{f(x, t)\phi(x, t) \mathrm{d}x}-\!\int_{0}^{1}\!{h(x, t)\psi(x, t) \mathrm{d}x},\quad \forall t\in \mathbb{R}_{>0}. \label{eq:34}		
	\end{align}
	By Young's inequality, Friedrichs' inequality, and Lemma \ref{lem:Lemma3-6}, for any ${{\varepsilon }_{3}}\in {\mathbb{R}}_{>0}$ and  $t\in {\mathbb{R}}_{\geq0}$, we obtain
	\begin{align*}
		\int_{0}^{1}{{{\psi }_{x}(x, t)}\int_{0}^{\infty }{g(s)\eta _{x}^{t}(x, s)\mathrm{d}s\mathrm{d}x}}&\leq\frac{G_{0}}{4{{\varepsilon }_{3}}}\int_{0}^{1}{\int_{0}^{\infty }{{{g}^{{\delta_{1}}}}(s){{\left| \eta _{x}^{t}(x, s) \right|}^{2}}\mathrm{d}s\mathrm{d}x}}+{{\varepsilon }_{3}}\int_{0}^{1}{\psi _{x}^{2}(x, t)\mathrm{d}x},\\
		-\int_{0}^{1}{f(x, t)\phi(x, t) \mathrm{d}x}&\leq{{\varepsilon }_{3}}\int_{0}^{1}{\phi _{x}^{2}(x, t)\mathrm{d}x}+\frac{1}{4{{\varepsilon }_{3}}}\int_{0}^{1}{f^{2}(x, t)}\mathrm{d}x,\\
		-\int_{0}^{1}{h(x, t)\psi(x, t) \mathrm{d}x}&\leq{{\varepsilon }_{3}}\int_{0}^{1}{\psi _{x}^{2}(x, t)\mathrm{d}x}+\frac{1}{4{{\varepsilon }_{3}}}\int_{0}^{1}{h^{2}(x, t)}\mathrm{d}x.
	\end{align*}
	Substituting the above estimates into \eqref{eq:34}, we obtain \eqref{eq:33}.$\hfill\square$
\end{pf}
\section{iISS assessment}\label{sec:6}This section is mainly devoted to the iISS assessment of system \eqref{eq:sys}. In the first subsection, we construct and estimate the Lyapunov functional based on the auxiliary functionals and related estimates established in Section \ref{sec:5}. In the second subsection, using the Lyapunov method, we perform the stability analysis of system \eqref{eq:sys} within the iISS framework and derive PiISS and EiISS estimates.
\par Before proceeding with the stability analysis, we first define the energy functional of system \eqref{eq:sys} as follows:
\begin{align}
	E(t)&=\frac{1}{2}\int_{0}^{1}\left(\rho_1\phi_{t}^{2}(x, t)+\rho_2\psi_{t}^{2}(x, t)+k(\phi_{x}(x, t)+\psi(x, t))^{2}+\hat{b}\psi_{x}^{2}(x, t)\right)\mathrm{d}x\nonumber\\
	&\quad+\frac{1}{2}\int_{0}^{1}\int_{0}^{\infty}g(s)|\eta_{x}^{t}(x, s)|^{2}\mathrm{d}s\mathrm{d}x,\quad\forall t\in {\mathbb{R}}_{\geq0}. \label{eq:5}
\end{align}
In addition, in the subsequent derivations, we let
\begin{align*}
	&e_{1}(t)=\int_{0}^{1}{\left(\phi _{t}^{2}(x, t)+\psi _{t}^{2}(x, t)+{{({{\phi }_{x}(x, t)}+\psi(x, t) )}^{2}}+\psi _{x}^{2}(x, t)\right)}\mathrm{d}x,\quad\forall t\in {\mathbb{R}}_{\geq0},
\end{align*}
and
\begin{align*}
	&e_{2}(t)=\int_{0}^{1}{\int_{0}^{\infty }{g(s){{\left| \eta _{x}^{t}(x, s) \right|}^{2}}}}\mathrm{d}s\mathrm{d}x,\quad\forall t\in {\mathbb{R}}_{\geq0}.
\end{align*}
\subsection{Construction of the Lyapunov functional}\label{subsec:6}In this subsection, combining the auxiliary functionals and their related estimates in Section \ref{sec:5}, we construct the Lyapunov functional $L(t)$. Afterwards, we estimate this Lyapunov functional and analyze the equivalence between $L(t)$ and $E(t)$.
Specifically, let
\begin{align}
	L(t)&=NE(t)+{{N}_{1}}{{I}_{1}}(t)+{{N}_{2}}{{I}_{2}}(t)+J(t)+\frac{{{\varepsilon }_{3}}}{k}\int_{0}^{1}{{{\rho }_{1}}q(x){{\phi }_{t}(x, t)}{{\phi }_{x}(x, t)}\mathrm{d}x}+\mu K(t)\nonumber\\
	&\quad+\frac{1}{4{{\varepsilon }_{3}}}\int_{0}^{1}{{{\rho }_{2}}q(x){{\psi }_{t}(x, t)}\left(\hat{b}{{\psi }_{x}(x, t)}+\int_{0}^{\infty }{g(s)\eta _{x}^{t}(x, s)\mathrm{d}s}\right)\mathrm{d}x},\quad \forall t\in \mathbb{R}_{\geq0},\label{eq:35}
\end{align}
where~$N, {{N}_{1}}, {{N}_{2}}, \mu$, and ${{\varepsilon }_{3}}$~are positive constants to be determined, and the auxiliary functionals~${{I}_{1}}(t), {{I}_{2}}(t), J(t)$, and $K(t)$ are defined by \eqref{eq:13}, \eqref{eq:14}, \eqref{eq:15}, and \eqref{eq:16}, respectively. 
\par To establish the iISS estimate for system \eqref{eq:sys}, we first estimate the derivative of $L(t)$, which is formalized in the proposition below.
\begin{proposition}
	There exist positive constants ${\sigma }_{1}, C_3$, and $C_4$ such that the Lyapunov functional $L(t)$ satisfies the following estimate:	
	\begin{align}
		\frac{\mathrm{d}}{\mathrm{d}t}L(t)&\le -{{\sigma }_{1}}\left(e_{1}(t)\!+\!\int_{0}^{1}\!\!{\int_{0}^{\infty }\!\!{{{g}^{{\delta_{1}}}}(s){{\left| \eta _{x}^{t}(x, s) \right|}^{2}}}}\mathrm{d}s\mathrm{d}x\right) \!+\!{C_3}\!\int_{0}^{1}\!\!{{{f}^{2}}(x, t)\mathrm{d}x}\!+\!{C_4}\!\int_{0}^{1}\!\!{{{h}^{2}}(x, t)\mathrm{d}x},\quad \forall t\in \mathbb{R}_{>0}. \label{eq:39}
	\end{align}
\end{proposition}
\begin{pf}
	By Young's inequality, for any~${{\varepsilon }_{4}}\in {\mathbb{R}}_{>0}$, we obtain that for all $t\in \mathbb{R}_{>0}$,
	\begin{align}
		\frac{\mathrm{d}}{\mathrm{d}t}E(t)&=\frac{1}{2}{{\int_{0}^{1}{\int_{0}^{\infty }{{g}'(s)\left| \eta _{x}^{t}(x, s) \right|^{2}}}}}\mathrm{d}s\mathrm{d}x+\int_{0}^{1}{f(x, t){{\phi }_{t}(x, t)}}\mathrm{d}x+\int_{0}^{1}{h(x, t){{\psi }_{t}(x, t)}}\mathrm{d}x\label{eq:36a}\\
		&\leq\frac{{\varepsilon }_{4}}{2}\int_{0}^{1}\left({{\phi }_{t}^{2}(x, t)}+{{\psi }_{t}}^{2}(x, t)\right)\mathrm{d}x +\frac{1}{2{\varepsilon }_{4}}\int_{0}^{1}\left({h^{2}(x, t)+f^{2}(x, t)}\right)\mathrm{d}x.\label{eq:36}
	\end{align}
	By Young's inequality and Friedrichs' inequality, for any~$t\in {\mathbb{R}}_{\geq0}$, we obtain
	\begin{align}
		\int_{0}^{1}{\phi _{x}^{2}(x, t)\mathrm{d}x}&=\int_{0}^{1}{{{\left({{\phi }_{x}(x, t)}+\psi(x, t) -\psi(x, t) \right)}^{2}}\mathrm{d}x}\nonumber\\
		&\le 2\int_{0}^{1}{{{\left({{\phi }_{x}(x, t)}+\psi(x, t) \right)}^{2}}\mathrm{d}x}+2\int_{0}^{1}{\psi _{x}^{2}(x, t)}\mathrm{d}x.\label{eq:37}
	\end{align}
	According to the construction of the Lyapunov functional~$L(t)$, \eqref{eq:17}, \eqref{eq:21}, \eqref{eq:23}, \eqref{eq:26}, and \eqref{eq:33}, we deduce that for all $t\in \mathbb{R}_{>0}$,
	\begin{align}
		\frac{\mathrm{d}}{\mathrm{d}t}L(t)&\le -{{\Lambda }_{1}}\int_{0}^{1}{\psi _{x}^{2}(x, t)}\mathrm{d}x-{{\Lambda }_{2}}\int_{0}^{1}{\phi _{t}^{2}(x, t)}\mathrm{d}x-{{\Lambda }_{3}}\int_{0}^{1}{\psi _{t}^{2}(x, t)}\mathrm{d}x-{{\Lambda }_{4}}\int_{0}^{1}{{{\left({{\phi }_{x}(x, t)}+\psi(x, t) \right)}^{2}}}\mathrm{d}x\nonumber\\
		&\quad+\left(\frac{N}{2}-{{C}_{1}}\right)\int_{0}^{1}{\int_{0}^{\infty }{g'(s){{\left| \eta _{x}^{t}(x, s) \right|}^{2}}\mathrm{d}s\mathrm{d}x}}+{{C}_{2}}\int_{0}^{1}{\int_{0}^{\infty }{{{g}^{{\delta_{1}}}}(s){{\left| \eta _{x}^{t}(x, s) \right|}^{2}}\mathrm{d}s\mathrm{d}x}}\nonumber\\
		&\quad +{C_3}\int_{0}^{1}{{{f}^{2}}(x, t)}\mathrm{d}x+{C_4}\int_{0}^{1}{{{h}^{2}}(x, t)}\mathrm{d}x,\label{eq:38}
	\end{align}
	where
	\begin{align*}
		{{\Lambda }_{1}}&=-\left(\left(-\hat{b}+3{{\lambda }_{1}}\right){{N}_{1}}+{{\varepsilon }_{2}}{{{\hat{b}}}^{2}}{{N}_{2}}+\frac{{\hat{b}}}{k}{{\varepsilon }_{3}}+\frac{1}{{{\varepsilon }_{3}}}\left(\varepsilon _{3}^{2}+{{{\hat{b}}}^{2}}+\frac{{{{\hat{b}}}^{2}}}{8{{\lambda }_{2}}}+2{{{\hat{b}}}^{2}}{{\varepsilon }_{3}}\right)+\mu \left(\hat{b}+2{{\varepsilon }_{3}}\right)+2\left(4+\mu \right){{\varepsilon }_{3}}\right), \\ 
		{{\Lambda }_{2}}&=-\left(\frac{N}{2}{{\varepsilon }_{4}}+{{N}_{1}}{{\varepsilon }_{1}}{{\rho }_{1}}+{{\varepsilon }_{3}}+\frac{2{{\rho }_{1}}}{k}{{\varepsilon }_{3}}-\mu {{\rho }_{1}}\right), \\ 
		{{\Lambda }_{3}}&=-\left(\frac{N}{2}{{\varepsilon }_{4}}+{{N}_{1}}\left({{\rho }_{2}}+\frac{{{\rho }_{1}}}{4{{\varepsilon }_{1}}}\right)-\frac{{{N}_{2}}}{2}{{\rho }_{2}}{{g}_{0}}+{{\rho }_{2}}+\frac{{{\rho }_{2}}}{4}+\frac{{{\rho }_{2}}\left(\hat{b}+{{g}_{0}}\right)}{2{{\varepsilon }_{3}}}-\mu {{\rho }_{2}}\right), \\ 
		{{\Lambda }_{4}}&=-\left({{N}_{2}}{{\varepsilon }_{2}}{{k}^{2}}-k+{{\varepsilon }_{3}}+\frac{{{k}^{2}}{{\lambda }_{2}}}{{{\varepsilon }_{3}}}+2\mu k+2(4+\mu ){{\varepsilon }_{3}}\right),\\
		{{C}_{1}}&={{N}_{2}}\frac{{{\rho }_{2}}g(0)}{2{{g}_{0}}}+g(0)\frac{\rho _{1}^{2}}{4{{\varepsilon }_{3}}{{k}^{2}}}+\frac{g(0){{\rho }_{2}}}{4\varepsilon _{3}^{2}}, \\ 
		{{C}_{2}}&={{N}_{1}}\frac{{G_{0}}}{4{{\lambda }_{1}}}+{{N}_{2}}{G_{0}}\left(1+\frac{1}{2{{\varepsilon }_{2}}}+{{\varepsilon }_{2}}\right)+\frac{{{\varepsilon }_{3}}}{k}{G_{0}}+{G_{0}}\left(\frac{1}{{{\varepsilon }_{3}}}+\frac{1}{8{{\lambda }_{2}}{{\varepsilon }_{3}}}+2\right)+\frac{\mu }{4{{\varepsilon }_{3}}}{G_{0}},\\
		{C_3}&=\frac{N}{2{{\varepsilon }_{4}}}+\frac{{{N}_{1}}}{4{{\lambda }_{1}}}+\frac{1}{4k{{\varepsilon }_{3}}}+\frac{{\hat{b}}}{4k{{\varepsilon }_{3}}}+\frac{{{\varepsilon }_{3}}}{{{k}^{2}}}+\frac{\mu }{4{{\varepsilon }_{3}}}, \\ 
		{C_4}&=\frac{N}{2{{\varepsilon }_{4}}}+\frac{{{N}_{1}}}{4{{\lambda }_{1}}}+\frac{{{N}_{2}}}{4{{\varepsilon }_{2}}}+\frac{1}{4{{\varepsilon }_{3}}}+\frac{1}{16{{\varepsilon }_{3}^{2}}}+\frac{\mu }{4{{\varepsilon }_{3}}}.  
	\end{align*}
	\par At this point, we determine the values of the parameters $N, N_{i}, {{\lambda }_{i}}(i=1, 2), {{\varepsilon }_{j}}(j=1, 2, 3, 4)$, and $\mu$.
	First, we fix the parameters
	\begin{equation*}
		{{\lambda }_{1}}=\frac{{\hat{b}}}{6},\quad {{\lambda }_{2}}=\varepsilon _{3}^{2},\quad \mu =\frac{1}{4},\quad {\varepsilon }_{4}=\frac{\rho _{1}}{12N}.
	\end{equation*}
	Then, we choose
	\begin{equation*}
		{{\varepsilon }_{3}}<\min \left\{\frac{k}{4\left({{k}^{2}}+\frac{19}{2}\right)}, \frac{{{\rho }_{1}}k}{12\left(k+2{{\rho }_{1}}\right)}\right\}.
	\end{equation*}
	When~${{\varepsilon }_{3}}$ and~$\mu$ are fixed,~${{N}_{1}}$ can be chosen sufficiently large such that
	\begin{equation*}
		\frac{{\hat{b}}}{4}{{N}_{1}}>\frac{1}{{{\varepsilon }_{3}}}\left(\frac{{\hat{b}}}{k}\varepsilon _{3}^{2}+\varepsilon _{3}^{2}+{{{\hat{b}}}^{2}}+\frac{{{{\hat{b}}}^{2}}}{8{{\lambda }_{2}}}+2{{{\hat{b}}}^{2}}{{\varepsilon }_{3}}\right)+\mu \left(\hat{b}+2{{\varepsilon }_{3}}\right)+2\left(4+\mu \right){{\varepsilon }_{3}}.
	\end{equation*} 
	After that, we choose~${\varepsilon }_{1}\leq\frac{1}{24{{N}_{1}}}$. Since~${\varepsilon }_{4}=\frac{\rho _{1}}{12N}$, we select~${{{N}_{2}}}$ sufficiently large such that
	\begin{equation*}
		\begin{split}
			\frac{{{N}_{2}}}{2}{{\rho }_{2}}{{g}_{0}}&>\frac{N}{2}{{\varepsilon }_{4}}+{{N}_{1}}\left({{\rho }_{2}}+\frac{{{\rho }_{1}}}{4{{\varepsilon }_{1}}}\right)+{{\rho }_{2}}+\frac{{{\rho }_{2}}}{4}+\frac{{{\rho }_{2}}\left(\hat{b}+{{g}_{0}}\right)}{2{{\varepsilon }_{3}}}-\mu {{\rho }_{2}}\\
			&=\frac{{\rho }_{1}}{24}+{{N}_{1}}\left({{\rho }_{2}}+\frac{{{\rho }_{1}}}{4{{\varepsilon }_{1}}}\right)+{{\rho }_{2}}+\frac{{{\rho }_{2}}}{4}+\frac{{{\rho }_{2}}\left(\hat{b}+{{g}_{0}}\right)}{2{{\varepsilon }_{3}}}-\mu {{\rho }_{2}}.
		\end{split}
	\end{equation*}
	When~${{{N}_{1}}}$ and ${{{N}_{2}}}$ are fixed, we choose~${{\varepsilon }_{2}}$ sufficiently small such that
	\begin{equation*}
		\quad {{\varepsilon }_{2}}<\min \left\{\frac{1}{4k{{N}_{2}}},\quad \frac{{{N}_{1}}}{4{{N}_{2}}\hat{b}}\right\}.
	\end{equation*}
	Since $g\in C^{0}({{\mathbb{R}}_{\geq 0}}; {{\mathbb{R}}_{>0}}) \cap C^{1}({{\mathbb{R}}_{\geq 0}}; {{\mathbb{R}_{\leq0}}})$, from \eqref{eq:2.6} we obtain
	\begin{equation*}
		H(g(s))\geq \left(\frac{g(s)}{g(0)}\right)^{\delta_1}H(g(0)),\quad\forall s\in {\mathbb{R}}_{\geq0}.
	\end{equation*}
	Finally, we choose~$N$ sufficiently large such that~$N>2\left({{C}_{1}}+\frac{{{C}_{2}}{{(g(0))}^{{\delta_{1}}}}}{{{k}_{0}}H(g(0))}\right)$, and thus we obtain
	\begin{equation*}
		\begin{split}
			&\quad{{C}_{2}}\int_{0}^{1}{\int_{0}^{\infty }{{{g}^{{\delta_{1}}}}(s){{\left| \eta _{x}^{t}(x, s) \right|}^{2}}\mathrm{d}s\mathrm{d}x}}+\left(\frac{N}{2}-{{C}_{1}}\right)\int_{0}^{1}{\int_{0}^{\infty }{g'(s){{\left| \eta _{x}^{t}(x, s) \right|}^{2}}\mathrm{d}s\mathrm{d}x}}\\
			&\le\left({{k}_{0}}\left({{C}_{1}}-\frac{N}{2}\right){{g}^{-{\delta_{1}}}}(0)H\left(g(0)\right)\right)\int_{0}^{1}{\int_{0}^{\infty }{{{g}^{{\delta_{1}}}}(s){{\left| \eta _{x}^{t}(x, s) \right|}^{2}}\mathrm{d}s\mathrm{d}x}}+ {{C}_{2}}\int_{0}^{1}{\int_{0}^{\infty }{{{g}^{{\delta_{1}}}}(s){{\left| \eta _{x}^{t}(x, s) \right|}^{2}}\mathrm{d}s\mathrm{d}x}}\\ 
			&=\left({{C}_{2}}+{{k}_{0}}\left({{C}_{1}}-\frac{N}{2}\right){{g}^{-{\delta_{1}}}}(0)H\left(g(0)\right)\right)\int_{0}^{1}{\int_{0}^{\infty }{{{g}^{{\delta_{1}}}}(s){{\left| \eta _{x}^{t}(x, s) \right|}^{2}}\mathrm{d}s\mathrm{d}x}}\\
			&<0,\quad \forall t\in {\mathbb{R}}_{\geq0}.
		\end{split}
	\end{equation*}
	In summary,~${{\Lambda }_{1}}, {{\Lambda }_{2}}, {{\Lambda }_{3}}, {{\Lambda }_{4}}$, and $ {{k}_{0}}\left(\frac{N}{2}-{{C}_{1}}\right){{g}^{-{\delta_{1}}}}(0)H\left(g(0)\right)-{{C}_{2}}$ are all positive real numbers.
	We take~${{\sigma }_{1}}\in {\mathbb{R}}_{>0}$ such that 
	\begin{center}
		$-{{\sigma }_{1}}=\max \left\{-{{\Lambda }_{1}}, -{{\Lambda }_{2}}, -{{\Lambda }_{3}}, -{{\Lambda }_{4}}, -\left({{k}_{0}}\left(\frac{N}{2}-{{C}_{1}}\right){{g}^{-{\delta_{1}}}}(0)H\left(g(0)\right)-{{C}_{2}}\right)\right\}$,
	\end{center} 
	and then we obtain \eqref{eq:39}.$\hfill\square$
\end{pf}

\par Next, we prove the equivalence between $E(t)$ and $L(t)$, as formulated in the proposition below.
\begin{proposition}
	There exist ${{\gamma }_{1}},{{\gamma }_{2}}\in {\mathbb{R}}_{>0}$ such that
	\begin{equation}
		{{\gamma }_{2}}E(t)\le L(t)\le {{\gamma}_{1}}E(t),\quad \forall t\in {\mathbb{R}}_{\geq0}.\label{eq:40}
	\end{equation}
\end{proposition}
\begin{pf}
	From \eqref{eq:35} we deduce that, for all $t\in {\mathbb{R}}_{\geq0}$,
	\begin{align}
		\left| L(t)-NE(t) \right|&\le {{N}_{1}}\left| {{I}_{1}}(t) \right|+{{N}_{2}}\left| {{I}_{2}}(t) \right|+\left| J(t) \right| +\frac{2{{\varepsilon }_{3}}{{\rho }_{1}}}{k}\int_{0}^{1}{\left| {{\phi }_{t}(x, t)}{{\phi }_{x}(x, t)} \right|\mathrm{d}x}+\mu \left| K(t) \right|\nonumber\\
		&\quad +\frac{{{\rho }_{2}}}{2{{\varepsilon }_{3}}}\int_{0}^{1}{\left| {{\psi }_{t}(x, t)}\left(\hat{b}{{\psi }_{x}(x, t)}+\int_{0}^{\infty }{g(s)\eta _{x}^{t}}(x, s)\mathrm{d}s\right) \right|dx}.
	\end{align}
	Combining \eqref{eq:13}-\eqref{eq:16} with \eqref{eq:37}, Young's inequality, Friedrichs' inequality, and Lemma \ref{lem:Lemma3-6}, for all $t\in {\mathbb{R}}_{\geq0}$, we obtain
	\begin{align*}
		\left| {{I}_{1}}(t) \right|&\le \frac{{{\rho }_{2}}}{2}\int_{0}^{1}{\psi _{t}^{2}(x, t)\mathrm{d}x}+\frac{{{\rho }_{1}}}{2}\int_{0}^{1}{\phi _{t}^{2}(x, t)\mathrm{d}x}+\frac{{{\rho }_{1}}+{{\rho }_{2}}}{2}\int_{0}^{1}{\psi _{x}^{2}(x, t)\mathrm{d}x}, \\
		\left| {{I}_{2}}(t) \right|&\le \frac{{{\rho }_{2}}}{2}\int_{0}^{1}{\psi _{t}^{2}(x, t)\mathrm{d}x}+\frac{{{\rho }_{2}}}{2}{G_{0}}\int_{0}^{1}{\int_{0}^{\infty }{g^{{\delta_{1}}}(s){{\left| \eta _{x}^{t}(x, s) \right|}^{2}}\mathrm{d}s\mathrm{d}x}},\\
		\left| J(t) \right|&\le \frac{{{\rho }_{2}}}{2}\int_{0}^{1}{\psi _{t}^{2}(x, t)\mathrm{d}x}+\frac{{{\rho }_{2}}}{2}\int_{0}^{1}{{{({{\phi }_{x}(x, t)}+\psi(x, t) )}^{2}}\mathrm{d}x}+\frac{{{\rho }_{1}}\hat{b}}{2k}\int_{0}^{1}{\psi _{x}^{2}(x, t)\mathrm{d}x}+\frac{{{\rho }_{1}}\left(1+\hat{b}\right)}{2k}\int_{0}^{1}{\phi _{t}^{2}(x, t)\mathrm{d}x}\\
		&\quad+\frac{{{\rho }_{1}}}{2k}{G_{0}}\int_{0}^{1}{\int_{0}^{\infty }{g^{{\delta_{1}}}(s){{\left| \eta _{x}^{t}(x, s) \right|}^{2}}\mathrm{d}s\mathrm{d}x}},\\
		\left| K(t) \right|&\le {{\rho }_{1}}\int_{0}^{1}{{{({{\phi }_{x}(x, t)}+\psi(x, t) )}^{2}}\mathrm{d}x}+\left({{\rho }_{1}}+\frac{{{\rho }_{2}}}{2}\right)\int_{0}^{1}{\psi _{x}^{2}(x, t)\mathrm{d}x}+\frac{{{\rho }_{1}}}{2}\int_{0}^{1}{\phi _{t}^{2}(x, t)\mathrm{d}x}+\frac{{{\rho }_{2}}}{2}\int_{0}^{1}{\psi _{t}^{2}(x, t)\mathrm{d}x}.
	\end{align*}
	By the previously defined linear function~$q(x)=2-4x  (x\in [0, 1])$, together with \eqref{eq:37}, Young's inequality, Friedrichs' inequality, and Lemma \ref{lem:Lemma3-6}, we deduce that for all $t\in {\mathbb{R}}_{\geq0}$,
	\begin{align*}
		&\quad\frac{2{{\varepsilon }_{3}}{{\rho }_{1}}}{k}\int_{0}^{1}{\left| {{\phi }_{t}(x, t)}{{\phi }_{x}(x, t)} \right|\mathrm{d}x}\\
		&\le \frac{{{\varepsilon }_{3}}{{\rho }_{1}}}{k}\int_{0}^{1}{\phi _{t}^{2}(x, t)\mathrm{d}x}+\frac{2{{\varepsilon }_{3}}{{\rho }_{1}}}{k}\int_{0}^{1}{\psi _{x}^{2}(x, t)\mathrm{d}x}+\frac{2{{\varepsilon }_{3}}{{\rho }_{1}}}{k}\int_{0}^{1}{{{\left({{\phi }_{x}(x, t)}+\psi(x, t)\right)}^{2}}\mathrm{d}x},\\
		&\quad\frac{{{\rho }_{2}}}{2{{\varepsilon }_{3}}}\int_{0}^{1}{\left| {{\psi }_{t}(x, t)}\left(\hat{b}{{\psi }_{x}(x, t)}+\int_{0}^{\infty }{g(s)\eta _{x}^{t}}(x, s)\mathrm{d}s\right) \right|\mathrm{d}x}\\
		&\le \frac{{{\rho }_{2}}}{4{{\varepsilon }_{3}}}\int_{0}^{1}{\psi _{t}^{2}(x, t)\mathrm{d}x}+\frac{{{\rho }_{2}}{{{\hat{b}}}^{2}}}{2{{\varepsilon }_{3}}}\int_{0}^{1}{\psi _{x}^{2}(x, t)\mathrm{d}x}+\frac{{{\rho }_{2}}}{2{{\varepsilon }_{3}}}{G_{0}}\int_{0}^{1}{\int_{0}^{\infty }{g^{{\delta_{1}}}(s){{\left| \eta _{x}^{t}(x, s) \right|}^{2}}\mathrm{d}s\mathrm{d}x}}.
	\end{align*}
	Since ${\delta_{1}}\in \big[1, \frac{3}{2}\big)$ and $g\in C^{0}({{\mathbb{R}}_{\geq 0}}; {{\mathbb{R}}_{>0}}) \cap C^{1}({{\mathbb{R}}_{\geq 0}}; {{\mathbb{R}_{\leq0}}})$, it follows that
	\begin{equation*}
		\begin{split}
			\int_{0}^{1}{\int_{0}^{\infty }{g^{{\delta_{1}}}(s){{\left| \eta _{x}^{t}(x, s) \right|}^{2}}\mathrm{d}s\mathrm{d}x}}
			&\leq g^{{\delta_{1}}-1}(0)\int_{0}^{1}{\int_{0}^{\infty }{g(s){{\left| \eta _{x}^{t}(x, s) \right|}^{2}}\mathrm{d}s\mathrm{d}x}},\quad \forall t\in {\mathbb{R}}_{\geq0}.
		\end{split}
	\end{equation*}
	Thus we deduce that, for all $t\in {\mathbb{R}}_{\geq0}$,
	\begin{align}
		\left| L(t)-NE(t) \right|
		&\le \frac{{{\rho }_{2}}}{2}\left({{N}_{1}}+{{N}_{2}}+1+\mu +\frac{1}{2{{\varepsilon }_{3}}}\right)\int_{0}^{1}{\psi _{t}^{2}(x, t)\mathrm{d}x}+\frac{{{\rho }_{1}}}{2}\left({{N}_{1}}+\frac{\hat{b}+1+2{{\varepsilon }_{3}}}{k}+\mu \right)\int_{0}^{1}{\phi _{t}^{2}(x, t)\mathrm{d}x}\nonumber\\ 
		&\quad+\left(\frac{{{\rho }_{1}}+{{\rho }_{2}}}{2}{{N}_{1}}+\frac{{{\rho }_{1}}\hat{b}}{2k}+ \mu\left({{\rho }_{1}}+\frac{{{\rho }_{2}}}{2}\right)+\frac{2{{\rho }_{1}}{{\varepsilon }_{3}}}{k}+\frac{{{\rho }_{2}}{{{\hat{b}}}^{2}}}{2{{\varepsilon }_{3}}}\right)\int_{0}^{1}{\psi _{x}^{2}(x, t)\mathrm{d}x}\nonumber\\
		&\quad+\left(\frac{{{\rho }_{2}}}{2}+{{\rho }_{1}}\mu +\frac{2{{\rho }_{1}}{{\varepsilon }_{3}}}{k}\right)\int_{0}^{1}{{{\left({{\phi }_{x}(x, t)}+\psi(x, t) \right)}^{2}}\mathrm{d}x}\nonumber\\
		&\quad+\left(\frac{{{\rho }_{2}}}{2}{{N}_{2}}+\frac{{{\rho }_{1}}}{2k}+\frac{{{\rho }_{2}}}{2{{\varepsilon }_{3}}}\right)g^{{\delta_{1}}-1}(0){G_{0}}\int_{0}^{1}{\int_{0}^{\infty }{g(s){{\left| \eta _{x}^{t}(x, s) \right|}^{2}}\mathrm{d}s\mathrm{d}x}}.\label{eq:5.7}
	\end{align}
	According to the previously chosen values of $N_{i}, {{\lambda }_{i}}(i=1, 2), {{\varepsilon }_{j}}(j=1, 2, 3, 4)$, and $\mu$, together with \eqref{eq:5.7}, there exists~$\gamma\in {\mathbb{R}}_{>0}$ such that
	\begin{equation*}
		\left| L(t)-NE(t) \right|\leq {\gamma}E(t),\quad \forall t\in {\mathbb{R}}_{\geq0}.
	\end{equation*}
	That is, there exists a sufficiently large ${N}\in {\mathbb{R}}_{>0}$ and ${{\gamma }_{1}},{{\gamma }_{2}}\in {\mathbb{R}}_{>0}$ such that \eqref{eq:40} holds. $\hfill\square$
\end{pf}

\subsection{Stability analysis}\label{subsec:7}
In this subsection, we perform the stability analysis of system \eqref{eq:sys} by employing the Lyapunov method within the iISS framework, providing details for establishing the PiISS and EiISS estimates.
\par Before proceeding to the stability analysis, we first state the following two lemmas.
\begin{lemma}\label{lem:Lemma 6.1}Let $p\geq 1$,~$T\in{\mathbb{R}}_{>0}$, and $y_0\in {\mathbb{R}}_{\geq0}$. Let~$u\in C([0,T]; {\mathbb{R}}_{\geq0})$ with $\underset{s\to+\infty }{\mathop{\lim }}\,\int_{0}^{s}{u(\tau )\mathrm{d}\tau}=+\infty $ and $v\in L^{2}(0, T)$. Suppose that $y(t)$ is a nonnegative absolutely continuous function on~$[0, T]$ and satisfies the following differential inequality
	\begin{align*}
		\begin{cases}
			y'(t) \leq -u(t)y^{p}(t)+v(t) & \text{a.e. in } [0,T], \\
			y(0) = y_0,
		\end{cases}
	\end{align*}
	then
	\begin{equation*}
		y(t)\leq \beta (y_{0}, t)+2\int_{0}^{t}v(s)\mathrm{d}s,\quad \forall t\in [0, T],
	\end{equation*}
	where for any $t\in [0, T]$ the function $\beta(\cdot, t)$ is defined by
	\begin{align*}
		\beta (s, t)=
		\begin{cases}
			0,&\quad s=0, \\ 
			{{\left({{s}^{-\left({p}-1\right)}}+\left({p}-1\right)\int_{0}^{t}{u(\tau)\mathrm{d}\tau}\right)}^{-\frac{1}{{p}-1}}},&\quad s \in {\mathbb{R}}_{>0}. \\ 
		\end{cases} 
	\end{align*}
\end{lemma}
The proof of Lemma \ref{lem:Lemma 6.1} can be found in \cite{Bi2025}, as Lemma \ref{lem:Lemma 6.1} is a special case of \cite[Lemma 3]{Bi2025}.
\par In view of the definitions of $\|\cdot\|_{\mathcal{H}}$ and $\|\cdot\|_\mathrm{H}$, the following lemma establishes the equivalence of these two norms.
\begin{lemma}\label{rem:4.2}
	In the Hilbert space $\mathcal{H}$, $\|\cdot\|_{\mathcal{H}}$ and $\|\cdot\|_\mathrm{H}$ are equivalent norms. 
\end{lemma}
\begin{pf}For any $X=(u, v, w, z, l)\in \mathcal{H}$, by Young's inequality and Friedrichs' inequality, we obtain
	\begin{align}
		\int_{0}^{1}{\left(u_{x}+w\right)^{2}}\mathrm{d}x&\leq 2\int_{0}^{1}{u_{x}^{2}}\mathrm{d}x+2\int_{0}^{1}{w^{2}}\mathrm{d}x\nonumber\\
		&\leq 2\int_{0}^{1}{u_{x}^{2}}\mathrm{d}x+2\int_{0}^{1}{w_{x}^{2}}\mathrm{d}x,\quad \forall t\in {\mathbb{R}}_{\geq0}.\label{eq:6.16}
	\end{align}
	Similarly, for~$\varepsilon_{0}\in\left(\frac{1}{\pi^{2}}, 1\right)$, we obtain
	\begin{align}
		2\int_{0}^{1}{\left| {{u}_{x}}w  \right|}\mathrm{d}x&\le {\varepsilon_{0}}\int_{0}^{1}{u_{x}^{2}}\mathrm{d}x+\frac{1}{{\varepsilon_{0}}}\int_{0}^{1}{{{w}^{2}}}\mathrm{d}x\nonumber\\
		&\le {\varepsilon_{0}}\int_{0}^{1}{u_{x}^{2}}\mathrm{d}x+\frac{1}{{\varepsilon_{0}}\pi^{2}}\int_{0}^{1}{w _{x}^{2}}\mathrm{d}x,\quad \forall t \in {\mathbb{R}}_{\geq0}.\label{eq:6.17}
	\end{align}
	From \eqref{eq:6.17}, it follows that for all $t \in {\mathbb{R}}_{\geq0}$,
	\begin{align}
		2\int_{0}^{1}{{{u}_{x}}w}\mathrm{d}x&\geq -{\varepsilon_{0}}\int_{0}^{1}{u _{x}^{2}}\mathrm{d}x-\frac{1}{{\varepsilon_{0}}\pi^{2}}\int_{0}^{1}{w_{x}^{2}}\mathrm{d}x.\label{eq:6.18}
	\end{align}
	Set $\alpha_{0}=\max\left\{k, \rho_{1}, \rho_{2}, \hat{b}, 1\right\}$ and $\tilde{\alpha}_{0}=\min\left\{k, \rho_{1}, \rho_{2}, \hat{b}, 1\right\}$. It follows that
	\begin{align*}
		\|X\|_\mathrm{H}^{2}&\leq \!\alpha_{0}\!\left(\|u_x+w\|_{L^2(0, 1)}^{2}+\|v\|_{L^2(0, 1)}^{2}+\|w_x\|_{L^2(0, 1)}^{2}+\|z\|_{L^2(0, 1)}^{2}+\|l\|_{L_{g}^{2}(\mathbb{R}_{\geq0}; H_{0}^{1}(0, 1))}^{2}\!\right)\\
		&\leq \!\alpha_{0}\!\left(2\|u_x\|_{L^2(0, 1)}^{2}+\|v\|_{L^2(0, 1)}^{2}+3\|w_x\|_{L^2(0, 1)}^{2}+\|z\|_{L^2(0, 1)}^{2}+\|l\|_{L_{g}^{2}(\mathbb{R}_{\geq0}; H_{0}^{1}(0, 1))}^{2}\!\right)\\
		&\leq \!3\alpha_{0}\!\left(\|u_x\|_{L^2(0, 1)}^{2}+\|v\|_{L^2(0, 1)}^{2}+\|w_x\|_{L^2(0, 1)}^{2}+\|z\|_{L^2(0, 1)}^{2}+\|l\|_{L_{g}^{2}(\mathbb{R}_{\geq0}; H_{0}^{1}(0, 1))}^{2}\!\right)\\
		&=\!3\alpha_{0}\|X\|_{\mathcal{H}}^{2},
	\end{align*}
	as well as
	\begin{align*}
		\|X\|_\mathrm{H}^{2}
		&\geq \!\tilde{\alpha}_{0}\!\left(\|u_x+w\|_{L^2(0, 1)}^{2}+\|v\|_{L^2(0, 1)}^{2}+\|w_x\|_{L^2(0, 1)}^{2}+\|z\|_{L^2(0, 1)}^{2}+\|l\|_{L_{g}^{2}(\mathbb{R}_{\geq0}; H_{0}^{1}(0, 1))}^{2}\!\right)\\
		&\geq \!\tilde{\alpha}_{0}\!\left((1-\varepsilon_{0})\|u_x\|_{L^2(0, 1)}^{2}+\|v\|_{L^2(0, 1)}^{2}+\left(1-\frac{1}{{\varepsilon_{0}}\pi^{2}}\right)\|w_x\|_{L^2(0, 1)}^{2}+\|z\|_{L^2(0, 1)}^{2}+\|l\|_{L_{g}^{2}(\mathbb{R}_{\geq0}; H_{0}^{1}(0, 1))}^{2}\!\right)\\
		&\geq \!\tilde{\alpha}_{0}\min\left\{(1-\varepsilon_{0}), \left(1-\frac{1}{{\varepsilon_{0}}\pi^{2} }\right)\right\}\|X\|_{\mathcal{H}}^{2},
	\end{align*}
	then
	\begin{align*}
		\tilde{\alpha}_{0}\min\left\{(1-\varepsilon_{0}), \left(1-\frac{1}{{\varepsilon_{0}}\pi^{2} }\right)\right\}\|X\|_{\mathcal{H}}^{2}	\leq \|X\|_\mathrm{H}^{2}\leq 3\alpha_{0}\|X\|_{\mathcal{H}}^{2}.
	\end{align*}
	Therefore, $\|\cdot\|_{\mathcal{H}}$ and $\|\cdot\|_\mathrm{H}$ are equivalent norms.$\hfill\square$
\end{pf}
Using the definition of $\|\cdot\|_{\mathcal{H}}$ and $\|\cdot\|_\mathrm{H}$ in the Hilbert space $\mathcal{H}$ together with Lemma \ref{rem:4.2}, the state $X(t)=(\phi, \phi_t, \psi, \psi_t, \eta^t)\in \mathcal{H}$ admits the following estimate:
\begin{equation*}
	\frac{\tilde{\alpha}_{0}}{2}\min\left\{(1-\varepsilon_{0}), \left(1-\frac{1}{{\varepsilon_{0}}\pi^{2} }\right)\right\}\|X\|_{\mathcal{H}}^{2}	\leq E(t)=\frac{1}{2}\|X\|_\mathrm{H}^{2}\leq \frac{3}{2}\alpha_{0}\|X\|_{\mathcal{H}}^{2}.
\end{equation*}
Set~${{\alpha }_{1}}=\frac{{\alpha_{0}}}{2}$ and ${{\alpha }_{2}}=\frac{\tilde{\alpha}_{0}}{2}$. Then it holds that
\begin{align}
	&\frac{1}{{{\alpha }_{1}}}E(t)\le e_{1}(t)+e_{2}(t)\le \frac{1}{{{\alpha }_{2}}}E(t),\quad\forall t\in {\mathbb{R}}_{\geq0},\label{eq:41}\\
	&{{\alpha }_{2}}\min\left\{(1-\varepsilon_{0}), \left(1-\frac{1}{{\varepsilon_{0}}\pi^{2} }\right)\right\}\|X\|_{\mathcal{H}}^{2}	\leq E(t)\leq 3\alpha_{1}\|X\|_{\mathcal{H}}^{2},\quad\forall t\in {\mathbb{R}}_{\geq0}.\label{eq:42}
\end{align}
\par In what follows, we give a complete proof of Theorem \ref{thm:main result} and establish the iISS estimates of system \eqref{eq:sys}.\\
\begin{theoremproof}{We discuss two cases separately: $\delta_{1}\in \left(1, \frac{3}{2}\right)$ and $\delta_{1}=1$.}
	\paragraph*{\textbf{Case 1: $\delta_{1}\in \left(1, \frac{3}{2}\right)$.}}
	From \eqref{eq:39} it is easy to see that
	\begin{equation*}
		\frac{\mathrm{d}}{\mathrm{d}t}L(t)\le {C_3}\int_{0}^{1}{{{f}^{2}}(x, t)\mathrm{d}x}+{C_4}\int_{0}^{1}{{{h}^{2}}(x, t)\mathrm{d}x},\quad \forall t\in \mathbb{R}_{>0}. 
	\end{equation*}
	Integrating the above inequality with respect to time~$t$, we obtain
	\begin{equation}
		L(t)\le L(0)+{C_3}\int_{0}^{t}{\left\| f(\cdot, \tau ) \right\|_{{{L}^{2}(0, 1)}}^{2}\mathrm{d}\tau }+{C_4}\int_{0}^{t}{\left\| h(\cdot, \tau ) \right\|_{{{L}^{2}(0, 1)}}^{2}\mathrm{d}\tau },\quad \forall t\in {\mathbb{R}}_{\geq0}.\label{eq:46}
	\end{equation}
	Let~$M(t)=L(0)+{C_3}\int_{0}^{t}{\left\| f(\cdot, \tau ) \right\|_{{{L}^{2}(0, 1)}}^{2}\mathrm{d}\tau }+{C_4}\int_{0}^{t}{\left\| h(\cdot, \tau ) \right\|_{{{L}^{2}(0, 1)}}^{2}\mathrm{d}\tau }$. From Assumption \ref{ass:Assumption1}, it follows that $\int_{0}^{\infty}\|f(\cdot, t)\|_{L^{2}(0, 1)}^{2}\mathrm{d}t<+\infty$ and $\int_{0}^{\infty}\|h(\cdot, t)\|_{L^{2}(0, 1)}^{2}\mathrm{d}t<+\infty$. This implies
	\begin{align}
		L(t)\le M(t)\leq \underset{t\geq0}{\mathop{\sup}}M(t)<+\infty, \quad \forall t\in {\mathbb{R}}_{\geq0}.\label{eq:6.17a}
	\end{align}
	From \eqref{eq:5}, \eqref{eq:40}, and \eqref{eq:41}, it is easy to see that 
	\begin{align}
		e_{1}(t)&\leq \frac{M(t)}{{{\gamma }_{2}}{{\alpha }_{2}}}<+\infty ,\quad \forall t\in {\mathbb{R}}_{\geq0},\label{eq:47}\\
		e_{2}(t)&\le \frac{2M(t)}{{{\gamma }_{2}}}<+\infty,\quad \forall t\in {\mathbb{R}}_{\geq0}.\nonumber
	\end{align}
	Let $m_0=\underset{\tau\geq0}{\mathop{\sup}} \left\|\psi_{0x}(\cdot, \tau) \right\|_{L^2(0, 1)}$. From \eqref{eq:his}, \eqref{eq:47}, and Assumption \ref{ass:Assumption1}, we deduce that
	\begin{align}
		&\quad \int_{0}^{\infty }{{{g}^{\frac{1}{2}}}(s)\int_{0}^{1}{{{\left| \eta _{x}^{t}(x, s) \right|}^{2}}\mathrm{d}x}}\mathrm{d}s\nonumber\\
		&=\int_{0}^{\infty }{{{g}^{\frac{1}{2}}}(s)\int_{0}^{1}{{{\left| {{\psi }_{x}}(x, t)-{{\psi }_{x}}(x, t-s) \right|}^{2}}\mathrm{d}x}}\mathrm{d}s\nonumber\\
		&\leq 2\int_{0}^{\infty }{{{g}^{\frac{1}{2}}}(s)\int_{0}^{1}{\psi _{x}^{2}(x, t)\mathrm{d}x}}\mathrm{d}s+2\int_{0}^{\infty }{{{g}^{\frac{1}{2}}}(s)\int_{0}^{1}{\psi _{x}^{2}(x, t-s)\mathrm{d}x}}\mathrm{d}s\label{eq:6.6}\\
		&\leq 2\frac{M(t)}{{{\gamma }_{2}}{{\alpha }_{2}}}\int_{0}^{\infty }{{{g}^{\frac{1}{2}}}(s)\mathrm{d}s}+2\int_{0}^{t}{{{g}^{\frac{1}{2}}}(s)\int_{0}^{1}{\psi _{x}^{2}(x, t-s)\mathrm{d}x}}\mathrm{d}s+2\int_{t}^{\infty }{{{g}^{\frac{1}{2}}}(s)\int_{0}^{1}{\psi _{x}^{2}(x, t-s)\mathrm{d}x}}\mathrm{d}s\nonumber\\
		&\leq 2\frac{M(t)}{{{\gamma }_{2}}{{\alpha }_{2}}}\int_{0}^{\infty }{{{g}^{\frac{1}{2}}}(s)\mathrm{d}s}+2\int_{0}^{t}{{{g}^{\frac{1}{2}}}(s)\int_{0}^{1}{\psi _{x}^{2}(x, t-s)\mathrm{d}x}}\mathrm{d}s+2\int_{t}^{\infty }{{{g}^{\frac{1}{2}}}(s)\int_{0}^{1}{\psi _{0x}^{2}(x, s-t)\mathrm{d}x}}\mathrm{d}s\nonumber\\
		&\leq 2\frac{M(t)}{{{\gamma }_{2}}{{\alpha }_{2}}}\int_{0}^{\infty }{{{g}^{\frac{1}{2}}}(s)\mathrm{d}s}+\frac{2}{{{\gamma }_{2}}{{\alpha }_{2}}}\int_{0}^{t}{{{g}^{\frac{1}{2}}}(s)M(t-s)\mathrm{d}s}+2\int_{0}^{\infty }{{{g}^{\frac{1}{2}}}(t+\tau )\int_{0}^{1}{\psi _{0x}^{2}(x, \tau )\mathrm{d}x}}\mathrm{d}\tau\nonumber\\
		&\leq 4\frac{M(t)}{{{\gamma }_{2}}{{\alpha }_{2}}}\int_{0}^{\infty }{{{g}^{\frac{1}{2}}}(s)\mathrm{d}s}+2\underset{\tau\geq0}{\mathop{\sup}} \left\|\psi_{0x}(\cdot, \tau) \right\|_{L^2(0, 1)}^{2}\int_{0}^{\infty }{{{g}^{\frac{1}{2}}}(\tau)}\mathrm{d}\tau\nonumber\\
		&\le \left(4\frac{M(t)}{{{\gamma }_{2}}{{\alpha }_{2}}}+2m_{0}^{2}\right)\int_{0}^{\infty }{{{g}^{\frac{1}{2}}}(s)\mathrm{d}s},\quad \forall t\in {\mathbb{R}}_{\geq0}.
	\end{align}
	Let $G_1=\int_{0}^{\infty }{{{g}^{\frac{1}{2}}}(s)\mathrm{d}s}$ and $\alpha_{3}=\max\left\{\frac{4}{{{\gamma }_{2}}{{\alpha }_{2}}}G_1, 2G_1\right\}$. Then we have
	\begin{equation*}
		\int_{0}^{\infty }{{{g}^{\frac{1}{2}}}(s)\int_{0}^{1}{{{\left| \eta _{x}^{t}(x, s) \right|}^{2}}\mathrm{d}x}}\mathrm{d}s\leq \alpha_{3}\left(M(t)+m_0^2\right),\quad \forall t\in {\mathbb{R}}_{\geq0}.
	\end{equation*}
	For all $t\in {\mathbb{R}}_{\geq0}$, by H\"{o}lder's inequality we obtain
	\begin{align}
		\int_{0}^{1}{\int_{0}^{\infty }{g(s){{\left| \eta _{x}^{t}(x, s) \right|}^{2}}}}\mathrm{d}s\mathrm{d}x
		&=\int_{0}^{1}{\int_{0}^{\infty }{{{g}^{\frac{{{\delta }_{1}}-1}{2{{\delta }_{1}}-1}}}(s)}}{{\left| \eta _{x}^{t}(x, s) \right|}^{\frac{4({{\delta }_{1}}-1)}{2{{\delta }_{1}}-1}}}{{g}^{\frac{{\delta }_{1}}{2{{\delta }_{1}}-1}}}(s){{\left| \eta _{x}^{t}(x, s) \right|}^{\frac{2}{2{{\delta }_{1}}-1}}}\mathrm{d}s\mathrm{d}x \nonumber\\
		&\le {{\left(\int_{0}^{1}{\int_{0}^{\infty }{{{g}^{\frac{1}{2}}}(s)}}{{\left| \eta _{x}^{t}(x, s) \right|}^{2}}\mathrm{d}s\mathrm{d}x \right)}^{\frac{2\left({{\delta }_{1}}-1\right)}{2{{\delta }_{1}}-1}}} {{\left(\int_{0}^{1}{\int_{0}^{\infty }{{{g}^{{{\delta }_{1}}}}(s)}}{{\left| \eta _{x}^{t}(x, s) \right|}^{2}}\mathrm{d}s\mathrm{d}x\right)}^{\frac{1}{2{{\delta }_{1}}-1}}},\label{eq:6.7}
	\end{align}
	which implies that
	\begin{align}
		\left(\int_{0}^{1}{\int_{0}^{\infty }{g(s){{\left| \eta _{x}^{t}(x, s) \right|}^{2}}}}\mathrm{d}s\mathrm{d}x\right)^{2{\delta }_{1}-1}&\leq \left(\alpha_{3}\left(M(t)+m_0^2\right)\right)^{2({\delta }_{1}-1)}\int_{0}^{1}{\int_{0}^{\infty }{{{g}^{{{\delta }_{1}}}}(s)}}{{\left| \eta _{x}^{t}(x, s) \right|}^{2}}\mathrm{d}s\mathrm{d}x.\label{eq:6.8}
	\end{align}
	From \eqref{eq:41}, \eqref{eq:47}, and \eqref{eq:6.8} we obtain
	\begin{align*}
		{{E}^{2{\delta_{1}}-1}}(t)&\le \alpha _{1}^{2{\delta_{1}}-1}{{\left(e_{1}(t)+e_{2}(t)\right)}^{2{\delta_{1}}-1}} \\ 
		& \le {{2}^{2({\delta_{1}}-1)}}\alpha _{1}^{2{\delta_{1}}-1}{{\left(e_{1}(t)\right)}^{2{\delta_{1}}-1}}+{{2}^{2({\delta_{1}}-1)}}\alpha _{1}^{2{\delta_{1}}-1}{\left(e_{2}(t)\right)}^{2{\delta_{1}}-1}\\
		& \le \alpha _{1}^{2{\delta_{1}}-1}{\left(\frac{2M(t)}{{{\gamma }_{2}}{{\alpha }_{2}}}\right)}^{2({\delta_{1}}-1)}e_{1}(t)+\alpha _{1}^{2{\delta_{1}}-1}\left(2\alpha_{3}\left(M(t)+m_0^2\right)\right)^{2(\delta_{1}-1)}\int_{0}^{1}{\int_{0}^{\infty }{{{g}^{{\delta_{1}}}}(s){{\left| \eta _{x}^{t}(x, s) \right|}^{2}}}}\mathrm{d}s\mathrm{d}x.
	\end{align*}
	Let $\alpha_{4}=\max\left\{\alpha _{1}^{2{\delta_{1}}-1}{\left(\frac{2}{{{\gamma }_{2}}{{\alpha }_{2}}}\right)}^{2({\delta_{1}}-1)}, \alpha _{1}^{2{\delta_{1}}-1}(2\alpha_{3})^{2({\delta_{1}}-1)}\right\}$. Then for all $t\in {\mathbb{R}}_{\geq0}$, we have
	\begin{align}
		{{E}^{2{\delta_{1}}-1}}(t)
		&\leq \alpha_{4}\left(M(t)+m_0^2\right)^{2(\delta_{1}-1)}\left(e_{1}(t)+\int_{0}^{1}{\int_{0}^{\infty }{{{g}^{{\delta_{1}}}}(s){{\left| \eta _{x}^{t}(x, s) \right|}^{2}}}}\mathrm{d}s\mathrm{d}x\right).\label{eq:6.9}
	\end{align}
	From \eqref{eq:39}, \eqref{eq:40}, and \eqref{eq:6.9} we obtain
	\begin{align*}
		\frac{\mathrm{d}}{\mathrm{d}t}L(t)&\le\! -\!\frac{{{\sigma }_{1}}}{\gamma _{1}^{2{\delta_{1}}\!-\!1}}\frac{1}{{{\alpha }_{4}}{\left( M(t)\!+\! m_0^2\right)^{2(\delta_{1}\!-\!1)}}}{{L}^{2{\delta_{1}}\!-\!1}}(t)\!+\!{C_3}\!\int_{0}^{1}\!{{{f}^{2}}(x, t)\mathrm{d}x}\!+\!{C_4}\!\int_{0}^{1}\!{{{h}^{2}}(x, t)\mathrm{d}x}.
	\end{align*}
	Let~$Q(t)=\dfrac{{{\sigma }_{1}}}{\gamma _{1}^{2{\delta_{1}}-1}{{\alpha }_{4}}\left(M(t)+m_0^2\right)^{2(\delta_{1}-1)}}$. Then
	\begin{equation}
		\frac{\mathrm{d}}{\mathrm{d}t}L(t)\le -Q(t){{L}^{2{\delta_{1}}-1}}(t)+{C_3}\int_{0}^{1}{{{f}^{2}}(x, t)\mathrm{d}x}+{C_4}\int_{0}^{1}{{{h}^{2}}(x, t)\mathrm{d}x},\quad \forall t\in {\mathbb{R}}_{>0}.\label{eq:50}
	\end{equation}
	Since $\delta_{1}\in \left(1, \frac{3}{2}\right)$, it follows that $2\delta_{1}-1>1$. By Assumption \ref{ass:Assumption1} and \eqref{eq:6.17a}, we deduce that $Q(t)$ in \eqref{eq:50} satisfies
	\begin{equation*}
		\begin{split}
			\underset{s\to+\infty }{\mathop{\lim }}\,\int_{0}^{s}{{Q}(\tau )\mathrm{d}\tau}&=\underset{s\to+\infty }{\mathop{\lim }}\,\int_{0}^{s}{\frac{{{\sigma }_{1}}}{\gamma _{1}^{2{\delta_{1}}-1}{{\alpha }_{4}}\left(M(\tau)+m_0^2\right)^{2(\delta_{1}-1)}}\mathrm{d}\tau}\\
			&\geq \underset{s\to+\infty }{\mathop{\lim }}\,\int_{0}^{s}{\frac{{{\sigma }_{1}}}{\gamma _{1}^{2{\delta_{1}}-1}{{\alpha }_{4}}\left(\underset{\tau\geq0}{\mathop{\sup}}M(\tau)+m_0^2\right)^{2(\delta_{1}-1)}}\mathrm{d}\tau}\\
			&=+\infty.
		\end{split}
	\end{equation*}
	Let $m_1=2({\delta_{1}}-1)\left(\dfrac{{\sigma }_{1}}{\alpha_{4}\gamma_{1}^{2\delta_{1}-1}}\right)$. Combining \eqref{eq:50} and Lemma \ref{lem:Lemma 6.1}, we obtain for all $t\in \mathbb{R}_{\geq0}$,
	\begin{equation*}
		L(t)\le \beta (L(0), t)+2{C_3}\int_{0}^{t}{\left\| f(\cdot, \tau ) \right\|_{{{L}^{2}(0, 1)}}^{2}\mathrm{d}\tau }+2{C_4}\int_{0}^{t}{\left\| h(\cdot, \tau ) \right\|_{{{L}^{2}(0, 1)}}^{2}\mathrm{d}\tau },
	\end{equation*}
	where
	\begin{align}
		\beta (L(0), t)&=\left(L^{-2({\delta_{1}}-1)}(0)+m_1\int_{0}^{t}{\frac{1}{\left(M(\tau )+m_0^2\right)^{2(\delta_{1}-1)}}}\mathrm{d}\tau\right)^{-\frac{1}{2(\delta_{1}-1)}}.\label{eq:53}
	\end{align}
	From the definition of $M(t)$, we deduce that
	\begin{align*}
		M(\tau)=L(0)+{{C}_{3}}\int_{0}^{\tau }{{{\left\| f(\cdot, s) \right\|}_{L^{2}(0, 1)}^{2}}\mathrm{d}s}+{{C}_{4}}\int_{0}^{\tau }{{\left\| h(\cdot, s) \right\|}_{L^{2}(0, 1)}^{2}}\mathrm{d}s.
	\end{align*}
	For any $T\in \mathbb{R}_{>0}$, we consider two cases.\\
	\quad {Case 1: }$L(0)+m_0^2\geq {{C}_{3}}\int_{0}^{T}{{{\left\| f(\cdot, s) \right\|}_{L^{2}(0, 1)}^{2}}\mathrm{d}s}+{{C}_{4}}\int_{0}^{T}{{\left\| h(\cdot, s) \right\|}_{L^{2}(0, 1)}^{2}}\mathrm{d}s$. In this case, it follows that
	\begin{align}
		M(\tau)+m_0^2&=L(0)+{{C}_{3}}\int_{0}^{\tau }{{{\left\| f(\cdot, s) \right\|}_{L^{2}(0, 1)}^{2}}\mathrm{d}s}+{{C}_{4}}\int_{0}^{\tau }{{\left\| h(\cdot, s) \right\|}_{L^{2}(0, 1)}^{2}}\mathrm{d}s+m_0^2\nonumber\\
		&\leq L(0)+{{C}_{3}}\int_{0}^{T}{{{\left\| f(\cdot, s) \right\|}_{L^{2}(0, 1)}^{2}}\mathrm{d}s}+{{C}_{4}}\int_{0}^{T}{{\left\| h(\cdot, s) \right\|}_{L^{2}(0, 1)}^{2}}\mathrm{d}s+m_{0}^{2}\nonumber\\
		&\leq 2\left(L(0)+m_0^2\right), \quad \forall \tau \in [0, T].\label{eq:6.13}
	\end{align}
	From \eqref{eq:53} and \eqref{eq:6.13}, we deduce that for all $t \in [0, T]$,
	\begin{align*}
		\beta (L(0), t)&\leq \left(L^{-2(\delta_{1}-1)}(0)+m_1\int_{0}^{t}{\left(\frac{1}{2\left(L(0)+m_0^2\right)}\right)^{2(\delta_{1}-1)}\mathrm{d}\tau}\right)^{-\frac{1}{2(\delta_{1}-1)}}\\
		&=\left(L^{-2(\delta_{1}-1)}(0)+\frac{m_1}{4^{\delta_{1}-1}}\left(L(0)+m_0^2\right)^{-2(\delta_{1}-1)}t\right)^{-\frac{1}{2(\delta_{1}-1)}}.
	\end{align*}
	Then for all $t\in [0, T]$, we obtain
	\begin{align}
		L(t)&\leq \left(L^{-2(\delta_{1}-1)}(0)+\frac{m_1}{4^{\delta_{1}-1}}\left(L(0)+m_0^2\right)^{-2(\delta_{1}-1)}t\right)^{-\frac{1}{2(\delta_{1}-1)}}+2{C_3}\int_{0}^{t}{\left\| f(\cdot, \tau ) \right\|_{{{L}^{2}(0, 1)}}^{2}\mathrm{d}\tau }\nonumber\\
		&\quad+2{C_4}\int_{0}^{t}{\left\| h(\cdot, \tau ) \right\|_{{{L}^{2}(0, 1)}}^{2}\mathrm{d}\tau }.\label{eq:51}
	\end{align}
	{Case 2: } $L(0)+m_0^2<{{C}_{3}}\int_{0}^{T}{{{\left\| f(\cdot, s) \right\|}_{L^{2}(0, 1)}^{2}}\mathrm{d}s}+{{C}_{4}}\int_{0}^{T}{{\left\| h(\cdot, s) \right\|}_{L^{2}(0, 1)}^{2}}\mathrm{d}s$. From \eqref{eq:46}, we obtain
	\begin{align}
		L(T)&\leq L(0)+{{C}_{3}}\int_{0}^{T}{{{\left\| f(\cdot, s) \right\|}_{L^{2}(0, 1)}^{2}}\mathrm{d}s}+{{C}_{4}}\int_{0}^{T}{{\left\| h(\cdot, s) \right\|}_{L^{2}(0, 1)}^{2}}\mathrm{d}s\nonumber\\
		&\leq 2{{C}_{3}}\int_{0}^{T}{{{\left\| f(\cdot, s) \right\|}_{L^{2}(0, 1)}^{2}}\mathrm{d}s}+2{{C}_{4}}\int_{0}^{T}{{\left\| h(\cdot, s) \right\|}_{L^{2}(0, 1)}^{2}}\mathrm{d}s\nonumber\\
		&\leq \left(L^{-2(\delta_{1}-1)}(0)+\frac{m_1}{4^{\delta_{1}-1}}\left(L(0)+m_0^2\right)^{-2(\delta_{1}-1)}T\right)^{-\frac{1}{2(\delta_{1}-1)}}+2{C_3}\int_{0}^{T}{\left\| f(\cdot, \tau ) \right\|_{{{L}^{2}(0, 1)}}^{2}\mathrm{d}\tau }\nonumber\\
		&\quad+2{C_4}\int_{0}^{T}{\left\| h(\cdot, \tau ) \right\|_{{{L}^{2}(0, 1)}}^{2}\mathrm{d}\tau }.\label{eq:52}
	\end{align}
	When $T=0$, \eqref{eq:52} still holds. Combining \eqref{eq:51} and \eqref{eq:52}, we conclude that for all $t\in {\mathbb{R}}_{\geq0}$,
	\begin{align}
		L(t)&\leq \left(L^{-2(\delta_{1}-1)}(0)+\frac{m_1}{4^{\delta_{1}-1}}\left(L(0)+m_0^2\right)^{-2(\delta_{1}-1)}t\right)^{-\frac{1}{2(\delta_{1}-1)}}+2{C_3}\int_{0}^{t}{\left\| f(\cdot, \tau ) \right\|_{{{L}^{2}(0, 1)}}^{2}\mathrm{d}\tau }+2{C_4}\int_{0}^{t}{\left\| h(\cdot, \tau ) \right\|_{{{L}^{2}(0, 1)}}^{2}\mathrm{d}\tau }\nonumber\\
		&\leq \left(\left(L(0)+m_0^2\right)^{-2(\delta_{1}-1)}+\frac{m_1}{4^{\delta_{1}-1}}\left(L(0)+m_0^2\right)^{-2(\delta_{1}-1)}t\right)^{-\frac{1}{2(\delta_{1}-1)}}+2{C_3}\int_{0}^{t}{\left\| f(\cdot, \tau ) \right\|_{{{L}^{2}(0, 1)}}^{2}\mathrm{d}\tau }\nonumber\\
		&\quad+2{C_4}\int_{0}^{t}{\left\| h(\cdot, \tau ) \right\|_{{{L}^{2}(0, 1)}}^{2}\mathrm{d}\tau }\nonumber\\
		&\leq \left(L(0)+m_0^2\right)\left(1+\frac{m_1}{4^{\delta_{1}-1}}t\right)^{-\frac{1}{2(\delta_{1}-1)}}+2{C_3}\int_{0}^{t}{\left\| f(\cdot, \tau ) \right\|_{{{L}^{2}(0, 1)}}^{2}\mathrm{d}\tau }+2{C_4}\int_{0}^{t}{\left\| h(\cdot, \tau ) \right\|_{{{L}^{2}(0, 1)}}^{2}\mathrm{d}\tau }.\label{eq:54}
	\end{align}
	Let $\alpha_{5}=\left(\min\left\{1, \frac{m_1}{4^{\delta_{1}-1}}\right\}\right)^{\frac{1}{2(\delta_{1}-1)}}$. From \eqref{eq:40} and \eqref{eq:54}, we deduce that for all $t\in {\mathbb{R}}_{\geq0}$,
	\begin{align}
		E(t)
		&\leq \frac{\gamma_{1}+1}{\gamma_{2}}\left(E(0)+m_0^2\right)\left(1+\frac{m_1}{4^{\delta_{1}-1}}t\right)^{-\frac{1}{2(\delta_{1}-1)}}+\frac{2C_3}{\gamma_{2}}\int_{0}^{t}{\left\| f(\cdot, \tau ) \right\|_{{{L}^{2}(0, 1)}}^{2}\mathrm{d}\tau }+\frac{2C_4}{\gamma_{2}}\int_{0}^{t}{\left\| h(\cdot, \tau ) \right\|_{{{L}^{2}(0, 1)}}^{2}\mathrm{d}\tau }\nonumber\\
		&\leq \frac{\gamma_{1}+1}{\gamma_{2}\alpha_{5}}\left(E(0)+m_0^2\right)\left(1+t\right)^{-\frac{1}{2(\delta_{1}-1)}}+\frac{2C_3}{\gamma_{2}}\int_{0}^{t}{\left\| f(\cdot, \tau ) \right\|_{{{L}^{2}(0, 1)}}^{2}\mathrm{d}\tau }+\frac{2C_4}{\gamma_{2}}\int_{0}^{t}{\left\| h(\cdot, \tau ) \right\|_{{{L}^{2}(0, 1)}}^{2}\mathrm{d}\tau }.
	\end{align}
	Letting $\alpha_{6}=\alpha_{2}\min\left\{1-\varepsilon_{0}, 1-\frac{1}{\varepsilon_{0}\pi^{2}}\right\}$ and combining with \eqref{eq:42}, we obtain
	\begin{align}
		\|X(t)\|_{\mathcal{H}}^{2}
		&\leq \frac{\left(\gamma_{1}+1\right)\max\left\{3\alpha_{1}, 1\right\}}{\gamma_{2}\alpha_{5}\alpha_{6}}\left(\|X_0\|_{\mathcal{H}}^{2}+m_0^2\right)(1+t)^{-\frac{1}{2(\delta_{1}-1)}}+\frac{2C_3}{\gamma_{2}\alpha_{6}}\int_{0}^{t}{\left\| f(\cdot, \tau ) \right\|_{{{L}^{2}(0, 1)}}^{2}\mathrm{d}\tau }\nonumber\\
		&\quad+\frac{2C_4}{\gamma_{2}\alpha_{6}}\int_{0}^{t}{\left\| h(\cdot, \tau ) \right\|_{{{L}^{2}(0, 1)}}^{2}\mathrm{d}\tau },\quad \forall t\in {\mathbb{R}}_{\geq0}.
	\end{align}
	Let~$\hat{a}_{1}=\frac{\left(\gamma_{1}+1\right)\max\left\{3\alpha_{1}, 1\right\}}{\gamma_{2}\alpha_{5}\alpha_{6}}, \hat{C}_{1}=\max \left\{\frac{2C_3}{\gamma_{2}\alpha_{6}}, \frac{2C_4}{\gamma_{2}\alpha_{6}}\right\}$. Then we obtain
	\begin{align*}
		\|X(t)\|_{\mathcal{H}}^{2}&\leq\hat{a}_{1}\left(\|X_0\|_{\mathcal{H}}^{2}+\underset{s\geq0}{\mathop{\sup}} \left\|\psi_{0x}(\cdot, s) \right\|_{L^2(0, 1)}^{2}\right)(1+t)^{-\frac{1}{2(\delta_{1}-1)}}\nonumber\\
		&\quad+\hat{C}_{1}\left(\int_{0}^{t}{\left\| f(\cdot, \tau ) \right\|_{{{L}^{2}(0, 1)}}^{2}\mathrm{d}\tau }+\int_{0}^{t}{\left\| h(\cdot, \tau ) \right\|_{{{L}^{2}(0, 1)}}^{2}\mathrm{d}\tau }\right),\quad \forall t\in {\mathbb{R}}_{\geq0}.
	\end{align*}
	For any $a\geq0$ and $b\geq 0$, it holds that $\left(a+b\right)^{\frac{1}{2}}\leq a^{\frac{1}{2}}+b^{\frac{1}{2}}$, which implies that for all $t\in {\mathbb{R}}_{\geq0}$,
	\begin{align}
		\|X(t)\|_{\mathcal{H}}
		&\leq \sqrt{\hat{a}_{1}}\left(\|X_0\|_{\mathcal{H}}+\underset{s\geq0}{\mathop{\sup}} \left\|\psi_{0x}(\cdot, s) \right\|_{L^2(0, 1)}\right)(1+t)^{-\frac{1}{4(\delta_{1}-1)}}+\sqrt{\hat{C}_{1}}\left(\int_{0}^{t}{\left\| f(\cdot, \tau ) \right\|_{{{L}^{2}(0, 1)}}^{2}\mathrm{d}\tau }\right)^{\frac{1}{2}}\nonumber\\
		&\quad+\sqrt{\hat{C}_{1}}\left(\int_{0}^{t}{\left\| h(\cdot, \tau ) \right\|_{{{L}^{2}(0, 1)}}^{2}\mathrm{d}\tau }\right)^{\frac{1}{2}}.\label{eq:6.23}
	\end{align}
	By Definition \ref{def:Def 1}, when $\delta_{1}\in \left(1, \frac{3}{2}\right)$, system \eqref{eq:sys} is PiISS. In addition, when $f\equiv h\equiv0$, from \eqref{eq:36a} we obtain
	\begin{align*}
		\frac{\mathrm{d}}{\mathrm{d}t}E(t)=\frac{1}{2}{{\int_{0}^{1}{\int_{0}^{\infty }{{g}'(s)\left| \eta _{x}^{t}(x, s) \right|^{2}}}}}\mathrm{d}s\mathrm{d}x\leq0, \quad \forall t\in \mathbb{R}_{>0},
	\end{align*}
	thus,
	\begin{equation}
		E(t)\leq E(0),\quad \forall t\in \mathbb{R}_{\geq0}.
	\end{equation}
	From \eqref{eq:6.6}, we obtain
	\begin{align*}
		\int_{0}^{\infty }\!\!{{{g}^{\frac{1}{2}}}(s)\int_{0}^{1}\!\!{{{\left| \eta _{x}^{t}(x, s) \right|}^{2}}\mathrm{d}x}}\mathrm{d}s&\leq \frac{4E(0)}{\hat{b}}G_1+2\int_{0}^{t}\!\!{{{g}^{\frac{1}{2}}}(s)\int_{0}^{1}\!\!{\psi _{x}^{2}(x, t-s)\mathrm{d}x}}\mathrm{d}s+2\int_{t}^{\infty}\!\!{{{g}^{\frac{1}{2}}}(s)\int_{0}^{1}\!\!{\psi _{x}^{2}(x, t-s)\mathrm{d}x}}\mathrm{d}s\\
		&\leq \left(\frac{8E(0)}{\hat{b}}+2m_0^2\right)G_1,\quad \forall t\in {\mathbb{R}}_{\geq0}.
	\end{align*}
	Let $\overline{\alpha}_{3}=\max\left\{\frac{8}{\hat{b}}G_1, 2G_1\right\}$. From \eqref{eq:6.7}, we obtain
	\begin{align*}
		\left(\int_{0}^{1}{\int_{0}^{\infty }{g(s){{\left| \eta _{x}^{t}(x, s) \right|}^{2}}}}\mathrm{d}s\mathrm{d}x\right)^{2{\delta }_{1}-1}&\leq\left(\overline{\alpha}_{3}\left(E(0)+m_0^2\right)\right)^{2({\delta }_{1}-1)}\int_{0}^{1}{\int_{0}^{\infty }{{{g}^{{{\delta }_{1}}}}(s)}}{{\left| \eta _{x}^{t}(x, s) \right|}^{2}}\mathrm{d}s\mathrm{d}x,\quad \forall t\in {\mathbb{R}}_{\geq0}.
	\end{align*}
	Similarly, let $\overline{\alpha}_{4}=\max\left\{\alpha _{1}^{2{\delta_{1}}-1}{\left(\frac{2}{{\alpha }_{2}}\right)}^{2({\delta_{1}}-1)}, \alpha _{1}^{2{\delta_{1}}-1}(2\overline{\alpha}_{3})^{2(\delta_{1}-1)}\right\}$. From \eqref{eq:41}, it follows that for all $t\in\mathbb{R}_{\geq0}$,
	\begin{align*}
		{{E}^{2{\delta_{1}}-1}}(t)
		& \le \alpha _{1}^{2{\delta_{1}}-1}{\left(\frac{2E(0)}{{\alpha }_{2}}\right)}^{2({\delta_{1}}-1)}e_{1}(t)+\alpha _{1}^{2{\delta_{1}}-1}(2\overline{\alpha}_{3})^{2(\delta_{1}-1)}\left(E(0)+m_0^2\right)^{2(\delta_{1}-1)}\int_{0}^{1}\!{\int_{0}^{\infty }\!\!{{{g}^{{\delta_{1}}}}(s){{\left| \eta _{x}^{t}(x, s) \right|}^{2}}}}\mathrm{d}s\mathrm{d}x\\
		&\leq \overline{\alpha}_{4}\left(E(0)+m_0^2\right)^{2(\delta_{1}-1)}\left(e_{1}(t)+\int_{0}^{1}\!{\int_{0}^{\infty }\!\!{{{g}^{{\delta_{1}}}}(s){{\left| \eta _{x}^{t}(x, s) \right|}^{2}}}}\mathrm{d}s\mathrm{d}x\right).
	\end{align*}
	From \eqref{eq:39} and \eqref{eq:40}, we obtain
	\begin{align*}
		\frac{\mathrm{d}}{\mathrm{d}t}L(t)
		&\leq-\frac{\sigma_{1}}{\overline{\alpha}_{4}\gamma_{1}^{2\delta_{1}-1}\left(E(0)+m_0^2\right)^{2(\delta_{1}-1)}}L^{2\delta_1-1}(t),\quad \forall t\in {\mathbb{R}}_{>0}.
	\end{align*}
	Let $\overline{m}_{1}=2(\delta_{1}-1)\frac{\sigma_{1}}{\overline{\alpha}_{4}\gamma_{1}^{2\delta_{1}-1}}$. By Lemma \ref{lem:Lemma 6.1}, we have
	\begin{align*}
		L(t)&\leq \left(L^{-2(\delta_{1}-1)}(0)+\frac{\overline{m}_{1}}{\left(E(0)+m_0^2\right)^{2(\delta_{1}-1)}}t\right)^{-\frac{1}{2(\delta_{1}-1)}},\quad \forall t\in {\mathbb{R}}_{\geq0}.
	\end{align*}
	Let $\overline{\alpha}_{5}=\left(\min\left\{1, \overline{m}_{1}\gamma_{1}^{2(\delta_{1}-1)}\right\}\right)^{\frac{1}{2(\delta_{1}-1)}}$. Then we obtain
	\begin{align*}
		E(t)&\leq \frac{1}{\gamma_{2}}\left(\left(\gamma_{1}E(0)\right)^{-2(\delta_{1}-1)}+\overline{m}_{1}\left(E(0)+m_0^2\right)^{-2(\delta_{1}-1)}t\right)^{-\frac{1}{2(\delta_{1}-1)}}\\
		&\leq \frac{\gamma_{1}}{\gamma_{2}}\left(E^{-2(\delta_{1}-1)}(0)+\overline{m}_{1}\gamma_{1}^{2(\delta_{1}-1)}\left(E(0)+m_0^2\right)^{-2(\delta_{1}-1)}t\right)^{-\frac{1}{2(\delta_{1}-1)}}\\
		&\leq\frac{\gamma_{1}}{\gamma_{2}\overline{\alpha}_{5}}\left(E^{-2(\delta_{1}-1)}(0)+\left(E(0)+m_0^2\right)^{-2(\delta_{1}-1)}t\right)^{-\frac{1}{2(\delta_{1}-1)}}\\
		&\leq\frac{\gamma_{1}}{\gamma_{2}\overline{\alpha}_{5}}\left(E(0)+m_0^2\right)(1+t)^{-\frac{1}{2(\delta_{1}-1)}},\quad \forall t\in {\mathbb{R}}_{\geq0}.\label{eq:66}
	\end{align*}
	From \eqref{eq:his}, \eqref{eq:5}, \eqref{eq:66}, and Young's inequality, we deduce that
	\begin{align*}
		\int_{0}^{t}{\int_{0}^{1}{{{\left| \eta _{x}^{t}(x, s) \right|}^{2}}\mathrm{d}x}}\mathrm{d}s&\leq 2\int_{0}^{t}{\int_{0}^{1}{{\psi_{x}^{2}(x, s)}\mathrm{d}x}}\mathrm{d}s+2\int_{0}^{t}{\int_{0}^{1}{{\psi_{x}^{2}(x, t-s)}\mathrm{d}x}}\mathrm{d}s\\
		&\leq \frac{4}{\hat{b}}E(t)t+\frac{4}{\hat{b}}\int_{0}^{t}{E(t-s)}\mathrm{d}s\\
		&\leq \frac{4\gamma_{1}}{\hat{b}\gamma_{2}\overline{\alpha}_{5}}\left(E(0)+m_0^2\right)(1+t)^{-\frac{3-2\delta_{1}}{2(\delta_{1}-1)}} +\frac{4\gamma_{1}}{\hat{b}\gamma_{2}\overline{\alpha}_{5}}\left(E(0)+m_0^2\right)\left[\frac{2(\delta_{1}-1)}{3-2\delta_{1}}(1+t-s)^{\frac{2\delta_{1}-3}{2(\delta_{1}-1)}}\right]_{s=0}^{s=t}\\
		&\leq\frac{4\gamma_{1}}{\hat{b}\gamma_{2}\overline{\alpha}_{5}}\left(E(0)+m_0^2\right)(1+t)^{-\frac{3-2\delta_{1}}{2(\delta_{1}-1)}} +\frac{4\gamma_{1}}{\hat{b}\gamma_{2}\overline{\alpha}_{5}}\left(E(0)+m_0^2\right)\frac{2(\delta_{1}-1)}{3-2\delta_{1}}\\
		&\leq\frac{4\gamma_{1}}{\hat{b}\gamma_{2}\overline{\alpha}_{5}}\left(E(0)+m_0^2\right)\left(1+\frac{2(\delta_{1}-1)}{3-2\delta_{1}}\right)\\
		&\leq \frac{4\gamma_{1}}{\hat{b}\gamma_{2}\overline{\alpha}_{5}(3-2\delta_{1})}\left(E(0)+m_0^2\right).
	\end{align*}
	Let $\zeta=\left(\frac{4\gamma_{1}}{\hat{b}\gamma_{2}\overline{\alpha}_{5}(3-2\delta_{1})}\right)^{\delta_{1}-1}$, where $\zeta$ is a constant independent of time $t$. Then we obtain
	\begin{align*}
		\left(\int_{0}^{\infty}{\int_{0}^{1}{{{\left| \eta _{x}^{t}(x, s) \right|}^{2}}\mathrm{d}x}}\mathrm{d}s\right)^{\delta_{1}-1}\leq \zeta\left(E(0)+m_0^2\right)^{\delta_{1}-1},\quad \forall t\in {\mathbb{R}}_{\geq0}.
	\end{align*}
	For all $t\in {\mathbb{R}}_{\geq0}$, by H\"{o}lder's inequality we obtain
	\begin{align*}
		\int_{0}^{1}{\int_{0}^{\infty }{g(s){{\left| \eta _{x}^{t}(x, s) \right|}^{2}}}}\mathrm{d}s\mathrm{d}x&\le{\left(\int_{0}^{1}{\int_{0}^{\infty }{{{g}^{\delta_{1}}(s)}}{{\left| \eta _{x}^{t}(x, s) \right|}^{2}}\mathrm{d}s\mathrm{d}x}\right)^{\frac{1}{\delta_{1}}}}{{\left(\int_{0}^{1}{\int_{0}^{\infty }}{{\left| \eta _{x}^{t}(x, s) \right|}^{2}}\mathrm{d}s\mathrm{d}x\right)}^{\frac{\delta_{1}-1}{\delta_{1}}}},
	\end{align*}
	which leads to
	\begin{align}
		\left(\int_{0}^{1}{\int_{0}^{\infty }{g(s){{\left| \eta _{x}^{t}(x, s) \right|}^{2}}}}\mathrm{d}s\mathrm{d}x \right)^{\delta_1 }
		&\leq \zeta\left( E(0)+m_0^2 \right)^{\delta_{1}-1}\int_{0}^{1}{\int_{0}^{\infty }{g^{\delta_{1}}(s){{\left| \eta _{x}^{t}(x, s) \right|}^{2}}}}\mathrm{d}s\mathrm{d}x.
	\end{align}
	Similarly, let $\tilde{\alpha}_{4}=\max\left\{\alpha _{1}^{\delta_{1}}{\left(\frac{2}{{\alpha }_{2}}\right)}^{\delta_{1}-1}, \alpha _{1}^{\delta_{1}}2^{\delta_{1}-1}\zeta\right\}$. From \eqref{eq:41}, we obtain
	\begin{align*}
		{{E}^{\delta_{1}}}(t)
		& \le \alpha _{1}^{\delta_{1}}{\left(\frac{2E(0)}{{\alpha }_{2}}\right)}^{\delta_{1}-1}e_{1}(t)+\alpha _{1}^{\delta_{1}}2^{\delta_{1}-1}\zeta\left(E(0)+m_0^2\right)^{\delta_{1}-1}\int_{0}^{1}{\int_{0}^{\infty }{{{g}^{{\delta_{1}}}}(s){{\left| \eta _{x}^{t}(x, s) \right|}^{2}}}}\mathrm{d}s\mathrm{d}x\\
		&\leq \tilde{\alpha}_{4}\left(E(0)+m_0^2\right)^{\delta_{1}-1}\left(e_{1}(t)+\int_{0}^{1}{\int_{0}^{\infty }{{{g}^{{\delta_{1}}}}(s){{\left| \eta _{x}^{t}(x, s) \right|}^{2}}}}\mathrm{d}s\mathrm{d}x\right),\quad \forall t\in {\mathbb{R}}_{\geq0}.
	\end{align*}
	From \eqref{eq:39} and \eqref{eq:40}, we obtain
	\begin{align*}
		\frac{\mathrm{d}}{\mathrm{d}t}L(t)
		&\leq-\frac{\sigma_{1}}{\tilde{\alpha}_{4}\left(E(0)+m_0^2\right)^{\delta_{1}-1}\gamma_{1}^{\delta_{1}}}L^{\delta_1}(t),\quad \forall t\in {\mathbb{R}}_{>0}.
	\end{align*}
	Let $ \tilde{m}_{1}=\frac{\sigma_{1}(\delta_{1}-1)}{\tilde{\alpha}_{4}\gamma_{1}^{\delta_{1}}}$. By Lemma \ref{lem:Lemma 6.1}, we have
	\begin{align*}
		L(t)&\leq \left(L^{-(\delta_{1}-1)}(0)+\frac{\sigma_{1}}{\tilde{\alpha}_{4}\left(E(0)+m_0^2\right)^{\delta_{1}-1}\gamma_{1}^{\delta_{1}}}(\delta_{1}-1)t\right)^{-\frac{1}{\delta_{1}-1}}\\
		&=\left(L^{-(\delta_{1}-1)}(0)+\tilde{m}_{1}\left(E(0)+m_0^2\right)^{-(\delta_{1}-1)}t\right)^{-\frac{1}{\delta_{1}-1}},\quad \forall t\in {\mathbb{R}}_{\geq0}.
	\end{align*}
	Let $\tilde{\alpha}_{5}=\left(\min\left\{1, \tilde{m}_{1}\gamma_{1}^{(\delta_{1}-1)}\right\}\right)^{\frac{1}{\delta_{1}-1}}$. Then we obtain
	\begin{align*}
		E(t)
		&\leq \frac{\gamma_{1}}{\gamma_{2}}\left(E^{-(\delta_{1}-1)}(0)+\tilde{m}_{1}\gamma_{1}^{(\delta_{1}-1)}\left(E(0)+m_0^2\right)^{-(\delta_{1}-1)}t\right)^{-\frac{1}{(\delta_{1}-1)}}\\
		&\leq\frac{\gamma_{1}}{\gamma_{2}\tilde{\alpha}_{5}}\left(E(0)+m_0^2\right)(1+t)^{-\frac{1}{(\delta_{1}-1)}},\quad \forall t\in {\mathbb{R}}_{\geq0}.
	\end{align*}
	Let $\hat{{a}}_{2}=\frac{\gamma_{1}\max\left\{3\alpha_{1}, 1\right\}}{\gamma_{2}\tilde{\alpha}_{5}\alpha_{6}}$. From \eqref{eq:42}, we deduce that
	\begin{align*}
		\|X(t)\|_{\mathcal{H}}^{2}&\leq \frac{\gamma_{1}}{\gamma_{2}\tilde{\alpha}_{5}\alpha_{6}}\left(3\alpha_{1}\|X_0\|_{\mathcal{H}}^{2}+\underset{s\geq0}{\mathop{\sup}} \left\|\psi_{0x}(\cdot, s) \right\|_{L^2(0, 1)}^{2}\right)(1+t)^{-\frac{1}{(\delta_{1}-1)}}\\
		&\leq \hat{{a}}_{2}\left(\|X_0\|_{\mathcal{H}}^{2}+\underset{s\geq0}{\mathop{\sup}} \left\|\psi_{0x}(\cdot, s) \right\|_{L^2(0, 1)}^{2}\right)(1+t)^{-\frac{1}{(\delta_{1}-1)}},\quad \forall t\in {\mathbb{R}}_{\geq0}.
	\end{align*}
	That is, for all $t\in {\mathbb{R}}_{\geq0}$, we have
	\begin{align}
		\|X(t)\|_{\mathcal{H}}&\leq \sqrt{\hat{{a}}_{2}}\left(\|X_0\|_{\mathcal{H}}+\underset{s\geq0}{\mathop{\sup}} \left\|\psi_{0x}(\cdot, s) \right\|_{L^2(0, 1)}\right)(1+t)^{-\frac{1}{2(\delta_{1}-1)}}.\label{eq:6.24}
	\end{align}
	\paragraph*{\textbf{Case 2: $\delta_{1}=1$}.}
	From \eqref{eq:39}, \eqref{eq:40}, and \eqref{eq:41}, we obtain
	\begin{equation*}
		\begin{split}
			\frac{\mathrm{d}}{\mathrm{d}t}L\left(t\right)
			& \le -\frac{{{\sigma }_{1}}}{{{\gamma }_{1}}{{\alpha }_{1}}}L(t)+{C_3}\int_{0}^{1}{{{f}^{2}}(x, t)\mathrm{d}x}+{C_4}\int_{0}^{1}{{{h}^{2}}(x, t)\mathrm{d}x},\quad \forall t\in {\mathbb{R}}_{>0}.
		\end{split}
	\end{equation*}
	By Gronwall's differential inequality, we obtain
	\begin{align}
		L(t)&\le L(0){\exp\left(-\frac{{{\sigma }_{1}}}{{{\gamma }_{1}}{{\alpha }_{1}}}t\right)}+\int_{0}^{t}{{{e}^{\frac{{{\sigma }_{1}}}{{{\gamma }_{1}}{{\alpha }_{1}}}\left(\tau -t\right)}}\left({C_3}\int_{0}^{1}{{{f}^{2}}\left(x, \tau \right)\mathrm{d}x}+{C_4}\int_{0}^{1}{{{h}^{2}}\left(x, \tau \right)\mathrm{d}x}\right)\mathrm{d}\tau } \nonumber\\ 
		& \le L(0){\exp\left(-\frac{{{\sigma }_{1}}}{{{\gamma }_{1}}{{\alpha }_{1}}}t\right)}+{C_3}\int_{0}^{t}{\left\| f(\cdot, \tau ) \right\|_{{{L}^{2}(0, 1)}}^{2}\mathrm{d}\tau }+{C_4}\int_{0}^{t}{\left\| h(\cdot, \tau ) \right\|_{{{L}^{2}(0, 1)}}^{2}\mathrm{d}\tau },\quad \forall t\in {\mathbb{R}}_{\geq0}.\label{eq:44}
	\end{align}
	From \eqref{eq:40} and \eqref{eq:44}, we deduce that for all $t\in {\mathbb{R}}_{\geq0}$,
	\begin{equation*}
		E(t)\le \frac{{{\gamma }_{1}}}{{{\gamma }_{2}}}E(0){\exp\left(-\frac{{{\sigma }_{1}}}{{{\gamma }_{1}}{{\alpha }_{1}}}t\right)}+\frac{{C_3}}{{{\gamma }_{2}}}\int_{0}^{t}{\left\| f(\cdot, \tau ) \right\|_{{{L}^{2}(0, 1)}}^{2}\mathrm{d}\tau }+\frac{{C_4}}{{{\gamma }_{2}}}\int_{0}^{t}{\left\| h(\cdot, \tau ) \right\|_{{{L}^{2}(0, 1)}}^{2}\mathrm{d}\tau }.\label{eq:45}
	\end{equation*}
	From \eqref{eq:42}, we obtain for all $t\in {\mathbb{R}}_{\geq0}$,
	\begin{align*}
		\|X(t)\|_{\mathcal{H}}^{2}&\leq 	\frac{3\alpha_{1}\gamma_{1}}{\gamma_{2}\alpha_{6}}\|X_0\|_{\mathcal{H}}^{2}\exp\left(-\frac{{{\sigma }_{1}}}{{{\gamma }_{1}}{{\alpha }_{1}}}t\right)+\frac{C_3}{\gamma_{2}\alpha_{6}}\int_{0}^{t}{\left\| f(\cdot, \tau ) \right\|_{{{L}^{2}(0, 1)}}^{2}\mathrm{d}\tau } +\frac{C_4}{\gamma_{2}\alpha_{6}}\int_{0}^{t}{\left\| h(\cdot, \tau ) \right\|_{{{L}^{2}(0, 1)}}^{2}\mathrm{d}\tau }.
	\end{align*}
	Similarly, let $\hat{a}_{3}=\frac{3\alpha_{1}\gamma_{1}}{\gamma_{2}\alpha_{6}}, \hat{C}_{2}=\max\left\{\frac{C_3}{\gamma_{2}\alpha_{6}}, \frac{C_4}{\gamma_{2}\alpha_{6}}\right\}$, and $M=\frac{\sigma_{1}}{\gamma_{1}\alpha_{1}}$, then we have
	\begin{align}
		\|X(t)\|_{\mathcal{H}}&\leq \sqrt{\hat{a}_{3}}\|X_0\|_{\mathcal{H}}\exp\left(-\frac{M}{2}t\right)+\sqrt{\hat{C}_{2}}\left(\int_{0}^{t}{\left\| f(\cdot, \tau ) \right\|_{{{L}^{2}(0, 1)}}^{2}\mathrm{d}\tau }\right)^{\frac{1}{2}}\nonumber\\
		&\quad +\sqrt{\hat{C}_{2}}\left(\int_{0}^{t}{\left\| h(\cdot, \tau ) \right\|_{{{L}^{2}(0, 1)}}^{2}\mathrm{d}\tau }\right)^{\frac{1}{2}},\quad \forall t\in {\mathbb{R}}_{\geq0}.\label{eq:6.29}
	\end{align}
	From Remark \ref{rem:remark2}, it follows that the system \eqref{eq:sys} is EiISS when $\delta_{1}=1$. In addition, when $f\equiv h\equiv0$, from \eqref{eq:6.29} we obtain
	\begin{equation}
		\|X(t)\|_{\mathcal{H}}\leq \sqrt{\hat{a}_{3}}\|X_0\|_{\mathcal{H}}\exp\left(-\frac{M}{2}t\right),\quad \forall t\in {\mathbb{R}}_{\geq0}.\label{eq:6.30}
	\end{equation}
	\par In summary, by choosing the constant $C=\max\left\{\sqrt{\hat{a}_{1}}, \sqrt{\hat{a}_{2}}, \sqrt{\hat{a}_{3}}, \sqrt{\hat{C}_{1}}, \sqrt{\hat{C}_{2}}\right\}$, and combining \eqref{eq:6.23}, \eqref{eq:6.24}, \eqref{eq:6.29}, and \eqref{eq:6.30}, we complete the proof of Theorem \ref{thm:main result}.
\end{theoremproof}
\section{Conclusion}\label{sec:8}
This paper investigated the well-posedness and stability of a class of Timoshenko equations subject to infinite history memory and external disturbances. Within the iISS framework, we established two types of iISS estimates: PiISS under the Tolksdorf condition on the relaxation function, and EiISS when the relaxation function decays exponentially. In particular, this paper introduced a rigorous definition of PiISS that avoids the use the graph norm and hence, remains consistent with the classical notion of iISS. It should be noted, however, that the proposed iISS framework is currently restricted to the case of equal wave speed ratios. For unequal wave speed ratios, the mixed derivative term $\psi_{xt}$ in \eqref{eq:5.25a} poses a significant challenge, as it cannot be handled directly within the present approach. In the absence of external disturbances, existing literature commonly employs second-order energy functionals to handle this term, with the cost of requiring higher regularity assumptions on the initial data and relying on intricate inequality estimates. Unfortunately, such techniques are not readily extendable to the iISS setting in the presence of external disturbances. Consequently, stability analysis for systems with unequal wave speed ratios within the iISS framework becomes substantially more involved and calls for novel analytical tools, which we intend to pursue in future work. Beyond this direction, our future research will also focus on further relaxing the constraints imposed on the relaxation function, as well as extending the current analysis to Timoshenko equations with other boundary conditions or with boundary disturbances.

\typeout{get arXiv to do 4 passes: Label(s) may have changed. Rerun}

\begin{thebibliography}{10}
	\providecommand \doibase [0]{http://dx.doi.org/}%
	
	\bibitem{Al-Mahdi2024}
	A.M. Al-Mahdi. New decay results for Timoshenko system in the light of the second spectrum of frequency with infinite memory and nonlinear damping of variable exponent type. {\it Asymptotic Analysis}. 2024\string;138\string:101--133.
	\newblock DOI:10.3233/ASY-231892.
	
	\bibitem{AlMahdi2021}
	A.M. Al-Mahdi, M.M. Al-Gharabli, A. Guesmia, S.A. Messaoudi. New decay results for a viscoelastic-type Timoshenko system with infinite memory. {\it Zeitschrift f{\"u}r Angewandte Mathematik und Physik}. 2021\string;72\string:1--24.
	\newblock DOI:10.1007/s00033-020-01446-x.
	
	\bibitem{AlOmari2021}
	S. Al-Omari. New decay results for a partially dissipative viscoelastic Timoshenko system with infinite memory. {\it arXiv:2109.05811}. 2021.
	
	\bibitem{AmmarKhodja2003}
	F. Ammar-Khodja, A. Benabdallah, J.E. Mu{\~n}oz Rivera, R. Racke. Energy decay for Timoshenko systems of memory type. {\it Journal of Differential Equations}. 2003\string;194\string:82--115.
	\newblock DOI:10.1016/S0022-0396(03)00185-2.
	
	\bibitem{Bi2025}
	Y. Bi, P. Deng, J. Zheng, G. Zhu. Local integral input-to-state stability for non-autonomous infinite-dimensional systems. {\it Communications in Nonlinear Science and Numerical Simulation}. 2026\string;162\string.
	\newblock DOI:10.1016/j.cnsns.2026.110324.
	
	\bibitem{Boulaaras2026}
	S. Boulaaras, D. Ouchenane. Well-posedness, general stability, and numerical analysis of a nonlinear thermodiffusive Timoshenko system with second sound, infinite hereditary memory, distributed delay, and logarithmic dissipation. {\it ZAMM - Journal of Applied Mathematics and Mechanics / Zeitschrift f{\"u}r Angewandte Mathematik und Mechanik}. 2026\string;106(6)\string.
	\newblock DOI:10.1002/zamm.70488.
	
	\bibitem{Chen2025}
	Q. Chen, J. Zheng, G. Zhu. Input-to-state stability and integral input-to-state stability in various norms for 1-D nonlinear parabolic PDEs. {\it Journal of Control and Decision}. 2025\string;12(5)\string:847--859.
	\newblock DOI:10.1080/23307706.2023.2284155.
	
	\bibitem{Dashkovskiy2013}
	S. Dashkovskiy, A. Mironchenko. Input-to-state stability of infinite-dimensional control systems. {\it Mathematics of Control, Signals, and Systems}. 2013\string;25\string:1--35.
	\newblock DOI:10.1007/s00498-012-0090-2.
	
	\bibitem{Dridi2021}
	H. Dridi, K. Zennir. Well-posedness and energy decay for some thermoelastic systems of Timoshenko type with Kelvin-Voigt damping. {\it SeMA Journal}. 2021\string;78\string:385-400.
	\newblock DOI:10.1007/s40324-021-00239-0.
	
	\bibitem{Evans:2010}
	L.C. Evans. {\it Partial Differential Equations}.
	\newblock Providence, RI: American Mathematical Society, 2010.
	
	\bibitem{Zheng:2021}
	X. Guo, J. Zheng. Lyapunov-type inequalities for $\psi$-Laplacian equations. {\it Italian Journal of Pure and Applied Mathematics}. 2021\string;45\string:977-989.
	
	\bibitem{Guesmia2008}
	A. Guesmia, S.A. Messaoudi. On the control of a viscoelastic damped Timoshenko-type system. {\it Applied Mathematics and Computation}. 2008\string;206\string:589-597.
	\newblock DOI:10.1016/j.amc.2008.05.122.
	
	\bibitem{Guesmia2012}
	A. Guesmia, S.A. Messaoudi, A. Soufyane. Stabilization of a linear Timoshenko system with infinite history and applications to the Timoshenko-heat systems. {\it Electronic Journal of Differential Equations}. 2012\string;193\string:1-45.
	\newblock https://inria.hal.science/hal-01281868v1.
	
	\bibitem{Jacob2019}
	B. Jacob, F. L. Schwenninger, H. Zwart. On continuity of solutions for parabolic control systems and input-to-state stability. {\it Journal of Differential Equations}. 2019\string;266(10)\string:6284--6306.
	\newblock DOI:10.1016/j.jde.2018.11.004.
	
	\bibitem{SilvaRacke2021}
	M.A. Jorge Silva, R. Racke. Effects of history and heat models on the stability of thermoelastic Timoshenko systems. {\it Journal of Differential Equations}. 2021\string;275\string:167-203.
	\newblock DOI:10.1016/j.jde.2020.11.041.
	
	\bibitem{Karafyllis2017}
	I. Karafyllis, M. Krstic. ISS in different norms for 1-D parabolic PDEs with boundary disturbances. {\it SIAM Journal on Control and Optimization}. 2017\string;55(3)\string:1716-1751.
	\newblock DOI:10.1137/16M1073753.
	
	\bibitem{Karafyllis2018}
	I. Karafyllis, M. Krstic. {\it Input-to-State Stability for {PDE}s}.
	\newblock Cham: Springer-Verlag, 2018.
	
	\bibitem{Karafyllis2021}
	I. Karafyllis, M. Krstic. ISS estimates in the spatial sup-norm for nonlinear 1-D parabolic PDEs. {\it ESAIM Control Optimisation and Calculus of Variations}. 2021\string;27\string.
	\newblock DOI:10.1051/cocv/2021053.
	
	\bibitem{2018Karafyllis}
	I. Karafyllis, M. Krstic. Decay estimates for 1-D parabolic PDEs with boundary disturbances. {\it ESAIM Control Optimisation and Calculus of Variations}. 2018\string;24\string:1511-1540.
	\newblock DOI:10.1051/cocv/2018043.
	
	\bibitem{Karafyllis2022}
	I. Karafyllis, M. Krstic. Damping-robustness of various 1-D wave PDEs under boundary feedback control. {\it IFAC-PapersOnLine}. 2022\string;55(30)\string:31--36.
	\newblock DOI:10.1016/j.ifacol.2022.11.024.
	
	\bibitem{Karafyllis2023}
	I. Karafyllis, M. Krstic. ISS-based robustness to various neglected damping mechanisms for the 1-D wave PDE. {\it Mathematics of Control, Signals, and Systems}. 2023\string;35(4)\string:741--779.
	\newblock DOI:10.1007/s00498-023-00353-6.
	
	\bibitem{Lamare2018}
	P.O. Lamare, J. Auriol, F. Di Meglio, U. F. Aarsnes. Robust output regulation of $2 \times 2$ hyperbolic systems: control law and input-to-state stability. In: Proceedings of the 2018 Annual American Control Conference (ACC). 2018\string; Milwaukee, WI, USA\string:1732--1739.
	\newblock DOI:10.1109/ACC.2018.8431421.
	
	\bibitem{Lieberman:1991}
	G.M. Lieberman. The natural generalization of the natural conditions of Ladyzhenskaya and Ural'tseva for elliptic equations. {\it Communications in Partial Differential Equations}. 1991\string;16\string:311-361.
	\newblock DOI:10.1080/03605309108820761.
	
	\bibitem{Liu2025}
	R.J. Liu, Q. Zhang. Stability of the Timoshenko beam equation with one weakly degenerate local {K}elvin-{V}oigt damping. {\it ZAMM - Journal of Applied Mathematics and Mechanics / Zeitschrift f{\"u}r Angewandte Mathematik und Mechanik}. 2025\string;105(3)\string.
	\newblock DOI:10.1002/zamm.202300262.
	
	\bibitem{LiuMao2024}
	Y. Liu, S. Mao. Decay property of the Timoshenko-Fourier system with memory-type dissipation. {\it Journal of Differential Equations}. 2024\string;401\string:469-491.
	\newblock DOI:10.1016/j.jde.2024.05.001.
	
	\bibitem{Logemann2013}
	H. Logemann. Stabilization of well-posed infinite-dimensional systems by dynamic sampled-data feedback. {\it SIAM Journal on Control and Optimization}. 2013\string;51(3)\string:1203-1231.
	\newblock DOI:10.1137/110850396.
	
	\bibitem{Messaoudi2009}
	S.A. Messaoudi, B. Said-Houari. Uniform decay in a Timoshenko-type system with past history. {\it Journal of Mathematical Analysis and Applications}. 2009\string;360\string:459-475.
	\newblock DOI:10.1016/j.jmaa.2009.06.064.
	
	\bibitem{Messaoudi2025}
	S.A. Messaoudi, A. Keddi. On the well posedness and the stability of a thermoelastic Gurtin-Pipkin-Timoshenko system without the second spectrum. {\it Zeitschrift f{\"u}r Analysis und ihre Anwendungen}. 2025\string;44\string:79-96.
	\newblock DOI:10.4171/ZAA/1759.
	
	\bibitem{Mironchenko2020}
	A. Mironchenko, C. Prieur. Input-to-state stability of infinite-dimensional systems: recent results and open questions. {\it SIAM Review}. 2020\string;62(3)\string:529-614.
	\newblock DOI:10.1137/19M1291248.
	
	\bibitem{Mironchenko2023a}
	A. Mironchenko. {\it Input-to-State Stability: Theory and Applications}.
	\newblock Cham: Springer-Verlag, 2023.
	\newblock DOI:10.1007/978-3-031-14674-9.
	
	\bibitem{Mironchenko2023pre}
	A. Mironchenko. Input-to-state stability of distributed parameter systems. {\it arXiv:2302.00535v1}. 2023.
	
	\bibitem{MP2011}
	F. Mazenc, C. Prieur. Strict Lyapunov functions for semilinear parabolic partial differential equations. {\it Mathematical Control and Related Fields}. 2011\string;1(2)\string:231-250.
	\newblock DOI:10.3934/mcrf.2011.1.231.
	
	\bibitem{Rivera:2008}
	J.E. Mu{\~n}oz Rivera, R. Racke. Timoshenko systems with indefinite damping. {\it Journal of Mathematical Analysis and Applications}. 2008\string;341\string:1068-1083.
	\newblock DOI:10.1016/j.jmaa.2007.11.012.
	
	\bibitem{Rivera2008}
	J.E. Mu{\~n}oz Rivera, H.D. Fern{\'a}ndez Sare. Stability of Timoshenko systems with past history. {\it Journal of Mathematical Analysis and Applications}. 2008\string;339\string:482-502.
	\newblock DOI:10.1016/j.jmaa.2007.07.012.
	
	\bibitem{Rivera2002}
	J.E. Mu{\~n}oz Rivera, R. Racke. Mildly dissipative nonlinear Timoshenko systems - global existence and exponential stability. {\it Journal of Mathematical Analysis and Applications}. 2002\string;276\string:248-278.
	\newblock DOI:1016/S0022-247X(02)00436-5.
	
	\bibitem{Pazy1983}
	A. Pazy. {\it Semigroups of {L}inear {O}perators and {A}pplications to {P}artial {D}ifferential {E}quations}.
	\newblock New York: Springer-Verlag, 1983.
	\newblock DOI:10.1007/978-1-4612-5561-1.
	
	\bibitem{PM2012}
	C. Prieur, F. Mazenc. ISS-{L}yapunov functions for time-varying hyperbolic systems of balance laws. {\it Mathematics of Control, Signals, and Systems}. 2012\string;24\string:111-134.
	\newblock DOI:10.1007/s00498-012-0074-2.
	
	\bibitem{Raposo2005}
	C.A. Raposo, J. Ferreira, M.L. Santos, N.N.O. Castro. Exponential stability for the Timoshenko system with two weak dampings. {\it Applied Mathematics Letters}. 2005\string;18\string:535-541.
	\newblock DOI:10.1016/j.aml.2004.03.017.
	
	\bibitem{Sontag1989}
	E.D. Sontag. Smooth stabilization implies coprime factorization. {\it IEEE Transactions on Automatic Control}. 1989\string;34(4)\string:435-443.
	\newblock DOI:10.1109/9.28018.
	
	\bibitem{Sontag1998}
	E.D. Sontag. Comments on integral variants of {ISS}. {\it Systems \& Control Letters}. 1998\string;34\string:93-100.
	\newblock DOI:10.1016/S0167-6911(98)00003-6.
	
	\bibitem{Soufyane2003}
	A. Soufyane, A. Wehbe. Uniform stabilization for the Timoshenko beam by a locally distributed damping. {\it Electronic Journal of Differential Equations}. 2003\string;1\string:1-14.
	
	\bibitem{Timoshenko1921}
	S.P. Timoshenko. On the correction for shear of the differential equation for transverse vibrations of prismatic bars. {\it The Philosophical Magazine}. 1921\string;41\string:744-746.
	\newblock DOI:10.1080/14786442108636264.
	
	\bibitem{Tolksdorf:1983}
	P. Tolksdorf. On the Dirichlet problem for quasilinear equations in domains with conical boundary points. {\it Communications in Partial Differential Equations}. 1983\string;8\string:773-817.
	\newblock DOI:10.1080/03605308308820285.
	
	\bibitem{Wakaiki:2022}
	M. Wakaiki. Semi-uniform input-to-state stability of infinite-dimensional systems. {\it Mathematics of Control, Signals, and Systems}. 2022\string;34\string:789-817.
	\newblock DOI:10.1007/s00498-022-00326-1.
	
	\bibitem{Zeng2025}
	S.D. Zeng, M. Aouadi. Polynomial decay rate in extensible Timoshenko-Boltzmann beam. {\it ZAMM - Journal of Applied Mathematics and Mechanics / Zeitschrift f{\"u}r Angewandte Mathematik und Mechanik}. 2025\string;105(5)\string.
	\newblock DOI:10.1002/zamm.70084.
	
	\bibitem{Zheng:2022}
	J. Zheng, L.S. Tavares. A free boundary problem with subcritical exponents in Orlicz spaces. {\it Annali di Matematica Pura ed Applicata}. 2022\string;201\string:695-731.
	\newblock DOI:10.1007/s10231-021-01134-1.
	
	\bibitem{Zheng2018}
	J. Zheng, G. Zhu. Input-to-state stability with respect to boundary disturbances for a class of semi-linear parabolic equations. {\it Automatica}. 2018\string;97\string:271-277.
	\newblock DOI:10.1016/j.automatica.2018.08.007.
	
	\bibitem{Zheng2019}
	J. Zheng, G. Zhu. A De Giorgi iteration-based approach for the establishment of {ISS} properties for Burgers' equation with boundary and in-domain disturbances. {\it IEEE Transactions on Automatic Control}. 2019\string;64(8)\string:3476--3483.
	\newblock DOI:10.1109/TAC.2018.2880160.
	
	\bibitem{Zheng2020}
	J. Zheng, G. Zhu. ISS-like estimates for nonlinear parabolic PDEs with variable coefficients on higher dimensional domains. {\it Systems \& Control Letters}. 2020\string;146\string.
	\newblock DOI:10.1016/j.sysconle.2020.104808.
	
	\bibitem{Zheng2020b}
	J. Zheng, G. Zhu. Input-to-state stability for a class of one-dimensional nonlinear parabolic PDEs with nonlinear boundary conditions. {\it SIAM Journal on Control and Optimization}. 2020\string;58(6)\string:2567-2587.
	\newblock DOI:10.1137/19M1283720.
	
	\bibitem{Zheng2021}
	J. Zheng, G. Zhu. Approximations of {L}yapunov functionals for {ISS} analysis of a class of higher dimensional nonlinear parabolic {PDE}s. {\it Automatica}. 2021\string;125\string.
	\newblock DOI:10.1016/j.automatica.2020.109414.
	
	\bibitem{Zheng2022b}
	J. Zheng, G. Zhu, S. Dashkovskiy. Relative stability in the sup-norm and input-to-state stability in the spatial sup-norm for parabolic PDEs. {\it IEEE Transactions on Automatic Control}. 2022\string;67(10)\string:5361--5375.
	\newblock DOI:10.1109/TAC.2022.3192325.
	
	\bibitem{Zheng2026}
	J. Zheng, G. Zhu. A unified {L}yapunov method for {ISS} of {PDE}s: a tutorial on constructing generalized {L}yapunov functionals for parabolic and hyperbolic equations. {\it Annual Reviews in Control}. 2026\string;62\string:101065.
	\newblock DOI:10.1016/j.arcontrol.2026.101065.
	
\end{thebibliography}
\end{document}